\documentclass{article}
\usepackage{inputenc}

\usepackage{amsmath,amsbsy,amssymb,amsthm,bbm,bm}

\usepackage{multirow,multicol,tabularx,booktabs}
\usepackage{placeins,color,url}
\usepackage{bbm,bm}
\usepackage[ruled]{algorithm2e}
\usepackage[shortlabels]{enumitem}
\usepackage{wrapfig}
\usepackage{relsize}
\usepackage{subfigure}
\usepackage{caption}
\usepackage{mathtools}
\usepackage[draft]{fixme}
\fxsetup{layout=margin}
\fxusetheme{color}
\usepackage{ragged2e}

\usepackage{mathrsfs}
\usepackage{graphicx}
\usepackage[top=2.5cm,bottom=2.5cm,right=2.5cm,left=2.5cm]{geometry}
\usepackage[hidelinks]{hyperref}

\allowdisplaybreaks

\DeclareMathOperator*{\argmax}{argmax}
\DeclareMathOperator*{\argmin}{argmin}

\newtheorem{conj}{Conjecture}[section]
\newtheorem{conj2}{Conjecture}[]
\newtheorem{thm}[conj2]{\bf Theorem}
\newtheorem{defi}[conj]{\bf Definition}
\newtheorem{cor}[conj2]{\bf Corollary}

\newtheorem{lemma}[conj2]{\bf Lemma}
\newtheorem{rem}{\bf Remark}

\newtheorem{assumpt}{\bf Assumption}

\newtheorem{innercustomgeneric}{\customgenericname}
\providecommand{\customgenericname}{}
\newcommand{\newcustomtheorem}[2]{%
  \newenvironment{#1}[1]
  {%
   \renewcommand\customgenericname{#2}%
   \renewcommand\theinnercustomgeneric{##1}%
   \innercustomgeneric
  }
  {\endinnercustomgeneric}
}

\newcustomtheorem{customthm}{Assumption}

\newcommand{\seven}{\bgroup
  \sbox0{7}\usebox0\llap{\rule[.5\ht0]{.4\wd0}{.05\ht0}\rule{.24\wd0}{0pt}}
\egroup}

\newcommand{\E}{\mathbb{E}}

\newcommand{\Var}{\operatorname{Var}}

\def\sign{{\; \text{sign}}}

\def\to{\rightarrow}

\def\Var{\operatorname{Var}}

\def\Bin{\operatorname{Bin}}

\newcommand{\bxi}{\bm{\xi}}
\newcommand{\btheta}{\bm{\theta}}

\def\TV{\operatorname{TV}}

\def\EE{ {\rm I} \kern-.15em {\rm E} }
\def\PP{ {\rm I} \kern-.15em {\rm P} }
\def\Var{{\mathbb{V}\text{ar}}}

\newcommand{\ba}{\mathbf{a}}

\newcommand{\bb}{\mathbf{b}}

\newcommand{\be}{\mathbf{e}}
\newcommand{\bbf}{\mathbf{f}}
\newcommand{\bg}{\mathbf{g}}
\newcommand{\bh}{\mathbf{h}}

\newcommand{\bk}{\mathbf{k}}
\newcommand{\bell}{\bm{\ell}}
\newcommand{\bbm}{\mathbf{m}}

\newcommand{\bp}{\mathbf{p}}
\newcommand{\bq}{\mathbf{q}}
\newcommand{\bbs}{\mathbf{s}}

\newcommand{\bT}{\mathbf{T}}
\newcommand{\bu}{\mathbf{u}}
\newcommand{\bv}{\mathbf{v}}

\newcommand{\bx}{\mathbf{x}}
\newcommand{\bX}{\mathbf{X}}
\newcommand{\by}{\mathbf{y}}
\newcommand{\bY}{\mathbf{Y}}
\newcommand{\bz}{\mathbf{z}}
\newcommand{\bZ}{\mathbf{Z}}

\newcommand{\mE}{\mathcal{E}}
\newcommand{\mF}{\mathcal{F}}

\newcommand{\mM}{\mathcal{M}}
\newcommand{\mN}{\mathcal{N}}

\newcommand{\stack}{\operatorname{stack}}
\renewcommand{\P}{\mathbb{P}}

\definecolor{dex}{RGB}{0, 102, 51}

\def\1{\mathbbm{1}}

\def\wh{\widehat}

\def\ol{\overline}

\usepackage{import}
\usepackage{xifthen}
\usepackage{pdfpages}
\usepackage{transparent}

\title{Generalization Bounds for Transformer-Based \\Next-Token Prediction in a Language Model}
\author{Insung Kong, Niklas Dexheimer, Johannes Schmidt-Hieber \\[0.15cm]
{\em Department of Applied Mathematics, University of Twente}}

\date{}

\begin{document}

\listoffixmes
\clearpage

\maketitle

\begin{abstract}
A refined statistical understanding of LLM pre-training requires the analysis of the transformer architecture for data distributions that encapsulate key characteristics of text data. To address this, we propose a text data distribution based on an extension of the log-bilinear language model from the natural language processing literature. For this data generating process, we derive generalization bounds for deep transformer architectures, highlighting the dependence on  the network architecture, the vocabulary size, the number of documents and the document length.
\end{abstract}

\section{Introduction}
Recent progress in statistical learning theory shows that the success of different neural network architectures can be explained by their respective inductive biases. Deep (feedforward) neural networks (DNNs) are naturally studied within the classical supervised learning paradigm and a well-studied inductive bias is the ability of DNNs to learn compositional structure in the target function \cite{zbMATH07082286, 10.1214/19-AOS1875, kohler2021rate, zbMATH07864496, ICLR2025_b9523d48, zbMATH08078908}. Similarly, convolutional neural networks (CNNs) originate from image classification problems. Statistical theory can reveal distinctive properties of CNNs if studied for a statistical image classification model \cite{zbMATH07604562, zbMATH07883300, chen2022novel, zbMATH08062197, zbMATH07923796}. 

Transformer architectures are at the heart of the current AI revolution. Because they are far more complex than standard feedforward networks, their generalization guarantees remain less understood. There is, however, an impressive line of very recent work studying transformers within the standard supervised learning framework and relating them to conventional models, such as kernel-based estimators or DNNs \cite{gurevych2022rate, kohler2023ratetransformer, takakura2023approximation, havrilla2024understanding, shen2025transformers, jiao2025approximation, 2026arXiv260220555L, shi2026learning, shi2025approximation}.

In this work, we argue that in order to unravel the specific inductive biases of LLM pre-training, a natural next step is to base the transformer analysis on statistical models for text data. Text data occur as a collection of documents, where each document is a sequence of words or tokens. During pre-training, the task is to learn predicting the next word/token in the sequence. This means to learn a conditional probability mass function on a discrete space. The specific structure of text induces constraints on the conditional probabilities. Inspired by the log-bilinear language model proposed in Mnih and Hinton \cite{Mnih2007}, we introduce a new statistical model for text data that we call the nonparametric two-factor (N2F) language model. The N2F language model incorporates key characteristics of text by assuming  that the conditional probabilities are a product of two (unobserved) factors, where the first factor is universal, and the second factor depends on a latent context function that changes from document to document. Both factors involve unknown functions, which makes the next word/token prediction problem nonparametric. Learning the first function requires aggregating information across all available documents, whereas learning the second factor only relies on the data in an individual document and bears some resemblance with the in-context learning task.

We quantify moreover to which extent the inductive bias of the transformer architecture allows learning next word/token prediction in the proposed N2F language model. To this end, we consider fitting a transformer architecture to data generated from the N2F language model, closely following the transformer architectures applied in practice. In particular, the analyzed architecture incorporates word encoding, positional encoding, the standard multi-head self-attention layers with the softmax function, and the possibility to stack different transformer blocks on top of each other. Studying minimization of the cross-entropy loss, our main result (Theorem \ref{theorem_main}) is a risk bound (generalization guarantee) for the transformer applied to the proposed statistical model for next-word/token prediction. It reveals the dependence of the risk on the various size parameters of the transformer architecture, the hardness of the underlying learning problem, and two quantities related to the uniformity of the word embeddings. By optimizing the size parameters of the transformer architecture, nonparametric rates for estimating the two functions in the N2F model occur (Corollary \ref{cor.main}). As the vocabulary size $W$ is large, we argue that one has to include it into the asymptotics, revealing an intricate interplay between $W$ and the convergence rates associated to the function learning problems. 

The proof of the generalization bound relies on an oracle inequality that decomposes the overall risk in an approximation error and a stochastic term. While the connection between the self-attention layer and the Nadaraya--Watson estimator has been made earlier for the in-context regression problem \cite{shen2026understanding}, a new technical difficulty of our construction is to 
identify a Nadaraya--Watson estimator which can be used to estimate 
 the second factor in the proposed N2F language model and to prove that the transformer can approximate it. The analysis requires to derive convergence rates for the Nadaraya--Watson estimator under dependent data generated from a Markov chain supported on a discrete state space and under a non-standard smoothness constraint, as we can only guarantee H\"older smoothness of the $\ell^2$-normalized function, see Section \ref{sec_approximation_sketch} for more details.

The article is structured as follows. Section \ref{sec_general_next_token} discusses relevant aspects of statistical text modeling and introduces the new statistical model for next-word/token prediction. In Section \ref{sec_def_trans}, the transformer class is formally introduced. The assumptions and the risk bounds are given in Section \ref{sec.risk_bd} together with an outline of the main steps in the proof. The main contribution is the control of the approximation error using transformers and key steps of the construction are explained in Section \ref{sec_approximation_sketch}. All related work is discussed in Section \ref{sec.rel_work}. Detailed proofs are deferred to the appendices.

\subsection{Notation}
We use standard notation for sets, $\mathbb{N} \coloneqq \{1,2,\ldots\}$, $\mathbb{N}_0 \coloneqq \{0,1,\ldots\}$, $[n] \coloneqq \{1,2,\ldots,n\}, \mathbb{B}^{m} \coloneq \{\bx \in \mathbb{R}^{m}, \|\bx\|_2 \leq 1 \}$,
$\mathbb{S}^{m-1} \coloneq \{\bx \in \mathbb{R}^{m}, \|\bx\|_2 = 1 \}$, and $\Delta^{m-1} \coloneq \{\bx \in [0,1]^{m}, \|\bx\|_1 = 1 \}$.
We write $\log(\cdot)$ for the natural logarithm and define $x \vee y \coloneqq \max(x,y)$ as well as $x \land y \coloneqq \min(x,y)$.
Vectors, matrix-valued and vector-valued functions are denoted by bold letters. 
We adopt the notation $\bm{0}_d = (0,\ldots,0)^\top \in \mathbb{R}^d$ and $\bm{1}_d = (1,\ldots,1)^\top \in \mathbb{R}^d$.
Unless otherwise specified, 
$\be_j \coloneq (0,\ldots,1,0,\ldots,0)^\top \in \mathbb{R}^W$ 
denotes the $j$-th standard basis vector, and we define $\mathcal{E}_W \coloneq \{\be_1, \ldots, \be_W\} \subset \mathbb{R}^W$.
For $v_1, \ldots, v_d \in \mathbb{R}$, we write $(v_i)_{i \in [d]} \coloneqq (v_1,\ldots,v_d)^{\top}$ and $\bv_{(i)}$ refers to the $i$-th entry of a vector $\bv$.
For column vectors $\bv_1, \ldots, \bv_m$, the stacked column vector is $\stack(\bv_1, \ldots, \bv_m) 
\coloneqq (\bv_1^{\top}, \ldots, \bv_m^{\top})^{\top}$. Its length is the sum of the lengths of $\bv_1,\ldots, \bv_m.$ 
Column-wise concatenation yields the matrix
$(\bv_i)_{i \in [m]} \coloneqq [\bv_1, \bv_2, \ldots, \bv_m]$ whose columns are $\bv_1, \ldots, \bv_m$ (assuming they are all of the same length).
For a matrix $A$, we denote its $t$-th column by $A_{:,t}$.
For vector or matrix-valued functions $\bbf$, we define $\|\bbf\|_{L^\infty(D)}\coloneq\sup_{\bx\in D} \|\bbf(\bx)\|_\infty$ with  $\|\cdot\|_\infty$ the maximum entry vector/matrix norm.
We write $f(n) \lesssim g(n)$ if $f(n) \leq C g(n)$ for some constant $C>0$, and $f(n) \asymp g(n)$ if $f(n) \lesssim g(n)$ and $f(n) \gtrsim g(n)$.
In addition, we write $f(n) = \tilde{O}(g(n))$ if $f(n) \leq C g(n) \log^k g(n)$ for some constant $C>0$ and $k > 0$.
Similarly, we write $f(n) = \tilde{\Omega}(g(n))$ if $f(n) \geq C g(n) \log^{-k} g(n)$ for some constant $C>0$ and $k > 0$. 
If $f(n)/g(n) \to 0$, we write $f(n) \ll g(n)$.
For real numbers $a<b$, we write
$\sum_{t \in [a,b]} \coloneqq \sum_{t \in ([a,b] \cap \mathbb{Z})}$.
We denote by $\bullet_m$ an $m$-dimensional vector with unspecified entries.

\section{Statistical modeling of next-token prediction}
\label{sec_general_next_token}

\subsection{Background}

Statistical next word prediction models the conditional probability distribution of the next word given a sequence of words. The string \texttt{`Yesterday, the weather'} should very likely be continued with \texttt{'was'}. In this case, the last words contain all the information for predicting the next word. However, to continue \texttt{`Yesterday, the weather was'} might be continued by \texttt{`nice'} or \texttt{`bad'}. Without any further context, the next word has to be sampled from a distribution of all possible continuations. The string \texttt{`We went to a nearby bar. We met some friends that we didn't see in a long time and ordered a \ldots'} should be continued with \texttt{`beer', `cocktail', `glass of wine', \ldots}. If one only includes the last words for the prediction of the next word, one would have forgotten that the scene happens in a bar, and the string could also be continued by \texttt{`three course menu'}. From a statistical modeling perspective,  one can either assume that there is a form of long-range dependence in documents that causes certain words (e.g., \texttt{`bar'}) to become relevant for the next word prediction much later in documents. Alternatively, one can assume that each document is associated with a \textit{latent context} that induces a different conditional probability distribution for the next word given the previous words. In the example above, the context could be `friends in a bar', which would then put more mass on words such as \texttt{`beer', `wine', \ldots } in the next word prediction problem above. While long-memory can be incorporated via a frequentist model, we could also imagine a document to be generated hierarchically by first drawing a latent context and then sampling tokens conditionally on the context. The randomly sampled context gives this modeling approach a Bayesian flavor.

In this paper, we adopt the latent context modeling approach  \cite{edelman2024evolution, varre2025learning}. 
Let $$\mathcal{E}_W \coloneqq \{\be_1, \ldots, \be_W\}$$ denote the set of token candidates, where each token is represented by a $W$-dimensional standard basis vector.
To generate a document, a context function $\bbf : (\mathcal{E}_W)^r \to \mathbb{R}^{W}$ is first drawn from some distribution.
Given $\bbf$, the probability that $w$ is the next token $\bX_{t+1}$ following the tokens $(\bX_{1}, \ldots, \bX_t)$ is
\begin{equation} 
\begin{aligned}
  \mathbb{P}\Big(\bX_{t+1} = \be_w \big| \bbf, \bX_1, \ldots, \bX_t \Big) 
  \propto \exp\Big( \bbf(\bX_{t-r+1}, \ldots, \bX_{t})\Big), \qquad w \in [W]. 
\end{aligned} \label{random_f_prob_model}
\end{equation} 

The goal is to estimate the conditional distribution of $\bX_{T+1} | \bX_1, \ldots, \bX_T$ given a collection of documents.
This requires estimating the transition probabilities of a discrete Markov chain of order $r$.
Classical methods include the $r$-gram estimator
\begin{equation}\label{eq: r-gram}
    \wh \bp\big(w|\bX_1, \ldots, \bX_T \big) \coloneq \frac{\Big|\Big\{t \in \{r, \ldots, T-1\} : (\bX_{t-r+1}, \ldots, \bX_{t}, \bX_{t+1}) = (\bX_{T-r+1}, \ldots, \bX_{T}, \be_w)\Big\}\Big|}{\Big|\Big\{t \in \{r, \ldots, T-1\} : (\bX_{t-r+1}, \ldots, \bX_{t}) = (\bX_{T-r+1}, \ldots, \bX_{T})\Big\}\Big|}
\end{equation}
(with tuning parameter $r=1,2,\ldots$) and Kneser--Ney smoothing \cite{KneserNey}.
Each r-gram $(\be_{w_1}, \ldots, \be_{w_r}) \in (\mathcal{E}_W)^r$ is expected to appear $T / W^{r}$ times on average in a document of length $T$ such that, without imposing any structural constraint, the convergence rate is $\gtrsim W^{r}/T$ (see Theorem 5 of \cite{han2023optimal} for the stronger lower bound $\tilde{\Omega}(W^{r+1}/T)$). The vocabulary size $W$ in large language models is typically in the range $3 \times 10^4$ to $10^5$, and cannot be regarded as negligible. However, these arguments ignore that the transition probabilities are highly structured objects.

Without assumption on the context function $\bbf$, we cannot achieve a better convergence rate than a $r$-gram estimator,
even when learning a model (e.g., a transformer) from a training dataset composed of multiple documents, as each document is associated with a different context function $\bbf$.
For this reason, we introduce assumptions on $\bbf$ through (i) token embeddings and (ii) a decomposition into a linguistic function and a context function.

The embedding of tokens into the $m$-dimensional Euclidean space should map semantically similar words to nearby points in the embedding space, causing
the conditional probabilities in the language model to take similar values.  
Although words form a discrete subset of $\mathbb{R}^m,$ this representation allows us to impose smoothness. A key example is the log-bilinear language model by Mnih and Hinton \cite{Mnih2007}, assuming
\begin{align*}
    \mathbb{P}\big(\bX_{t+1} = \be_w \big| \bX_1, \bX_2, \ldots, \bX_t \big) \propto \exp\Bigg( \bell_0 \Big( \bm{\phi}(\bX_{1}) , \ldots, \bm{\phi}(\bX_{t}) \Big)^{\top} \bm{\phi}(\be_w) \Bigg).
\end{align*}
Here, $\bm{\phi}: \mathcal{E}_W \to \mathbb{R}^m$ is an embedding, and $\bell_0: \mathbb{R}^{m \times t} \to \mathbb{R}^m$ is an affine map shared across all documents.

Although we model latent context via a random score function $\bbf$ in \eqref{random_f_prob_model}, one may expect that score functions across different documents share common basic linguistic structure, since their short-range dependencies are largely invariant across documents. 
For example, \texttt{`Yesterday, the weather was'} is likely to be followed by \texttt{`nice'} or \texttt{`bad'} but not by \texttt{`cat'} or \texttt{`attacked'}, regardless of the overall context of a document. 
Motivated by this intuition, we decompose the score function into a linguistic function and a context function,
which capture short-range dependencies and 
latent context, respectively.

\subsection{Data generating process}

As transformers are commonly applied to tokens, we refer to the states as tokens. Tokens are coded as standard basis vectors and are therefore elements of $\mathcal{E}_W \coloneq \{\be_1, \ldots, \be_W\}.$
For fixed but unknown $m_1 \in \mathbb{N}$ and $m_2 \in \{2,3,\ldots,\}$, 
we assume that there exist unknown token embeddings 
\begin{align}
\bm{\phi}_1 : \mathcal{E}_W \to [-1,1]^{m_1} \quad \text{and} \quad     \bm{\phi}_2 : \mathcal{E}_W \to \mathbb{S}^{m_2-1}, \label{def_phi12}
\end{align}
mapping the tokens $\be_1, \ldots, \be_W$ to the hypercube $[-1,1]^{m_1}$ and 
the unit sphere $\mathbb{S}^{m_2-1}\subset \mathbb{R}^{m_2}.$ We call 
$\bm{\phi}_1$ the \textit{linguistic embedding}, 
and $\bm{\phi}_2$ the \textit{context embedding}.
To facilitate the analysis, we will impose in Definition \ref{def_pos_assumption} regularity conditions ensuring that the context embedding vectors $\bm{\phi}_2(\be_1), \ldots, \bm{\phi}_2(\be_W)$ are (almost) uniformly located on the unit sphere $\mathbb{S}^{m_2-1}$.

Furthermore, we assume that the information of the last $r_1$ tokens can be mapped to a cube through a 
\begin{align}
\text{\textit{universal} linguistic function} \quad \ \bg_0: [-1,1]^{m_1 r_1} \to [-\delta_1, \delta_1]^{m_1}, \quad \text{for some} \ \ \delta_1>0.
 \label{def_uni_g0}
\end{align}   
To account for varying contexts across texts, we further assume that each token is associated with 
a
\begin{align}
    \text{\textit{local} context function} \quad \ \bh : [-1,1]^{m_2 r_2} \to \delta_2 \mathbb{S}^{m_2-1} , \quad \text{for some} \ \ \delta_2>0. 
    \label{def_loc_h}
\end{align}
Throughout the paper, we do not track the dependence on $(m_1, m_2, r_1, r_2, \delta_1, \delta_2)$ and treat them as constants.

To generate a document, a context function $\bh$ is first drawn from a distribution supported on a function class $\mathcal{H} \subset \{ \bh : \mathbb{R}^{m_2 r_2} \to \delta_2 \mathbb{S}^{m_2-1} \}$.
Given $\bh$, we assume that the probability that $w$ is the next token $\bX_{t+1}$ following the tokens $(\bX_{1}, \ldots, \bX_t)$ is
\begin{equation} \label{eq.next_token_prob}
\begin{aligned}
  &\mathbb{P}\big(\bX_{t+1} = \be_w \big| \bh, \bX_1, \ldots, \bX_t \big) \\
  &\propto \exp\bigg( \bg_0 \Big( \bm{\phi}_1(\bX_{t-r_1+1}) , \ldots, \bm{\phi}_1(\bX_{t}) \Big)^{\top} \bm{\phi}_1(\be_w)
  + \bh \Big( \bm{\phi}_2(\bX_{t-r_2+1}) , \ldots, \bm{\phi}_2(\bX_{t}) \Big)^{\top}  \bm{\phi}_2(\be_w) \bigg), 
\end{aligned}
\end{equation} 
setting 
$\bm{\phi}_1(\bX_s) \coloneq\bm{0}_{m_1}$ and $\bm{\phi}_2(\bX_s) \coloneq\bm{0}_{m_2}$ whenever $s\leq 0$. We refer to this as the \textit{nonparametric two-factor (N2F) language model.}   

The function representation is, in general, non-unique, as there could be different pairs $(\bg_0, \bh)$ leading to the same probability distribution. While the model captures characteristics of text data, it is not intended as a model to generate text. For instance, the next-token distribution is typically supported on a sparse subset of all tokens, giving rise to negligible probabilities for most tokens. To incorporate this in the N2F model would require that the right hand side in \eqref{eq.next_token_prob} can change by orders of magnitude, 
violating the boundedness constraints $\|\bg_0\|_\infty\leq\delta_1$ and $\|\bh\|_2\leq\delta_2.$
Instead, the N2F language model should be considered as a starting point for a statistical learning type analysis of text data focusing on the ability of transformers to quickly learn the context from individual documents.

The training data consist of $n$  independent training documents of length $T+1$,
    \begin{align*}        \big(\bX_1^{(i)},\ldots,\bX_{T+1}^{(i)}\big), \quad i=1,\ldots,n.
    \end{align*}
 In particular, each of the $n$ documents is generated by first independently sampling a context function $\bh \in \mathcal{H}$, and then recursively generating tokens from the probabilistic model \eqref{eq.next_token_prob}. In practice, the task is to always use the previous tokens to predict the next token for the whole document. For the theory, we only do this for the last token. Specifically, the considered task is to predict the $(T+1)$-st token distribution given the first $T$ tokens, where the predictor $\wh \bp$ belongs to $$\mathcal{P}_0 \coloneq \{\bp: (\mathcal{E}_W)^{T} \to \Delta^{W-1}\}, \ \  \text{with} \ \ (\mathcal{E}_W)^{T}\coloneq\underbrace{\mathcal{E}_W\times \ldots \times \mathcal{E}_W}_{T \ \text{times}}.$$
For any (new) document $(\bX_1,\ldots,\bX_{T+1})$ generated from the same distribution as the training samples, 
\[
    \wh \bp(\bX_{1}, \ldots, \bX_{T})
\]
is an approximation of the true conditional distribution
\begin{align}
    \bp_0(\bh, \bX_{1}, \ldots, \bX_{T})\coloneq\bigg(\mathbb{P}\big(\bX_{T+1} = \be_w \big| \bh, \bX_1, \bX_2, \ldots, \bX_T \big)\bigg)_{w\in [W]}.
    \label{eq.fh_def}
\end{align}
Although $\wh \bp(\bX_{1}, \ldots, \bX_{T})$ does not get any direct information of the latent function $\bh$ as input, the estimator can learn an approximation $\wh \bh$ of $\bh$ from the token sequence $(\bX_1,\ldots,\bX_T)$.

Given $N$ independently generated test documents \[\big(\bX_1^{(i)},\ldots,\bX_{T+1}^{(i)}\big), \quad i=n+1,\ldots,n+N,\]
drawn from the same distribution as the training data, the statistical task it to evaluate how well, in average, the estimator $\wh{\bp}$ allows us to predict the $(T+1)$-st token from the first $T$ tokens.
The corresponding test error (also referred to as \textit{log-perplexity} in the NLP literature) is given by 
\[
    - \frac 1N \sum_{i=n+1}^{n+N}  \bX^{(i) \top}_{T+1} \log\big( \wh{\bp}(\bX^{(i)}_{1}, \ldots, \bX^{(i)}_{T})\big).
\]
With $\bp_0(\bh,\bX_{1}, \ldots, \bX_{T})$ as defined in \eqref{eq.fh_def} and using that $\bX_{T+1}$ is one-hot encoded, the associated population risk is
\begin{equation} \label{eq_CE_KL}
    \begin{aligned} 
        - \E_{\mathrm{test}}
        \Big[ \bX_{T+1}^\top \log\big( \wh \bp(\bX_{1}, \ldots, \bX_{T})\big)
    \Big]
    &= \E_{\mathrm{test}} \bigg[\log\frac{ \bX_{T+1}^\top  \bp_0(\bh, \bX_{1}, \ldots, \bX_{T})}{ \bX_{T+1}^\top  \wh \bp(\bX_{1}, \ldots, \bX_{T})}
    \bigg]\\
    & \quad - \E_{\mathrm{test}} \Big[\bX_{T+1}^\top  \log\big(\bp_0(\bh, \bX_{1}, \ldots, \bX_{T})\big)
    \Big]
\end{aligned}
\end{equation}
for an independently sampled 
$(\bh, \bX_1,\ldots,\bX_{T+1})$
that has the same distribution as the training documents. 
The first term on the right hand side is the expected Kullback--Leibler divergence between the conditional distributions of $\bX_{T+1}$ given $(\bh, \bX_1,\ldots,\bX_{T})$, and the second term does not depend on the estimator. 
Hence, the first term serves as the excess (population) risk. 

\section{The transformer function class} \label{sec_def_trans}

We provide a mathematical definition of the individual components of the transformer architecture, which then leads to the mathematical definition of the transformer function class at the end of the section.  

\begin{defi}[Softmax] \label{def_softmax}
For a vector $\bx = (x_1,\ldots,x_d)^{\top}$, the softmax function is defined as
$$\operatorname{softmax}(\bx) \coloneq \Big(\frac{\exp(x_1)}{\sum_{i=1}^d \exp(x_{i})}, \ldots, \frac{\exp(x_{d})}{\sum_{i=1}^d \exp(x_{i})}\Big)^{\top}.$$
If the input is a $d\times T$ matrix, the softmax function is applied row-wise,
    $$\operatorname{softmax}\big((x_{i,t})_{i,t}\big) \coloneq
    \bigg(\frac{\exp(x_{i,t})}{\sum_{t=1}^T \exp(x_{i,t})}
    \bigg)_{i,t}.$$
\end{defi}
\begin{defi}[Self-attention] \label{def_sa}
    Given a dimension $N \in \mathbb{N}$ and learnable matrices $U, V \in \mathbb{R}^{N \times N}$,
    the output of a self-attention layer with input  
    $X \in \mathbb{R}^{N \times T}$ is given by
    \begin{align*}
        \operatorname{SA}_{U,V}(X) \coloneq VX \operatorname{softmax}\big(X^{\top} U X\big) \in \mathbb{R}^{N \times T}.
    \end{align*}
    For $B>0$, we define the class of self-attention layers with bounded parameters
    $$\mathcal{SA}(N,B) \coloneq \Big\{\operatorname{SA}_{U,V} : U, V \in [-B,B]^{N \times N}\Big\}.$$ 
\end{defi}

For a matrix input $X \coloneq (\bx_1, \ldots, \bx_T) \in \mathbb{R}^{N \times T}$,
the $t$-th column of $\operatorname{SA}_{U,V}(X)$ can be written as
\begin{align}
    \operatorname{SA}_{U,V}(X)_{:,t} = \sum_{s=1}^T \frac{\exp(\bx^\top_s U \bx_{t})}{\sum_{r=1}^T \exp(\bx^{\top}_r U \bx_t)} V \bx_s. \label{eq_sa_t}
\end{align}
The original definition of the self-attention layer further factorizes the matrix $U$ as $K^{\top} Q$ with $K$ the so-called key matrix and $Q$ the so-called query matrix. For notational simplicity, we directly work here with the matrix product.

\begin{defi}[Multi-head self-attention]
    Given a dimension $N\in \mathbb{N}$, number of heads $H \in \mathbb{N}$, $B>0$, $W \in [-B,B]^{N \times N}$ and $f_1, \ldots, f_H \in \mathcal{SA}(N,B)$, the output of a multi-head self-attention layer with input  
    $X \in \mathbb{R}^{N \times T}$ is given by     $$\operatorname{MA}_{f_1,\ldots,f_H}(X) \coloneq X + \sum_{h=1}^H f_h(X) \in \mathbb{R}^{N \times T}.$$
    The class of multi-head self-attention layers with bounded parameters is
    $$\mathcal{MA}(N,H,B) \coloneq \Big\{\operatorname{MA}_{f_1,\ldots,f_H} : f_1,\ldots,f_H \in \mathcal{SA}(N,B) \Big\}.$$ 
\end{defi}

The function class $\mathcal{MA}(N,H,B)$ has $H(N^2 + N^2)=2H N^2$ many parameters.

\begin{defi}[Multilayer perceptron]    
    Given input and output dimensions $d_{\mathrm{in}}, d_{\mathrm{out}} \in \mathbb{N}$, depth $L_{\mathrm{mlp}} \in \mathbb{N}$ and width $N_{\mathrm{mlp}} \in \mathbb{N}$, 
    let $\operatorname{MLP}_{\bm{\theta}}: \mathbb{R}^{d_{\mathrm{in}}} \to \mathbb{R}^{d_{\mathrm{out}}}$ be a multilayer perceptron parameterized by $\bm{\theta}$, with $L_{\mathrm{mlp}}$ hidden layers, $N_{\mathrm{mlp}}$ neurons in each hidden layer and the ReLU activation function $x \mapsto x \vee 0$.
    For $B>0$, we define 
    $$\mathcal{MLP}\big(d_{\mathrm{in}}, d_{\mathrm{out}}, L_{\mathrm{mlp}},N_{\mathrm{mlp}},B\big) \coloneq \Big\{ \operatorname{MLP}_{\bm{\theta}}: \mathbb{R}^{d_{\mathrm{in}}} \to \mathbb{R}^{d_{\mathrm{out}}} \Big|
\operatorname{MLP}_{\boldsymbol{\theta}} \text{has depth } L_{\mathrm{mlp}}, \text{width } N_{\mathrm{mlp}}, \text{and } \|\boldsymbol{\theta}\|_\infty \leq B \Big\}.$$  
    For sequential input $X \coloneq (\bx_1, \ldots, \bx_T) \in \mathbb{R}^{N \times T}$, the output of the multilayer perceptron is given by $$\operatorname{MLP}_{\bm{\theta}}(X) \coloneq \Big(\operatorname{MLP}_{\bm{\theta}}(\bx_1),\ldots,\operatorname{MLP}_{\bm{\theta}}(\bx_T)\Big).$$
\end{defi}
The formal definition of multilayer perceptrons is deferred to Appendix~\ref{sec_construct_firstmlp}.
For simplicity, we omit skip connections in the MLP layers. However, our results can be readily extended to models with skip connections (e.g., see Lemma B.3 of \cite{ma2025provable}).

\begin{defi}[Transformer block]
    For $N, H, L_{\mathrm{mlp}}, N_{\mathrm{mlp}} \in \mathbb{N}$ and $B>0$, we define the class of transformer blocks
    $$\mathcal{TB}(N,H,L_{\mathrm{mlp}},N_{\mathrm{mlp}}, B) \coloneq \Big\{ \bh \circ \bg: \bg  \in \mathcal{MA}(N,H,B), \, \bh \in \mathcal{MLP}(N,N,L_{\mathrm{mlp}},N_{\mathrm{mlp}},B) \Big\}.$$
\end{defi}

The multi-head self-attention has $2H N^2$ learnable parameters and the multilayer perceptron has $(N+1)N_{\mathrm{mlp}}+(L_{\mathrm{mlp}}-1)(N_{\mathrm{mlp}}+1)N_{\mathrm{mlp}} + (N_{\mathrm{mlp}}+1)N$ many learnable parameters.
Hence, the number of parameters of the transformer block is 
\begin{align}
    2H N^2+ 2NN_{\mathrm{mlp}}+(L_{\mathrm{mlp}}-1)N_{\mathrm{mlp}}^2+L_{\mathrm{mlp}}N_{\mathrm{mlp}} + N.
    \label{eq.nr_params_TB}
\end{align}

For an input matrix 
$$X \coloneq \big[ \bx_1, \bx_2, \ldots, \bx_T \big] \in \mathbb{R}^{W\times T},$$ 
we consider a transformer model $\bT$ with $L$ transformer blocks, model dimension $N$, attention layers with $H$ heads, multilayer perceptrons with depth $L_{\mathrm{mlp}}$ and width $N_{\mathrm{mlp}}$, and parameter range $B$.
The output of this model can be represented as 
\begin{align} \label{eq_def_transformer}
    \bbf(X) =\Big( D \cdot \operatorname{TB}_L \circ \cdots \circ \operatorname{TB}_1  (P + E X)\Big)_{:,T} \in \mathbb{R}^{W}.
\end{align}
For each $\ell \in [L]$,
$\operatorname{TB}_\ell \in \mathcal{TB}(N,H,L_{\mathrm{mlp}},N_{\mathrm{mlp}}, B)$ is a transformer block, composed with a multi-head self-attention with $H$ heads and a multilayer perceptron with depth $L_{\mathrm{mlp}}$ and width $N_{\mathrm{mlp}}$. 
Here, $E \in [-B,B]^{N \times W}$ and 
$D \in [-B,B]^{W \times N}$ are learnable encoding and decoding matrices, respectively. The positional encoding matrix
$P \in \mathbb{R}^{N \times T}$ is not learnable.
Following the original transformer paper \cite{vaswani2017attention}, 
we use the following multi-scale sinusoidal positional encoding.

\begin{defi}[Positional encoding] \label{def_PE}
    For $A \leq N/2$, we define the positional encoding matrix $P \in \mathbb{R}^{N \times T}$ as 
    $$P \coloneq \left[\begin{matrix}
        \bm{0}_{N-2A}& \ldots & \bm{0}_{N-2A} \\
        \bq_{1}& \ldots& \bq_{T} 
    \end{matrix}\right], \quad \bq_{t} \coloneq \begin{pmatrix}
        \bq_{1,t}\\
        \bq_{2,t}\\
        \vdots\\
        \bq_{A,t}
    \end{pmatrix},  \quad 
    \bq_{a,t} \coloneq 
    \left(
    \begin{matrix}
        \sin\left(\frac{t}{T^{a/A}}\right) \\ \cos\left(\frac{t}{T^{a/A}}\right)
    \end{matrix}
    \right).
   $$
   Throughout the paper, we assume that $N$ is a multiple of $4$ and set $A \coloneq N/4$.
\end{defi}

\begin{figure}[t]
  \centering  \includegraphics[width=\textwidth]{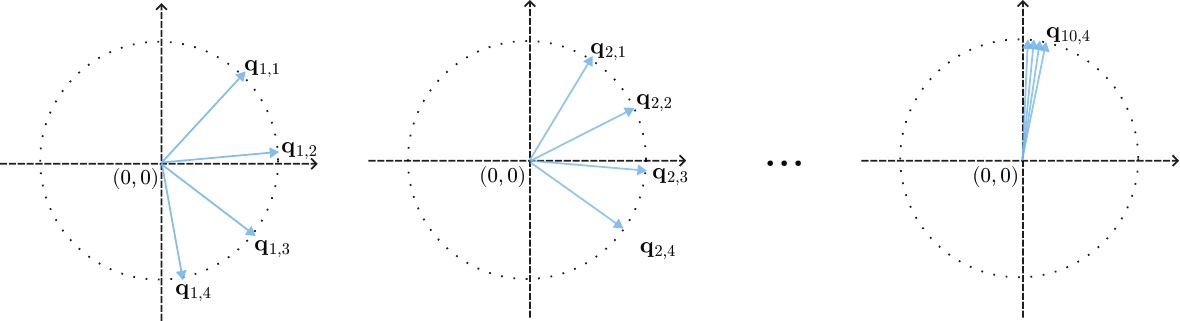}
  \caption{Illustration of the multi-scale sinusoidal positional encoding defined in Definition~\ref{def_PE} for $A=10$ and $T=20$. 
  For each $a \in [A]$,
  the vectors $\bq_{a,t} = (\sin(t / T^{a/A}), \cos(t / T^{a/A}))^{\top}$, $t = 1,\ldots, T$, lie on the unit circle, with smaller $a$ corresponding to faster rotation.
    } \label{fig_pos_encode}
\end{figure}

Figure~\ref{fig_pos_encode} shows an illustration of the positional encoding defined above.
The positional encoding is necessary in the transformer architecture because all other operators (multi-head self-attention, multilayer perceptron, encoder and decoder) are permutation invariant\footnote{$f : \mathbb{R}^{d \times T} \to \mathbb{R}^{d' \times T}$ is permutation invariant if $f(\Pi X) = \Pi f(X)$ for any column-wise permutation operator $\Pi$ and $X \in \mathbb{R}^{d \times T}$.}. 
Most existing works on the statistical analysis of transformers employ the single-scale sinusoidal positional encoding (i.e., $A=1$). For our analysis, it is crucial to use instead the multi-scale sinusoidal positional encoding with a sufficiently large $A$. In particular, this reduces the magnitude of the parameters from polynomial order in $T$ to logarithmic order when approximating the hardmax function by the softmax function in the self-attention layer, see Remark~\ref{remark_benefit_multi_sinusodial}.

In the definition of $\bq_{q,t}$, $T$ can be replaced by any $T_{\max}$ with $T_{\max} \geq T$.
Although this formulation is more natural, we take $T_{\max}= T$ for  notational simplicity. We now give the final definition of the transformer function class used in this paper.

\begin{defi}[Transformer function class]
\label{defi.transfomer_class}
We define 
\begin{align*}
\mathcal{T} \big(W,T,L,N,H,L_{\mathrm{mlp}},N_{\mathrm{mlp}},B\big) \coloneq \Big\{ \bbf :  (\mathbb{R}^{W})^T \to \mathbb{R}^W \text{ has the form \eqref{eq_def_transformer} with } E, D^\top \in [-B,B]^{N \times W}& \\
\text{ and } \operatorname{TB}_1, \ldots, \operatorname{TB}_L \in \mathcal{TB}(N,H,L_{\mathrm{mlp}},N_{\mathrm{mlp}},B)
 \, & \Big\}.
\end{align*}
In addition, we define the class of clipped transformer
networks for $M>0$ as
\[
\mathcal{T}_M \big(W,T,L,N,H,L_{\mathrm{mlp}},N_{\mathrm{mlp}},B\big) \coloneq \Big\{ \operatorname{clip}_{M} \circ \, \bbf : \bbf \in   \mathcal{T} \big(W,T,L,N,H,L_{\mathrm{mlp}},N_{\mathrm{mlp}},B\big) \Big\},
\]
where $\operatorname{clip}_{M} : x \mapsto (-M \vee x) \land M$
is applied element-wise.
\end{defi}
By \eqref{eq.nr_params_TB}, the number of learnable parameters in the (clipped) transformer class is 
\begin{align}
    2WN+ L\Big(2H N^2+ 2NN_{\mathrm{mlp}}+(L_{\mathrm{mlp}}-1)N_{\mathrm{mlp}}^2+L_{\mathrm{mlp}}N_{\mathrm{mlp}} + N\Big).
    \label{eq.nr_params_TC}
\end{align}

\section{Risk bound}
\label{sec.risk_bd}

Regularity conditions in fixed design settings are commonly imposed in nonparametric statistics (e.g., see \cite{zbMATH05358920}). As mentioned in Section~\ref{sec_general_next_token}, we require the context embedding vectors $\bm{\phi}_2(\be_1), \ldots, \bm{\phi}_2(\be_W)$ to be regularly spaced on the unit sphere.
The regularity is measured by the function $\varphi_W$ and the quantity $\kappa_W.$ 

\begin{defi} \label{def_pos_assumption}
Let $\varphi_W: (0,\infty) \to (0,1]$ and $\kappa_W > 0$  be given by 
\[
\varphi_W(\epsilon) \coloneqq \inf_{w \in [W]} \frac{1}{W} \Big| \big\{ {w'} \in [W]: \| \bm{\phi}_2(\be_w) - \bm{\phi}_2(\be_{w'})\|_2 \leq \epsilon \big\}\Big|
\]
and
$$\kappa_W \coloneq\sup_{\bv \in \delta_2 \mathbb{S}^{m_2-1}}\left\|\frac{\sum_{w=1}^W \exp\big(\bv^{\top}\bm{\phi}_2(w)\big) \bm{\phi}_2(w)}{\sum_{w=1}^W \exp\big(\bv^{\top}\bm{\phi}_2(w)\big)}  - 
\frac{\int_{\bz \in \mathbb{S}^{m_2 - 1}} \exp(\bv^{\top} \bz) \bz d\bz }{\int_{\bz \in \mathbb{S}^{m_2 - 1}} \exp(\bv^{\top} \bz) d\bz  }
\right\|_{2}.$$
\end{defi}

By definition, $\varphi_W(\epsilon) \geq 1/W$ holds for any $\epsilon \in (0,\infty)$. 
In addition, if $\bm{\phi}_2(\be_1), \ldots, \bm{\phi}_2(\be_W)$ are  independent and identically distributed samples from the uniform distribution on $\mathbb{S}^{m_2-1}$ with $m_2 \geq 3$, then a covering number argument yields the following assumption with high probability; see Appendix \ref{app_discuss_emb} for detail.
\begin{assumpt} \label{assumption_word_embedding}
    For any $\epsilon>0$ with
$W^{-\frac{1}{m_2-1}} \log W \leq \epsilon \ll 1$,
\begin{align*}
    \varphi_W(\epsilon) = \tilde{\Omega}(\epsilon^{m_2 - 1}) \quad \text{ and } \quad
    \kappa_W = \tilde{O}(W^{-1/2}),
\end{align*}
where $\tilde{\Omega}, \tilde{O}$ denote the lower/upper bounds for the rate up to log-factors (see section on notation for precise definitions).
\end{assumpt}

The token probabilities following the sequences \texttt{`I am very glad to'} and \texttt{`I feel really happy to'} should be similar. 
To this end, we impose the H\"older smoothness assumption on the linguistic function $\bg_0$ and the context function $\bh$. 
This guarantees that if $\bm{\phi}_1(\bx_s) \approx \bm{\phi}_1(\bx'_s)$ for $t-r_1+1 \leq s \leq t$, then 
\[
\bg_0\big(\bm{\phi}_1(\bx_{t-r_1+1}), \ldots, \bm{\phi}_1(\bx_{t})\big) 
\approx 
\bg_0\big(\bm{\phi}_1(\bx'_{t-r_1+1}), \ldots, \bm{\phi}_1(\bx'_{t})\big),
\]
and the same holds for $\bm{\phi}_2$ and $\bh$.

\begin{defi}[$\beta$-H\"older smooth functions]
Let $\beta = q+s$ for some $q \in \mathbb{N}_0$ and $s \in (0,1]$.
The $\beta$-H\"older norm of a function $\bbf = (f_1, \ldots, f_{d})^{\top} : [-1,1]^d \to \mathbb{R}^{d'}$ is defined as
\begin{align*}
\left\| \bbf \right\|_{\mathcal{C}^{\beta}} \coloneqq
  \max_{j \in [d']} \bigg( \sum_{\substack{\bm{\alpha} \in \mathbb{N}_0^d \\ \bm{\alpha}:\|\bm{\alpha}\|_1 \leq q}} \left\|\partial^{\bm{\alpha}} f_j \right\|_{\infty} 
+\sum_{\substack{\bm{\alpha} \in \mathbb{N}_0^d \\ \bm{\alpha}:\|\bm{\alpha}\|_1 =q}} \sup _{\underset{\bx_1 \ne \bx_2}{\bx_1, \bx_2 \in [-1,1]^d}} 
\frac{\left|\partial^{\bm{\alpha}} f_j(\bx_1)-\partial^{\bm{\alpha}} f_j(\bx_2)\right|}{\|\bx_1-\bx_2\|_{\infty}^{s}} \bigg).
\end{align*} 
We say that a function $\bbf$ is $\beta$-H\"older smooth if $\| \bbf \|_{\mathcal{C}^{\beta}}$ exists and is finite.
\end{defi}

\begin{assumpt} \label{assumption_gh_smoothness}
    (i) $\|\bg_0\|_{\mathcal{C}^{\beta_1}} < \infty$ for some $\beta_1 > 0$, and
    (ii)  $\sup_{\bh \in \mathcal{H}} \|\bh\|_{\mathcal{C}^{\beta_2}} < \infty$ for some $\beta_2 \in (0,1]$. 
\end{assumpt}

Any function on a finite discrete domain can be extended to a Hölder-smooth function on a continuous domain with finite Hölder norm.
Therefore, the Hölder smoothness assumption is only meaningful on a discrete domain if one considers joint asymptotics where also the vocabulary size $W$ grows.

For a class of candidate functions $\mathcal{P} \subset \mathcal{P}_0$, we consider the empirical risk minimizer $\wh{\bp}$ with respect to the cross-entropy loss, that is,
\begin{align} \label{eq_estimator_definition}
    \wh{\bp} \in  \argmin_{\bp \in \mathcal{P}} \,  - \frac{1}{n} \sum_{i=1}^n \bX^{(i) \top}_{T+1} \log\big( \bp(\bX^{(i)}_{1}, \ldots, \bX^{(i)}_{T})\big),
\end{align}
 where the logarithm is applied componentwise. As we control the stochastic fluctuation using the multiclass classification oracle inequality in \cite{zbMATH07524984}, one could moreover incorporate the optimization error in the analysis. For notational simplicity, we only consider the empirical risk minimizer. 

In particular, we consider 
\begin{align}
    \mathcal{P} = \Big\{\operatorname{softmax} \circ \, \bbf : \bbf \in \mathcal{T}_M \big(W,T,L,N,H,L_{\mathrm{mlp}},N_{\mathrm{mlp}},B\big)\Big\} \label{candidate_functions_transformer}
\end{align}
as the class of candidate functions in \eqref{eq_estimator_definition}.
We further assume that the transformer class is sufficiently large as follows.
\begin{assumpt} \label{assumption_transformer_size_lower}
$(L, N, H, L_{\mathrm{mlp}},N_{\mathrm{mlp}},B, M)$ satisfies 
\begin{align*}
    &L \geq 2, \quad N \geq 8 \log T, \quad H \geq r_1 + r_2, \quad
    N_{\mathrm{mlp}} \geq N \vee T^{\alpha}, \quad L_{\mathrm{mlp}} \geq \log^3(N_{\mathrm{mlp}} \vee T),
    \\
    &B \geq 9\log^2 T, \quad M \geq \delta_1 m_1 + \delta_2
\end{align*}
for some $\alpha>0$.
\end{assumpt}

We now present the main result of this paper. It provides convergence rates for the expected excess risk \eqref{eq_CE_KL} of transformer-based next-token prediction models.

\begin{thm} \label{theorem_main}
    Suppose that Assumptions~\ref{assumption_gh_smoothness} and    \ref{assumption_transformer_size_lower} hold and that
    $\kappa_W \leq (1+m_2/\delta_2)^{-1}/2$.
    For sufficiently large $n, W,$ and $T$,
    the empirical risk minimizer \eqref{eq_estimator_definition} with the candidate class \eqref{candidate_functions_transformer} 
    satisfies
    \begin{align*}
        \E \bigg[\log\frac{ \bX_{T+1}^\top  \bp_0(\bh, \bX_{1}, \ldots, \bX_{T})}{ \bX_{T+1}^\top  \wh \bp(\bX_{1}, \ldots, \bX_{T})}
    \bigg]
    &= \tilde{O}\bigg( \frac{M (WN + HLN^2 + L L_{\mathrm{mlp}} N_{\mathrm{mlp}}^2 + L^2 L_{\mathrm{mlp}}  + \log B)}{n} \\
    & \qquad \qquad \, \, +
    N_{\mathrm{mlp}}^{-\frac{4 \beta_1}{m_1 r_1}} + \inf_{\lambda \in [T^{-1},\log^{-1} (W \land  T)]}\Big( \lambda^{2 \beta_2} + \frac{M^2}{T \varphi_W(\lambda)^{r_2}} \Big)  + \kappa_W^2\bigg),
    \end{align*} 
    where the expectation is taken over all randomness (training data and $(\bh, \bX_1, \ldots, \bX_T)$).
    The implicit    
    constant depends only on $(m_1, m_2, r_1, r_2, \delta_1, \delta_2, \beta_1, \beta_2, \| \bg_0 \|_{\mathcal{C}^{\beta_1}}, \sup_{\bh \in \mathcal{H}} \| \bh \|_{\mathcal{C}^{\beta_2}})$.
\end{thm}

The first line on the right-hand side represents the stochastic error. 
This term increases with the model size but can be controlled by taking a sufficiently large sample size $n$, i.e., by using more training documents.
The second line corresponds to the approximation error.
In particular, the first term $N_{\mathrm{mlp}}^{-4 \beta_1/(m_1 r_1)}$ corresponds to approximating the universal linguistic function $\bg_0$, which can be controlled by increasing the width of the transformer.

The remaining terms of the error bound originate from the learning of the context function $\bh$.
Since $\bh$ is random, it cannot be directly used in constructing the transformer for approximation and must instead be estimated from the data.  
In the proof, we construct a Nadaraya--Watson estimator with bandwidth $\lambda$ whose direction converges to that of $\bh$. 
The term involving the infimum arises from the bias-variance trade-off of the estimator and decreases as $T$ increases. 
Finally, the $\kappa_W$ term appears due to the discreteness of the token embeddings and is typically negligible under mild assumptions (e.g., Assumption~\ref{assumption_word_embedding}). We want to stress again, that as we apply the risk bound in Theorem 3.5 of \cite{zbMATH07524984}, one can go beyond the analysis of the empirical risk minimizer at the cost of an additional optimization error term in the rate.

The following corollary follows from Theorem~\ref{theorem_main} by additionally imposing Assumption~\ref{assumption_word_embedding} and selecting an appropriate range of transformer architectures.

\begin{cor}
\label{cor.main}
    Suppose Assumptions~\ref{assumption_word_embedding}, \ref{assumption_gh_smoothness} and \ref{assumption_transformer_size_lower}, and further assume that there exists $k>0$ such that
\begin{align}
    \max(L, N, H, L_{\mathrm{mlp}}, B, M) \lesssim \log^k(n \vee W \vee T) \quad \text{and} \quad
N_{\mathrm{mlp}} \asymp n^{\frac{m_1 r_1}{2(2 \beta_1 + m_1 r_1)}}. \label{network_size_upper}
\end{align}
For sufficiently large $n, W,$ and $T$ with $T \geq \log W$,
    the empirical risk minimizer \eqref{eq_estimator_definition} with the candidate class \eqref{candidate_functions_transformer} 
    satisfies
\begin{align*}
        \E \bigg[\log\frac{ \bX_{T+1}^\top  \bp_0(\bh, \bX_{1}, \ldots, \bX_{T})}{ \bX_{T+1}^\top  \wh \bp(\bX_{1}, \ldots, \bX_{T})}
    \bigg]
    & = \tilde{O} \Big( \frac{W}{n} + n^{-\frac{2 \beta_1}{2 \beta_1 + m_1 r_1}} +  
     T^{-\frac{2 \beta_2}{2 \beta_2 + (m_2-1)r_2 }}  + W^{-\frac{2 \beta_2}{m_2 - 1}} \Big),
    \end{align*} 
where the expectation is taken over all randomness (training data and $(\bh, \bX_1, \ldots, \bX_T)$).
    The implicit    
    constant depends only on $(m_1, m_2, r_1, r_2, \delta_1, \delta_2, \beta_1, \beta_2, \| \bg_0 \|_{\mathcal{C}^{\beta_1}}, \sup_{\bh \in \mathcal{H}} \| \bh \|_{\mathcal{C}^{\beta_2}})$.
\end{cor}

\begin{proof}
The result follows from Theorem~\ref{theorem_main} for the choices in \eqref{network_size_upper}, using Assumption \ref{assumption_word_embedding},
    and by choosing 
    \[\lambda = T^{-\frac{1}{2 \beta_2 + r_2(m_2 - 1)}} \vee W^{-\frac{1}{m_2 - 1}} \log W.\]
\end{proof}

The error term $W/n$ corresponds to learning the word embeddings $\bm{\phi}_1 : \mathcal{E}_W \to \mathbb{R}^{m_1}$ and $\bm{\phi}_2 : \mathcal{E}_W \to \mathbb{R}^{m_2}$. 
It scales linearly in the vocabulary size $W$ and can be effectively controlled by using a sufficiently large training set.
The error term $n^{-\frac{2 \beta_1}{2 \beta_1 + m_1 r_1}}$ relates to learning the $\beta_1$-H\"older smooth and $(m_1r_1)$-variate universal linguistic function $\bg_0.$ This is the standard nonparametric convergence rate for sample size $n.$ The error term $T^{-\frac{2 \beta_2}{2 \beta_2 + (m_2 - 1) r_2}}$ corresponds to learning the $\beta_2$-H\"older smooth local context function $\bh$ on the domain $(\mathbb{S}^{m_2 - 1})^{r_2}.$ The convergence rate in the document length $T$ is surprising as the training loss $- \tfrac{1}{n} \sum_{i=1}^n \bX^{(i) \top}_{T+1} \log\big( \bp(\bX^{(i)}_{1}, \ldots, \bX^{(i)}_{T})\big)$ considers $n$ prediction problems. It remains an open problem, whether the function $\bg_0$ can be estimated with a faster rate also involving $T.$ As $n\gg T$ in applications, we view this problem, however, as less relevant. The term $W^{-\frac{2 \beta_2}{m_2 - 1}}$ arises in the convergence rate from the discrete structure of the domain, and becomes negligible when $W \gtrsim T^{1/r_2}$.
We emphasize again that for small (and fixed) vocabulary size $W$, classical methods such as the
$r$-gram estimator or Kneser–Ney smoothing remain optimal, without smoothness assumptions.

To circumvent the curse of dimensionality in learning $\bg_0$, one can additionally assume that $\bg_0$ is compositionally sparse or is supported on a low-dimensional manifold $\mM\subset [-1,1]^{m_1r_1}.$ In the proof, one has to replace then the approximation theoretic result for Hölder functions in Lemma \ref{lemma_approx_smooth} by the respective results for compositionally sparse functions in \cite{10.1214/19-AOS1875, kong2025posterior} or functions supported on low-dimensional manifolds in \cite{JMLR:v21:20-002, zbMATH07714177}.

Similarly, it is of interest whether faster rates can be achieved for the estimation of the context function $\bh.$ A key step in the proof is to exploit that the transformer architecture can represent a Nadaraya--Watson estimator. The H\"older smoothness $\beta$ of the regression function improves the convergence rate of the Nadaraya--Watson estimator for $\beta\leq 1$ (e.g.\ Chapter 1 in \cite{zbMATH05358920}). For the context function $\bh$, the smoothness index is denoted by $\beta_2$ and $\beta_2\leq 1$ is imposed in Assumption \ref{assumption_gh_smoothness} (ii). While the transformer can represent the Nadaraya--Watson estimator for a specific kernel, it remains unclear whether the transformer can also represent higher-order local polynomial estimators. If possible, this would allow to exploit higher-order smoothness of $\bh$ beyond Lipschitz continuity ($\beta_2=1$). If the data are supported on a low-dimensional manifold $\mM$, the convergence rate of the Nadaraya--Watson estimator can adapt to the dimension of $\mM$ \cite{BickelBo07}. We conjecture that this phenomenon allows the transformer to circumvent the curse of dimensionality.

We derive the generalization bound for specific transformer architectures satisfying the constraints in Assumption \ref{assumption_transformer_size_lower}. In particular, we assume that the number of layers in the MLPs satisfies $L_{\mathrm{mlp}} \geq \log^3(N_{\mathrm{mlp}} \vee T).$ One can also obtain a similar generalization bound for shallow networks $L_{\mathrm{mlp}}=1$ by increasing the number of transformer blocks $L$ instead. To see this, observe that 
if the multi-head self-attention represents the identity map, the transformer block coincides with the MLP. If each MLP is shallow, then, stacking $L$ transformer blocks represents an MLP of depth $L.$ 

Alternatively, one can derive generalization bounds for models with only self-attention layers by following the recipe for translating ReLU approximation results to the softmax attention mechanism in \cite{hu2026transformer}.

\subsection{Proof of the risk bound}

We now provide a proof of Theorem \ref{theorem_main}.
Notice that the learning task is an i.i.d.\ multiclass classification problem with inputs $(\bX_1^{(i)},\ldots,\bX_T^{(i)}) \in \mathcal{E}_W^T$ and corresponding output $\bX_{T+1}^{(i)} \in \mathcal{E}_W,$ for $i=1,\ldots,n$. 
Excess risk bounds for binary classification have been studied in \cite{kim2021fast,  zbMATH07923796, 10.1214/23-EJS2187, yang2025rates, JMLR:v25:22-0049} and for multiclass classification in \cite{zbMATH07524984, 10.3150/25-BEJ1887}.
For the proof, we use the oracle inequality in \cite{zbMATH07524984}.

It is known that the VC dimension of self-attention layers involving exponential operations is difficult to control, which leads to suboptimal rates in regression problems (see, e.g., the discussion following Theorem~3 in \cite{jiao2025approximation}). 
For this reason, we instead bound the parameter magnitudes and derive an upper bound for the metric entropy of the transformer class  $\mathcal{T}(W,T,L,N,H,L_{\mathrm{mlp}},N_{\mathrm{mlp}},B)$ as follows. 
The proof is deferred to Appendix \ref{section_proof_covering_number}. 
Similar results on the covering numbers of transformers can be found in \cite{ching2026efficient} for linear attention, \cite{bai2023transformers, havrilla2024understanding} for ReLU attention, and \cite{takakura2023approximation, 2026arXiv260220555L} for softmax attention.

\begin{thm}\label{thm.covering_bd_TC}
Let $B \geq 1$ and $N_{\mathrm{mlp}} \geq N$. For any $\delta>0,$
    \begin{align*}
    &\log \, \mathcal{N}\Big(\delta, \mathcal{T}\big(W,T,L,N,H,L_{\mathrm{mlp}},N_{\mathrm{mlp}},B\big), \|\cdot\|_{L^\infty((\mE_W)^T)}\Big) \\
    &\leq \Big(2WN+ 4HLN^2+3LL_{\mathrm{mlp}}N_{\mathrm{mlp}}^2\Big)  \log\bigg(3+\frac{(6H)^{10 L^2} \big(B(N_{\mathrm{mlp}}+1)\big)^{38L^2 L_{\mathrm{mlp}}}B}{\delta}\bigg).
\end{align*}
\end{thm}

Interestingly, the bound is independent of the document length $T.$ 
For polynomially small $\delta$ and ignoring logarithmic factors, the upper bound in the theorem scales as $(WN + HLN^2 + L L_{\mathrm{mlp}} N_{\mathrm{mlp}}^2) L^2 L_{\mathrm{mlp}} \log B.$  
The bound in Theorem~5.3 of
\cite{takakura2023approximation} is similar but derived
for finite $W$ and the log covering number derived in Lemma~4 of \cite{2026arXiv260220555L}
includes an additional $\log T$ factor.

In the proposed statistical model \eqref{eq.next_token_prob} and by \eqref{eq.fh_def}, the probabilities in the vector $\bp_0(\bh, \bX_{1}, \ldots, \bX_{T})$ are up to normalization 
\[
\propto \exp\bigg( \underbrace{ \bg_0 \Big( \bm{\phi}_1(\bX_{T-r_1+1}) , \ldots, \bm{\phi}_1(\bX_{T}) \Big)^{\top} \bm{\phi}_1(\be_w)
  + \bh \Big( \bm{\phi}_2(\bX_{T-r_2+1}) , \ldots, \bm{\phi}_2(\bX_{T}) \Big)^{\top}  \bm{\phi}_2(\be_w)}_{\displaystyle \eqcolon \bbf_0(\bh, \bX_1,\ldots,\bX_T)_{(w)}} \bigg), \quad w\in[W].
\]
The following theorem provides a bound on the approximation error, and serves as a key contribution of this paper. 
We provide a proof sketch in Section \ref{sec_approximation_sketch} and the formal proof is stated in Appendix~\ref{section_proof_theorem_upper_app}.

\begin{thm} \label{theorem_upper_app}
    Suppose Assumptions~\ref{assumption_gh_smoothness} and \ref{assumption_transformer_size_lower} hold. For sufficiently large $W$ and $T$ and any $\lambda \in [T^{-1},\log^{-1} (W \land T)]$, there exists a transformer $\bbf^* \in \mathcal{T}_M (W,T,L,N,H,L_{\mathrm{mlp}},N_{\mathrm{mlp}},B)$ such that
    \[
    \sup_{\bh \in \mathcal{H}} \E \Big[ \|\bbf^{*}(\bX_{1}, \ldots, \bX_{T}) - \bbf_0(\bh, \bX_{1}, \ldots, \bX_{T})\|_{\infty}^2
    \big | \bh \Big] 
    = \tilde{O} \bigg(N_{\mathrm{mlp}}^{-\frac{4 \beta_1}{m_1 r_1}} +  \lambda^{2 \beta_2} + \frac{M^2}{T \varphi_W(\lambda)^{r_2}} + \kappa_W^2\bigg),
    \]
    where the implicit    
    constant depends only on $(m_1, m_2, r_1, r_2, \delta_1, \delta_2, \beta_1, \beta_2, \| \bg_0 \|_{\mathcal{C}^{\beta_1}}, \sup_{\bh \in \mathcal{H}} \| \bh \|_{\mathcal{C}^{\beta_2}})$.
\end{thm}
The main difficulty in the construction underlying Theorem \ref{theorem_upper_app} is that $\bh$ has to be approximated by the data $\bX_1, \ldots, \bX_T$. 
Compared to in-context nonparametric regression \cite{kim2024transformers, ma2025provable, shen2026understanding, ching2026efficient}, the considered statistical model for next-token prediction is significantly more challenging to analyze and requires novel techniques.
In addition, the error in estimating $\bh$ should not deteriorate due to the presence of $\bg_0$, which further complicates the analysis.
See Remark~\ref{remark_advantage_mhat} for further discussion.

\begin{proof}[Proof of Theorem~\ref{theorem_main}]
By assumption $M \geq \delta_1 m_1 + \delta_2.$ Since
\[
    \bp_0(\bh, \bX_{1}, \ldots, \bX_{T}) =
    \operatorname{softmax} \circ \,  \bbf_0(\bh, \bX_1,\ldots,\bX_T) 
\]
and the $w$-th component of $\bbf_0$ has the form,
\[
    \bbf_0(\bh, \bX_1,\ldots,\bX_T)_{(w)} \coloneqq \bg_0 \Big( \bm{\phi}_1(\bX_{T-r_1+1}) , \ldots, \bm{\phi}_1(\bX_{T}) \Big)^{\top} \bm{\phi}_1(\be_w)
  + \bh \Big( \bm{\phi}_2(\bX_{T-r_2+1}) , \ldots, \bm{\phi}_2(\bX_{T}) \Big)^{\top}  \bm{\phi}_2(\be_w),
\]
we obtain $\|\bbf_0(\bh, \bX_{1}, \ldots, \bX_{T})\|_{\infty} \leq \delta_1 m_1 + \delta_2\leq M$ and
\[
    \bp_0(\bh, \bX_{1}, \ldots, \bX_{T}) 
    < \frac{\exp(2M)}{W}. 
\]
Because of the clipping operation in the transformer function class (Definition \ref{defi.transfomer_class}), every $\wh{\bp} \in \mathcal{P}$ satisfies
$
  \exp(-2M)/W  
  < \wh{\bp}(\bX_{1}, \ldots, \bX_{T}),
$
and hence
\[
\log\bigg(\frac{\bp_0(\bh, \bX_{1}, \ldots, \bX_{T})}{\wh{\bp}(\bh, \bX_{1}, \ldots, \bX_{T})}\bigg) \leq 4M.
\]
By using the oracle inequality for multiclass classification
(Theorem 3.5 of \cite{zbMATH07524984}), 
\begin{align*}
    \E \bigg[\log\frac{ \bX_{T+1}^\top  \bp_0(\bh, \bX_{1}, \ldots, \bX_{T})}{ \bX_{T+1}^\top  \wh \bp(\bX_{1}, \ldots, \bX_{T})}
    \bigg] 
    &= \frac{M}{n} \log \, \mathcal{N}\Big(\frac{1}{n},  \log(\mathcal{P}), \|\cdot\|_{L^\infty((\mE_W)^T)}\Big)
    && \text{(stochastic error)}
 \\
    &\quad +  \tilde{O}\bigg( \inf_{\bp \in \mathcal{P}} \E \bigg[\log\frac{ \bX_{T+1}^\top  \bp_0(\bh, \bX_{1}, \ldots, \bX_{T})}{ \bX_{T+1}^\top  \bp(\bX_{1}, \ldots, \bX_{T})}
    \bigg]\bigg)  && \text{(approximation error)},
\end{align*}
where
$\log(\mathcal{P}) \coloneqq \{\bbf = \log(\bp) : \bp \in \mathcal{P}\}$ and $\|\bbf\|_{L^\infty((\mE_W)^T)} \coloneqq
\max \{ \|\bbf(\bX_{1}, \ldots, \bX_{T})\|_{\infty} : \bX_{1}, \ldots, \bX_{T} \in \mE_W\}.$

Define
\[\mathcal{F} \coloneqq \mathcal{T}_M (W,T,L,N,H,L_{\mathrm{mlp}},N_{\mathrm{mlp}},B).\]
Combining $\log(\mathcal{P}) = \{\log \circ \operatorname{softmax} \circ \, \bbf : \bbf \in \mathcal{F}\}$, the inequality 
\[
\big\|\log \circ \operatorname{softmax} \circ \operatorname{clip}_{M}(\bx) - \log \circ \operatorname{softmax} \circ \operatorname{clip}_{M}(\bx') \big\|_{\infty} \leq 2 \big\| \bx - \bx' \big\|_{\infty} 
\]
and Theorem~\ref{thm.covering_bd_TC} yield
\[
\frac{M}{n} \log \, \mathcal{N}\Big(\frac{1}{n},  \log(\mathcal{P}), \|\cdot\|_{L^\infty((\mE_W)^T)}\Big)\bigg) = \tilde{O}\Big(\frac{M (WN + HLN^2 + L L_{\mathrm{mlp}} N_{\mathrm{mlp}}^2 + L^2 L_{\mathrm{mlp}} + \log B)}{n}\Big).
\]
Using Lemma \ref{lemma_KL_MSE}, 
the approximation error can be further bounded by
\begin{align*}
    \inf_{\bp \in \mathcal{P}} \E \bigg[\log\frac{ \bX_{T+1}^\top  \bp_0(\bh, \bX_{1}, \ldots, \bX_{T})}{ \bX_{T+1}^\top  \bp(\bX_{1}, \ldots, \bX_{T})}
    \bigg]
    &=
    \inf_{\bp \in \mathcal{P}} \E \bigg[\sum_{w=1}^W \bp_0(\bh, \bX_{1}, \ldots, \bX_{T})_{(w)} \log\frac{  \bp_0(\bh, \bX_{1}, \ldots, \bX_{T})_{(w)}}{ \bp(\bX_{1}, \ldots, \bX_{T})_{(w)}}
    \bigg]\\
    &\leq 4 \inf_{\bbf \in \mathcal{F}} \, \E \Big[ \big\|\bbf_0(\bh, \bX_{1}, \ldots, \bX_{T}) - \bbf(\bX_{1}, \ldots, \bX_{T})\big\|_{\infty}^2
    \Big]\\
    &\leq 4 \sup_{\bh \in \mathcal{H}} \, \inf_{\bbf \in \mathcal{F}} \, \E\Big[ \big\|\bbf_0(\bh, \bX_{1}, \ldots, \bX_{T}) - \bbf(\bX_{1}, \ldots, \bX_{T})\big\|_{\infty}^2
    \, \Big| \, \bh\Big].
\end{align*}
Together with Theorem \ref{theorem_upper_app}, we get
\[
 \inf_{\bp \in \mathcal{P}} \E \bigg[\log\frac{ \bX_{T+1}^\top  \bp_0(\bh, \bX_{1}, \ldots, \bX_{T})}{ \bX_{T+1}^\top  \bp(\bX_{1}, \ldots, \bX_{T})}
    \bigg] = \tilde{O}\bigg(N_{\mathrm{mlp}}^{-\frac{4 \beta_1}{m_1 r_1}} + \inf_{\lambda \in [T^{-1},\log^{-1} (W \land T)]} \, \Big( \lambda^{2 \beta_2} + \frac{M^2}{T \varphi_W(\lambda)^{r_2}} \Big)  + \kappa_W^2\bigg).
\]
\end{proof}

\section{Proof Sketch of the Upper Bound for the Approximation Error}
\label{sec_approximation_sketch}

We now outline the construction of a transformer $\bbf^* \in \mathcal{T}(W,T,L,N,H,L_{\mathrm{mlp}},N_{\mathrm{mlp}},B)$ satisfying the approximation error bound in  Theorem~\ref{theorem_upper_app}. Recall that the input tokens $\bX_1,\ldots, \bX_T$ are coded as $W$-dimensional standard basis vectors. The input of the transformer is the $W \times T$ matrix 
    $$X \coloneq
    \left[\begin{matrix}
        \bX_1 & \bX_2 & \ldots & \bX_T
    \end{matrix}
    \right].$$
It suffices to describe the construction of an encoder $E \in \mathbb{R}^{N \times W}$, a decoder $D \in \mathbb{R}^{W \times N}$, multi-head self-attentions $\bm{\mu}_1, \bm{\mu}_2 \in \mathcal{MA}(N,H,B),$ and multilayer perceptrons $\bm{\nu}_1, \bm{\nu}_2 \in \mathcal{MLP}(N, L_{\mathrm{mlp}},N_{\mathrm{mlp}}, B)$, such that the $T$-th output column of the transformer is
\begin{align*}
    \bbf^*(X) & = \operatorname{clip}_{M} \circ \, \Big(D \cdot \bm{\nu}_2 \circ \bm{\mu}_2 \circ \bm{\nu}_1 \circ \bm{\mu}_1 (P + E X)\Big)_{:,T}\\ &\approx \bbf_0(\bh, X)     
    \coloneqq \left( \begin{matrix}
    \bg_0 \big( \bm{\phi}_1(\bX_{T-r_1+1}) , \ldots, \bm{\phi}_1(\bX_{T}) \big)^{\top} \bm{\phi}_1(\be_1)
  + \bh \big( \bm{\phi}_2(\bX_{T-r_2+1}) , \ldots, \bm{\phi}_2(\bX_{T}) \big)^{\top}  \bm{\phi}_2(\be_1)\\
\bg_0 \big( \bm{\phi}_1(\bX_{T-r_1+1}) , \ldots, \bm{\phi}_1(\bX_{T}) \big)^{\top} \bm{\phi}_1(\be_2)
  + \bh \big( \bm{\phi}_2(\bX_{T-r_2+1}) , \ldots, \bm{\phi}_2(\bX_{T}) \big)^{\top}  \bm{\phi}_2(\be_2)\\
  \vdots\\
\bg_0 \big( \bm{\phi}_1(\bX_{T-r_1+1}) , \ldots, \bm{\phi}_1(\bX_{T}) \big)^{\top} \bm{\phi}_1(\be_W)
  + \bh \big( \bm{\phi}_2(\bX_{T-r_2+1}) , \ldots, \bm{\phi}_2(\bX_{T}) \big)^{\top}  \bm{\phi}_2(\be_W)
    \end{matrix} \right).
\end{align*}
See Figure~\ref{fig_overall_structure} for an illustration.
The main challenge is that we cannot use the latent local context function $\bh$ to construct $\bbf^*$. Instead, $\bh$ has itself to be approximated by the available input data.
\begin{figure}[t]
  \centering  \includegraphics[width=.8\textwidth]{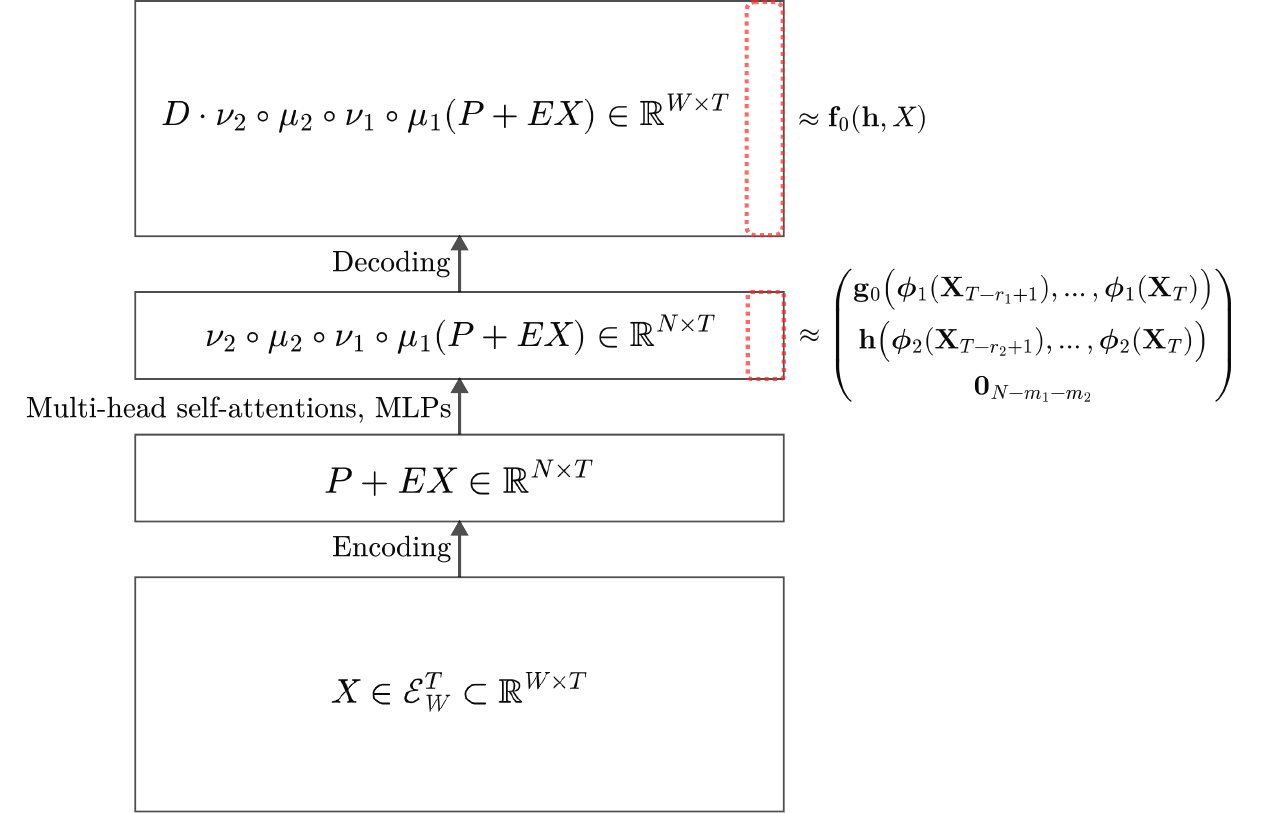}
\caption{{Main building blocks of the transformer construction to control the approximation error.} 
    } \label{fig_overall_structure}
\end{figure}
For the encoder operation, we use the $N \times W$ matrix 
\begin{align}
    E \coloneq \left[\begin{matrix}
\bm{\phi}_1(\be_1) & \bm{\phi}_1(\be_2) & \cdots & \bm{\phi}_1(\be_W) \\
\bm{\phi}_2(\be_1) & \bm{\phi}_2(\be_2) & \cdots & \bm{\phi}_2(\be_W) \\
\bm{0}_{N-m_1 - m_2} & \bm{0}_{N-m_1 - m_2} & \cdots & \bm{0}_{N-m_1 - m_2} 
\end{matrix}\right]. \label{def_E}
\end{align}
For every token $w \in [W]$, we have $E \be_w = \stack(\bm{\phi}_{1}(\be_w), \bm{\phi}_{2}(\be_w), \bm{0}_{N-m_1-m_2}),$ consequently,
$$EX = E \big[\bX_1, \bX_2, \ldots, \bX_T\big] =
\left[\begin{matrix}
\bm{\phi}_1(\bX_1) & \bm{\phi}_1(\bX_2) & \cdots & \bm{\phi}_1(\bX_T) \\
\bm{\phi}_2(\bX_1) & \bm{\phi}_2(\bX_2) & \cdots & \bm{\phi}_2(\bX_T) \\
\bm{0}_{N-m_1-m_2} & \bm{0}_{N-m_1-m_2} & \cdots & \bm{0}_{N-m_1-m_2}
\end{matrix}\right],$$
and combined with the definition of the positional encoding matrix $P$ (Definition \ref{def_PE}),
$$P + E X = \left[\begin{matrix}
\bm{\phi}_1(\bX_1) & \bm{\phi}_1(\bX_2) & \cdots & \bm{\phi}_1(\bX_T) \\
\bm{\phi}_2(\bX_1) & \bm{\phi}_2(\bX_2) & \cdots & \bm{\phi}_2(\bX_T) \\
\bm{0}_{\tfrac{N}{2} -m_1-m_2} & \bm{0}_{\tfrac{N}{2} -m_1-m_2} & \cdots & \bm{0}_{\tfrac{N}{2} -m_1-m_2} \\
\bq_1 & \bq_2 & \cdots & \bq_T 
\end{matrix}\right] = 
\left( 
\begin{matrix}
    \bm{\phi}_1(\bX_t)\\
    \bm{\phi}_2(\bX_t)\\
    \bm{0}_{\tfrac{N}{2} -m_1-m_2}\\
    \bq_t
\end{matrix}
\right)_{t \in [T]}.$$

The projected input $P + E X$ is propagated to the first multi-head self attention layer $\bm{\mu}_1 \in \mathcal{MA}(N,H,B)$.
For $t \in [T]$, we define $\bZ_{1,t} \in \mathbb{R}^{r_1 m_1}$ and $\bZ_{2,t} \in \mathbb{R}^{r_2 m_2}$ as the linguistic and context embeddings of the most recent $r_1$ and $r_2$ tokens up to time $t-1$, respectively, defined as
\begin{align}
    \bZ_{1,t}
\coloneqq
\begin{pmatrix}
\bm{\phi}_1(\bX_{t-r_1}) \\
\bm{\phi}_1(\bX_{t-r_1+1}) \\
\vdots \\
\bm{\phi}_1(\bX_{t-1})
\end{pmatrix},
\qquad
\bZ_{2,t}
\coloneqq
\begin{pmatrix}
\bm{\phi}_2(\bX_{t-r_2}) \\
\bm{\phi}_2(\bX_{t-r_2+1}) \\
\vdots \\
\bm{\phi}_2(\bX_{t-1})
\end{pmatrix}.
\label{def_Z}
\end{align}
By exploiting the positional encoding, one can construct a multi-head self-attention layer $\bm{\mu}_1 \in \mathcal{MA}(N,H,B)$ such that
$$
    \bm{\mu}_1\left(\left( 
\begin{matrix}
    \bm{\phi}_1(\bX_t)\\
    \bm{\phi}_2(\bX_t)\\
    \bm{0}_{\tfrac{N}{2} -m_1-m_2}\\
    \bq_t
\end{matrix}
\right)_{t \in [T]}
\right) \approx \left(\begin{matrix}
    \bm{\phi}_1(\bX_t)\\
    \bm{\phi}_2(\bX_t)\\
    \bZ_{1,t} \\
    \bZ_{2,t} \\
    \bm{0}_{\tfrac{N}{2}- (r_1+1)m_1 -  (r_2+1)m_2}\\
    \bq_t
    \end{matrix}\right)_{t \in [T]}
    \begin{array}{l}
        \in \mathbb{R}^{m_1}\\
        \in \mathbb{R}^{m_2}\\
        \in \mathbb{R}^{r_1 m_1}\\
        \in \mathbb{R}^{r_2 m_2}\\
        \left. \right.\\        
        \in \mathbb{R}^{N/2}\\
    \end{array}
    .
$$ 
The approximation is valid for the $t$-th column for $t \geq (r_1 \vee r_2) + 1$.
See Lemma \ref{lemma_approx_rgram} for a bound on the approximation error.

We now employ a multilayer perceptron to approximate two key quantities:  $\bg_0(\bZ_{1,t+1})$, corresponding to the linguistic information, and 
$$\bY_{t} \coloneqq \frac{\bm{\phi}_2(\bX_t)}{\exp\big(\bg_0(\bZ_{1,t} )^{\top} \bm{\phi}_1(\bX_t)\big)}  \in \mathbb{R}^{m_2},$$
which will later be used to estimate the local context function $\bh$.
Specifically, we construct a multi-layer perceptron $\bm{\nu}_1 \in \mathcal{MLP}(N, L_{\mathrm{mlp}},N_{\mathrm{mlp}}, B)$ such that
\begin{align}
    \bm{\nu}_1 \left( \left( \begin{matrix}
    \bm{\phi}_1(\bX_t)\\
    \bm{\phi}_2(\bX_t)\\
    \bZ_{1,t} \\
    \bZ_{2,t} \\
    \bm{0}_{\tfrac{N}{2}- (r_1+1)m_1 -  (r_2+1)m_2}\\
    \bq_t
    \end{matrix}\right)_{t \in [T]} \right) \approx 
    \left(\begin{matrix}
    \bg_0(\bZ_{1,t+1}) \\
    \bY_{t}\\
    \frac{1}{\lambda} \bZ_{2, t}\\
    \frac{1}{\lambda} \bm{\phi}_2(\bX_t)\\
    \frac{\sqrt{r_2}}{\lambda} \\
    T^3 \, \mathbb{I}(t \in [\log^2 T, T - \log^2 T])\\
        \bm{0}_{N - m_1 - (r_2+2)m_2-2}
    \end{matrix}\right)_{t \in [T]}
        \begin{array}{l}
            \in \mathbb{R}^{m_1}\\
            \in \mathbb{R}^{m_2}\\
            \in \mathbb{R}^{r_2 m_2}\\
            \in \mathbb{R}^{m_2}\\
            \in \mathbb{R}\\
            \in \mathbb{R}\\
            \left. \right.        
        \end{array}.
    \label{tmp_9}
\end{align}
The approximation error is controlled in Lemma \ref{lemma_construct_firstmlp}.

For $\bz \in \mathbb{R}^{r_2 m_2}$, define 
\begin{align}
    \wh{\bbm}_{\lambda}(\bz) \coloneqq
\sum_{s \in [\log^2 T, T - \log^2 T]}
    \frac{K_\lambda\big(\bZ_{2,s}, \bz\big)}{\sum_{i \in [\log^2 T, T - \log^2 T]} K_\lambda\big(\bZ_{2,i}, \bz\big) } \bY_{s}.
    \label{def_mhat_m}
\end{align}
with $K_\lambda(\bz,\bz')\coloneq \exp(-\Vert \bz - \bz' \Vert^2/(2\lambda^2)).$ The second multi-head self-attention layer $\bm{\mu}_2 \in \mathcal{MA}(N, H, B)$ is constructed to map the right-hand side of \eqref{tmp_9} to a $N \times T$ matrix.
The last column of this matrix is given by
$\stack(\bg_0(\bZ_{1,T+1}), \bullet_{N-m_1-m_2}, \bv) \in \mathbb{R}^N$, where the vector $\bv \in \mathbb{R}^{m_2}$ is defined as
    \begin{align}
    \bv 
    \coloneqq & \sum_{s=1}^T  \frac{\exp \Big(\frac{-r_2+\bZ_{2,s}^{\top} \bZ_{2,T+1}}{\lambda^2} + T^3 \, \mathbb{I}(s \in [\log^2 T, T - \log^2 T])\Big)}{\sum_{i=1}^T \exp \Big(\frac{-r_2+\bZ_{2,i}^{\top} \bZ_{2,T+1}}{\lambda^2} + T^3 \, \mathbb{I}(i \in [\log^2 T, T - \log^2 T])\Big)} \bY_{s} \approx \wh{\bbm}_{\lambda}(\bZ_{2,T+1}).
    \label{tmp_5}
    \end{align}
    This approximation follows from the fact that $\|\bZ_{2,t}\|_2^2 = r_2$ for every $t \geq r_2 + 1$.
    The following lemma characterizes the approximation error of the construction up to this point.
    In  Appendix~\ref{prove_lemma_approx_firstpart}, we restate the result as Lemma~\ref{lemma_approx_firstpart_rewrite} and provide the formal proof.
\begin{lemma} \label{lemma_approx_firstpart}
    Suppose Assumptions \ref{assumption_gh_smoothness} and \ref{assumption_transformer_size_lower} hold. 
    For sufficiently large $T$ and 
    for any $\lambda \geq T^{-1}$, there exist multi-head self-attentions $\bm{\mu}_1, \bm{\mu}_2 \in \mathcal{MA}(N,H,B)$ and
    a multilayer perceptron $\bm{\nu}_1 \in \mathcal{MLP}(N, L_{\mathrm{mlp}}, N_{\mathrm{mlp}}, B)$ such that
    \begin{align*}
        \Big( \bm{\mu}_2 \circ \bm{\nu}_1 \circ \bm{\mu}_1 \big(P + E X \big) \Big)_{:,T}
        =
            \left( \begin{matrix} 
        \bk_1\\
        \bullet_{N-m_1-m_2}\\
        \bk_2
        \end{matrix}\right)
            \begin{array}{l}
        \in \mathbb{R}^{m_1}\\
        \left. \right.\\   
        \in \mathbb{R}^{m_2}
    \end{array}
    \end{align*}
    with
    \begin{align*}
        \Big\| \bk_1 - \bg_0(\bZ_{1,T+1}) \big\|_{\infty} \vee \,
        \big\| \bk_2 - \wh{\bbm}_{\lambda}(\bZ_{2,T+1})\big\|_{\infty}
    \lesssim (N_{\mathrm{mlp}})^{-\frac{2\beta_1}{m_1 r_1}} + \frac{1}{T},
    \end{align*}
     where the implicit    
    constant depends only on $(m_1, m_2, r_1, r_2, \delta_1, \delta_2, \beta_1, \| \bg_0 \|_{\mathcal{C}^{\beta_1}})$.

\end{lemma}
We now study the properties of \eqref{tmp_5}.
The estimator
$\wh{\bbm}_{\lambda}$ can be identified as a Nadaraya--Watson estimator applied to dependent samples with discrete support.  
From the classical nonparametric regression theory (see e.g.\ Section 1.5 in \cite{zbMATH05358920}), one would expect that $\wh{\bbm}_{\lambda}(\bZ_{2,T+1})$ converges to the conditional expectation $\mathbb{E}[ \bY_{T+1} | \bh, \bZ_{2,T+1}]$ under standard smoothness assumptions on the function $\bh$.
In the present setting, however, such convergence cannot be rigorously established, since 
$\mathbb{E}[ \bY_{T+1} | \bh, \bZ_{2,T+1} = \bz]$ can vary considerably as a function of $\bz$.
To see this, recall that for $t \in [T]$, the distribution of the tokens is described by the imposed statistical model
\begin{align}    \mathbb{P}\big(\bX_{t+1} = \be_w \big| \bh, \bX_1, \ldots, \bX_{t} \big) 
\propto
\exp\Big(\bg_0(\bZ_{1,t+1})^{\top} \bm{\phi}_1(\be_w) + \bh(\bZ_{2,t+1})^{\top} \bm{\phi}_2(\be_w)\Big), \quad w \in [W].
\label{tmp_57}
\end{align}
One can factorize the conditional expectation in this model via
    \begin{align}
    &\mathbb{E}[ \bY_{T+1} |\bh,  \bZ_{2,T+1} = \bz] = \mathbb{E}\left[ \frac{\bm{\phi}_2(\bX_{T+1})}{\exp\big(\bg_0(\bZ_{1,T+1} )^{\top} \bm{\phi}_1(\bX_{T+1})\big) } \middle| \bh, \bZ_{2,T+1} = \bz \right] \nonumber \\
    & = \mathbb{E}\left[\frac{\sum_{w=1}^W \exp\big(\bg_0(\bZ_{1,T+1})^{\top} \bm{\phi}_1(\be_w) + \bh(\bz)^{\top} \bm{\phi}_2(\be_w)\big) \frac{\bm{\phi}_2(\be_{w})}{\exp\big(\bg_0(\bZ_{1,T+1})^{\top} \bm{\phi}_1(\be_w)\big)}}{\sum_{w=1}^W \exp\big(\bg_0(\bZ_{1,T+1})^{\top} \bm{\phi}_1(\be_w) + \bh(\bz)^{\top} \bm{\phi}_2(\be_w)\big)} \middle|\bh, \bZ_{2,T+1} = \bz \right] \nonumber \\
    & = \mathbb{E}\left[\frac{\sum_{w=1}^W \exp\big( \bh(\bz)^{\top} \bm{\phi}_2(\be_w)\big)
    \bm{\phi}_2(\be_{w})}{\sum_{w=1}^W \exp\big(\bg_0(\bZ_{1,T+1})^{\top} \bm{\phi}_1(\be_w) + \bh(\bz)^{\top} \bm{\phi}_2(\be_w)\big)} \middle|\bh, \bZ_{2,T+1} = \bz \right] \nonumber \\
    & = \underbrace{\frac{\sum_{w=1}^W \exp\big( \bh(\bz)^{\top} \bm{\phi}_2(\be_w) \big) \bm{\phi}_2(\be_{w})}{\sum_{w=1}^W \exp\big( \bh(\bz)^{\top} \bm{\phi}_2(\be_w) \big)}}_{\displaystyle \in \, \mathbb{R}^{m_2}} \, \underbrace{\mathbb{E}\left[\frac{\sum_{w=1}^W \exp\big(\bh(\bz)^{\top} \bm{\phi}_2(\be_w)\big)}{\sum_{w=1}^W \exp\big(\bg_0(\bZ_{1,T+1})^{\top} \bm{\phi}_1(\be_w) + \bh(\bz)^{\top} \bm{\phi}_2(\be_w)\big)} \middle|\bh, \bZ_{2,T+1} = \bz \right]}_{\displaystyle \in \, \mathbb{R}}. \label{tmp_33} 
    \end{align} 
    The two vectors $\bZ_{1,T+1}$ and $\bZ_{2,T+1}$ are the values with respect to the embeddings $\bm{\phi}_1$ and $\bm{\phi}_2.$ Points that are mapped to nearby points with respect to $\bm{\phi}_2$ could be mapped to far away points with respect to $\bm{\phi}_1.$ Thus, small changes in $\bZ_{2,T+1}=\bz$ can lead to a large change in the conditional expectation of $\bZ_{1,T+1}|\bZ_{2,T+1}=\bz$ such that we cannot guarantee any H\"older smoothness of the second factor in \eqref{tmp_33}. The first factor of \eqref{tmp_33} does, however, not depend on the interplay between $\bm{\phi}_1$ and $\bm{\phi}_2.$ For the
    direction of $\bz\mapsto \mathbb{E}[ \bY_{T+1} | \bh, \bZ_{2,T+1} = \bz]$ it is therefore possible to establish H\"older smoothness.
    Consequently, one can show that the direction of $\wh{\bbm}_{\lambda}(\bz)$ converges to that of $\mathbb{E}[ \bY_{T+1} \mid \bh, \bZ_{2,T+1} = \bz]$, which is close to $\bh/\|\bh\|_2$.
    With this intuition 
    and properties of the von Mises--Fisher distribution, we present the following key lemma that in particular allows us later to construct a high-probability lower bound of $\| \wh{\bbm}_{\lambda}(\bZ_{2,T+1})  \|_2$. The proof of the lemma is deferred to Appendix~\ref{section_proof_lemma_key_NW}.

\begin{lemma}\label{lemma_key_NW_version2}
Suppose Assumption \ref{assumption_gh_smoothness} holds. For sufficiently large $W$ and $T$ and 
    for any $\lambda \leq \log^{-1} (W \land T)$,
\begin{align*}
    \mathbb{E}\left[\left\|\frac{\delta_2}{\|\wh{\bbm}_{\lambda}(\bZ_{2,T+1})\|_2} \wh{\bbm}_{\lambda}(\bZ_{2,T+1})
- \bh(\bZ_{2,T}) \right\|^2_2 \, \bigg| \, \bh \right] = \tilde{O}\Big(\lambda^{2 \beta_2} + \tfrac{1}{T \varphi_W(\lambda)^{r_2}} + \kappa_W^2\Big).
\end{align*}   
Moreover, there exists a constant $c>0$  such that
    \[
        \mathbb{P}\big[ \| \wh{\bbm}_{\lambda}(\bZ_{2,T+1})  \|_2 < c \, \big| \, \bh \big] 
    \lesssim \frac{1}{T \varphi_W(\lambda)^{r_2}}.
    \]
All implicit constants and the constant $c$ depend only on $(m_1, m_2, r_1, r_2, \delta_1, \delta_2, \beta_2, \|\bh\|_{\mathcal{C}^{\beta_2}})$.
\end{lemma}

\begin{rem} \label{remark_advantage_mhat}
The convergence rate in Lemma~\ref{lemma_key_NW_version2} does not depend on $\bg_0$ although the true probability model depends on $\bg_0$. 
In an additive regression setting $Y_t = \bg_0(\bX_{1,t}) + \bh(\bX_{2,t}) + \epsilon_t$, $\E(\epsilon_t)=0,$ $t=1,\ldots,T,$ such a result can be easily obtained by considering the adjusted samples $\{\bX_{2,t},\, Y_t - \bg_0(\bX_{1,t})\}_{t=1}^T$. 
However, we argue that this approach cannot be extended to the
statistical model \eqref{tmp_57} even when $\bg_0$, $\bm{\phi}_1$, and $\bm{\phi}_2$ are known.
\end{rem}

By Lemma \ref{lemma_approx_firstpart} and Lemma \ref{lemma_key_NW_version2}, 
    \begin{align*}
        \Big( \bm{\mu}_2 \circ \bm{\nu}_1 \circ \bm{\mu}_1 \big(P + E X \big) \Big)_{:,T}
        \eqcolon
            \left( \begin{matrix} 
        \bk_1\\
        \bullet_{N-m_1-m_2}\\
        \bk_2
        \end{matrix}\right)
            \begin{array}{l}
        \in \mathbb{R}^{m_1}\\
        \left. \right.\\   
        \in \mathbb{R}^{m_2}
    \end{array}
    \end{align*}
    satisfies $\bk_1 \approx \bg_0(\bZ_{1,T+1})$ 
    and $\delta_2 \bk_2 / \|\bk_2\|_2 \approx \bh(\bZ_{2,T+1})$. 
    In addition, $\|\bk_2\|_2 \geq c/2$ holds with high probability.
Using these results, we construct the second MLP layer $\bm{\nu}_2 \in \mathcal{MLP}(N, L_{\mathrm{mlp}},N_{\mathrm{mlp}}, B)$ such that
\[
        \bm{\nu}_2  \left( \begin{matrix}
    \bk_1 \\
        \bullet_{N-m_1 - m_2}\\
    \bk_2
    \end{matrix} \right)
    \approx
    \left( \begin{matrix}
    \bg_0(\bZ_{1,T+1}) \\
    \bh(\bZ_{2,T+1})   \\
    \bm{0}_{N-m_1 - m_2}\\
    \end{matrix} \right)
   \begin{array}{l}
        \in \mathbb{R}^{m_1}\\
        \in \mathbb{R}^{m_2}\\
        \left. \right.   
    \end{array}.
\]
See Lemma~\ref{lemma_normalize} for the detailed approximation error.
By considering the decoder matrix $D \coloneq E^{\top}$,
we finally get
\begin{align*}
    D \left(\begin{matrix}
    \bg_0(\bZ_{1,T+1}) \\
    \bh(\bZ_{2,T+1}) \\
    \bm{0}_{N-m_1 - m_2}
    \end{matrix}\right) = \left(
\begin{array}{c}
\bm{\phi}_1(\be_1)^{\top} \bg_0(\bZ_{1,T+1})  + \bm{\phi}_2(\be_1)^{\top} \bh(\bZ_{2,T+1})     \\
\bm{\phi}_1(\be_2)^{\top} \bg_0(\bZ_{1,T+1})  + \bm{\phi}_2(\be_2)^{\top} \bh(\bZ_{2,T+1})     \\
\vdots\\
\bm{\phi}_1(\be_W)^{\top} \bg_0(\bZ_{1,T+1})  + \bm{\phi}_2(\be_W)^{\top} \bh(\bZ_{2,T+1})     \\
\end{array}
\right) = \bbf_0(\bh, X).
\end{align*}
See Section~\ref{section_proof_theorem_upper_app} for a formal proof of Theorem~\ref{theorem_upper_app}.

\section{Related work}
\label{sec.rel_work}

There is by now a substantial literature on the statistical learning theory of neural networks. 
For DNNs, theoretical results have been developed for a wide range of statistical tasks, including nonparametric regression and classification, with convergence rates depending on smoothness, (intrinsic) dimension, or compositional structure \cite{zbMATH07082286, suzuki2018adaptivity, 10.1214/19-AOS1875, kohler2021rate, lu2021deep, zbMATH07714177, yang2024nonparametric}.
Statistical guarantees have also been obtained for other architectures.
CNNs have been analyzed under structural assumptions motivated by the translation-invariance of image data \cite{zbMATH07604562, zbMATH07883300, chen2022novel, zbMATH08062197, zbMATH07923796}. 
The statistical analysis of recurrent neural networks (RNNs) focuses on sequential data with temporal dependence \cite{kohler2023rate, jiao2024approximation, yu2026nonparametric}.

As transformers have become central to modern machine learning, recent studies started to investigate their approximation and representation capabilities.
Among them, \cite{Yun2020Are} proved universal approximation results for continuous sequence-to-sequence functions and analyzed the roles of self-attention and feed-forward layers.
\cite{pmlr-v162-edelman22a} showed that transformers can efficiently represent sparse functions of long input sequences, while
\cite{sanford2023representational}
identified tasks where transformers are more efficient than DNNs and RNNs.
\cite{hu2026transformer} showed that self-attention can approximate generalized ReLU-type operations, providing another route to universal approximation.

Recent statistical papers provided a statistical analysis of transformers within the standard supervised learning framework. 
\cite{gurevych2022rate} analyzed nonparametric binary classification with transformers under hierarchical composition structure, and \cite{kohler2023ratetransformer} extended this result to overparametrized transformer classifiers trained by gradient descent.
\cite{takakura2023approximation} studied nonparametric regression with transformers and established convergence rates for sequence-to-sequence functions with infinite-dimensional inputs under suitable assumptions on the target function.
\cite{havrilla2024understanding} analyzed transformers with ReLU self-attention for nonparametric regression and input data supported on a manifold, and \cite{shen2025transformers} extended this analysis to noisy manifold settings. 
Most recently, \cite{jiao2025approximation, 2026arXiv260220555L, shi2026learning} analyzed
transformers for nonparametric regression under H\"older smoothness, while \cite{shi2025approximation} studied vision transformers with hierarchically compositional target functions.

Another line of statistical research studies the theoretical properties of transformers for in-context learning in regression settings, showing that they can generalize to unseen tasks.
For instance,
\cite{bai2023transformers, von2023transformers} showed that linear-attention transformers can implement standard learning algorithms, including gradient descent, in context.
\cite{kim2024transformers} analyzed transformers for in-context nonparametric regression and showed that pretraining can improve in-context estimation by encoding basis representations relevant to the target function class.
\cite{ma2025provable, wakayama2025context} studied in-context learning with transformers pretrained on mixtures of nonparametric regression tasks, establishing adaptation to task difficulty and robustness to certain forms of test-time distribution shift.
Recently, \cite{shen2026understanding} and \cite{ching2026efficient} connected transformers with the Nadaraya--Watson and local polynomial estimators for in-context nonparametric regression.

The proposed N2F language model generates the next token from a latent context function and a finite token history.
Topic models have a rich statistical theory and can also be viewed as latent context models. Specifically, in the probabilistic latent semantic indexing (pLSI) model \cite{Hofmann1999} one assumes that for each observed document there is a latent weight vector that determines the mixture weights over $K$ topics (e.g.\ sport, politics, news, \ldots ). The word frequency in the document is then modeled as weighted sum over the word frequencies of the individual topics. The pLSI model is considerably simpler, as it only models the word frequencies but not the text as a sequence. The latent object is moreover a weight vector and not a function as in the proposed model in Section \ref{sec_general_next_token}. There is an impressive line of existing statistical theory \cite{6375276, pmlr-v28-arora13, NEURIPS2018_731c83db, zbMATH07193942, JMLR:v21:20-079, zbMATH07820394, zbMATH07751813, zbMATH07641127, zbMATH07783529, 2023arXiv231006730T}, but the assumptions, the mathematical analysis, and the theoretical results differ significantly from our work.

A growing line of recent research applies Markov chain theory to study language models that generate token from a finite token history.
In particular, \cite{lotfi2024unlocking, zekri2024large, yuksel2025sample} analyzed estimation errors under the assumption that the training and test documents all share the same transition probabilities.
In contrast, in the proposed N2F language model each document has its own transition probabilities.
\cite{varre2025learning} proved that transformers can represent the $r$-gram estimator, and that the corresponding parameters form a stationary point of the population cross-entropy loss.
This analysis is different as it considers a single document, does not introduce an estimator, and relies on the population cross-entropy loss.

Estimating Markov transition matrices is a classical problem \cite{anderson1957statistical, billingsley1961statistical, morvai1996nonparametric}.
If the Markov chain of interest has only finitely many states, a natural estimator of the next-token distribution is given by the empirical frequency. The $r$-gram estimator is a special case of the empirical frequency under the implicit assumption that each document is generated according to a Markov chain of order $r$. 
If applied to the language model setting, \cite{han2023optimal} proved that if $\text{(document length)} \geq \text{(vocabulary size)}^{r+1}$, then an averaged version of an additively smoothed empirical frequency estimator is minimax optimal. This suggests that the $r$-gram estimator is minimax rate-optimal. However, for text data the transition matrices are highly structured and should allow for faster estimation rates. Moreover, 
$\text{(document length)} \ll \text{(vocabulary size)}^{r+1}$ is the more relevant regime.

The derived theory is also related to nonparametric estimation for Markov processes. Past research has for example focused on invariant density estimation \cite{dexheimer22,strauch18} and drift estimation for diffusion processes \cite{strauch16,hoffmann25,gobet04}, with Nadaraya--Watson type estimators studied in e.g.\ \cite{dexheimer22kindiff,dexheimer25,aeck22}.

\section*{Acknowledgments}
We are extremely grateful to Richard Samworth and Tianyi Ma for sharing their insights on the topic. All authors acknowledge support from ERC grant A2B (grant agreement no. 101124751).  Parts of the research were carried out while the authors visited the Simons Institute in
Berkeley.

\bibliographystyle{acm}
\bibliography{references}

\clearpage

\appendix

\section{Additional Notation}
\label{app_new_notation}
We denote by $O_{d_1\times d_2}$ the $d_1\times d_2$ zero matrix and by
$I_d$ the $d\times d$ identity matrix.
We use $x_+ \coloneqq x \vee 0$.
We denote the inner product by $\langle \bu, \bv \rangle \coloneqq \bu^{\top}\bv$.
For notational simplicity, given vectors $\bv_1 \in \mathbb{R}^{d_1}$ and $\bv_2 \in \mathbb{R}^{d_2}$ and a function $\bbf : \mathbb{R}^{d_1+d_2} \to \mathbb{R}^m$, we write $\bbf(\bv_1,\bv_2)$ for $\bbf(\operatorname{stack}(\bv_1,\bv_2))$.
For vectors $\{\bm{\psi}_j\}_{j \in I}$ and nonnegative scalars $\{a_j > 0\}_{j \in I}$,
we introduce
\begin{align}
    \ol{\sum}_{j\in I} (a_j; \bm{\psi}_j)\coloneq\sum_{j\in I} \frac{a_j}{\sum_{i\in I} a_i} \bm{\psi}_j \label{def_weighted_average}
\end{align}
a more compact notation for the weighted average of $\{\bm{\psi}_j\}_{j\in I}$ with normalized weights.

\section{Discussion about Assumption \ref{assumption_word_embedding}} \label{app_discuss_emb}

In this section, we analyze the function $\varphi_W(\cdot)$ and the quantity $\kappa_W$ from  Definition~\ref{def_pos_assumption}.
For notational simplicity, let $m \coloneqq m_2 \geq 3$ and $\delta \coloneqq \delta_2>0$.
Suppose that $\bm{\phi}_2(\be_1), \ldots, \bm{\phi}_2(\be_W)$ are independent and identically distributed according to the uniform distribution on $\mathbb{S}^{m-1}$.
Then, letting $\bZ, \bZ_1, \ldots, \bZ_W \in \mathbb{S}^{m-1}$ be samples from the same distribution, we can rewrite $\varphi_W(\epsilon)$ and $\kappa_W$ as
\[
\varphi_W(\epsilon) = 
\inf_{w \in [W]} \frac{1}{W} \Big| \big\{ {w'} \in [W]: \| \bZ_w - \bZ_{w'} \|_2 \leq \epsilon \big\}\Big| 
\]
and
\[
\kappa_W = \sup_{\bv \in \delta \mathbb{S}^{m-1}} \left\|\frac{\sum_{w=1}^W  \exp(\bv^{\top} \bZ_w) \bZ_w}{\sum_{w=1}^W  \exp(\bv^{\top} \bZ_w)} - \frac{\mathbb{E}\left[ \exp(\bv^{\top} \bZ) \bZ \right]}{\mathbb{E}\left[ \exp(\bv^{\top} \bZ) \right]}\right\|_2.
\]
It is sufficient to show that 
\begin{align}
    \mathbb{P}\Big(\varphi_W(\epsilon) \geq \epsilon^{m - 1} \log^{-1}(\epsilon^{-1})\Big) \xrightarrow[W \to \infty]{} 1
    \label{tmp_60}
\end{align}
for $W^{-\frac{1}{m-1}} \log W \leq \epsilon \ll 1$
and
\begin{align}
    \mathbb{P}\Big(\kappa_W \leq \log W / \sqrt{W}\Big) \xrightarrow[W \to \infty]{} 1.
    \label{tmp_61}
\end{align}

Because of 
\begin{align}
    \varphi_W(\epsilon) \geq 
\inf_{\bz \in \mathbb{S}^{m-1}} \frac{1}{W} \Big| \big\{ {w'} \in [W]: \| \bz - \bZ_{w'} \|_2 \leq \epsilon \big\}\Big|.
\label{tmp_32} 
\end{align} 
\eqref{tmp_60} follows if we can establish
\[
 \mathbb{P}\left( \inf_{\bz \in \mathbb{S}^{m-1}}  \Big| \big\{ {w} \in [W]: \| \bZ_{w} - \bz \|_2 \leq \epsilon \big\}\Big| \leq W \epsilon^{m - 1}  \log^{-1}(\epsilon^{-1})  \right) \xrightarrow[W \to \infty]{} 0.
\]
Let $c_m>0$ be a constant that depends only on $m$ such that
\begin{align*}
    p\coloneq\mathbb{P}( \| \bZ - \bz \|_2  \leq \epsilon/2 ) \geq c_m \epsilon^{m-1}
\end{align*} 
holds for all $\epsilon \in [0,1]$ and all
$\bz \in \mathbb{S}^{m-1}$. 
See Lemma~\ref{lemma_existence_constant} for the existence of such a constant.
In addition, 
let $\mathcal{U}$ be an $(\epsilon/2)$-net of $\mathbb{S}^{m-1}$ with respect to the metric $\|\cdot\|_2$, whose cardinality satisfies
$|\mathcal{U}| \lesssim \epsilon^{-m}.$  
Then, for any $\bz \in \mathbb{S}^{m-1}$, there exists $\bu \in \mathcal{U}$ such that $\{ {w} \in [W]: \| \bZ_{w} - \bz \|_2 \leq \epsilon\} \supseteq \{ {w} \in [W]: \| \bZ_{w} - \bu \|_2 \leq \epsilon/2\}$. The multiplicative form of Chernoff's inequality for Bernoulli random variables states that if $X\sim \Bin(W,\rho)$ and $0\leq \delta \leq 1,$ then, $\mathbb{P}(X\leq (1-\delta)W\rho)\leq e^{-\delta^2W\rho /2}.$ It follows that
\begin{align*}
    &\mathbb{P}\left( \inf_{\bz \in \mathbb{S}^{m-1}}  \Big| \big\{ {w} \in [W]: \| \bZ_{w} - \bz \|_2 \leq \epsilon \big\}\Big| \leq W \epsilon^{m - 1} \log^{-1}(\epsilon^{-1})  \right)\\
    &\leq 
    \mathbb{P}\left( \min_{\bu \in \mathcal{U}}  \Big| \big\{ {w} \in [W]: \| \bZ_{w} - \bu \|_2 \leq \epsilon/2 \big\}\Big| \leq W p c_m^{-1}\log^{-1}(\epsilon^{-1})  \right)\\
    &\leq 
    \sum_{\bu \in \mathcal{U}} \mathbb{P}\bigg( \sum_{w=1}^W  \mathbb{I}\big( \| \bZ_{w} - \bu \|_2  \leq \epsilon/2  \big) \leq \Big(1-\big(\underbrace{1 -c_m^{-1}\log^{-1}(\epsilon^{-1})}_{\eqcolon\delta} \big)\Big) W p \bigg)\\
    &\leq |\mathcal{U}| e^{-\delta^2 Wp/2}\\
    & \lesssim \epsilon^{-m} \exp\left( - \delta^2 \frac{c_m \epsilon^{m-1} W}{2}  \right) \xrightarrow[W \to \infty]{} 0,
\end{align*}
provided that $\epsilon \ll 1$, $\epsilon^{m-1} \geq W^{-1} \log^{m-1} W$ and $m \geq 3$. Combining this with \eqref{tmp_32} yields \eqref{tmp_60}.

For the second claim,
let $\mathcal{V}$ be an $W^{-1/2}$-net of $\delta \mathbb{S}^{m-1}$ with respect to the Euclidean metric, whose cardinality satisfies
$|\mathcal{V}| \lesssim (\delta \sqrt{W})^{m} \lesssim W^{m/2}.$  
Observe that for any $\bz \in \mathbb{S}^{m-1}$ and $\bv, \bu \in \delta \mathbb{S}^{m-1}$,
$$\left\|\exp(\bv^{\top} \bz) \bz - \exp(\bu^{\top} \bz) \bz \right\|_2
\leq 
\left|\exp(\bv^{\top} \bz) - \exp(\bu^{\top} \bz)\right| \leq \sup_{t \in [-\delta,\delta]} e^{t} \left|(\bv-\bu)^{\top} \bz \right| \leq e^{\delta} \|\bv-\bu\|_2.$$
Hence, for any $\bv \in \delta \mathbb{S}^{m-1}$, there exists $\tilde{\bv} \in \mathcal{V}$ such that 
$\sup_{\bz \in \mathbb{S}^{m-1}}|\exp(\bv^{\top} \bz) - \exp(\tilde{\bv}^{\top} \bz)| \leq e^{\delta}W^{-1/2}$
and
$\sup_{\bz \in \mathbb{S}^{m-1}}\|\exp(\bv^{\top} \bz)\bz - \exp(\tilde{\bv}^{\top} \bz)\bz\|_2 \leq e^{\delta} W^{-1/2}$.
It follows that for $t > 2 e^{\delta} /\sqrt{W}$,
\begin{align*}
    &\mathbb{P}\left( \sup_{\bv \in \delta \mathbb{S}^{m-1}} \left|\frac{1}{W}\sum_{w=1}^W  \exp(\bv^{\top} \bZ_w) - \mathbb{E}\left[ \exp(\bv^{\top} \bZ) \right] \right| > t \right) \\
    &\leq
    \mathbb{P}\left( \max_{\bv \in \mathcal{V}} \left|\frac{1}{W}\sum_{w=1}^W  \exp(\bv^{\top} \bZ_w) - \mathbb{E}\left[ \exp(\bv^{\top} \bZ) \right] \right| > t-2 e^{\delta} W^{-1/2} \right)\\
    &\leq \sum_{v \in \mathcal{V}} \mathbb{P}\left( \left|\frac{1}{W}\sum_{w=1}^W  \exp(\bv^{\top} \bZ_w) - \mathbb{E}\left[ \exp(\bv^{\top} \bZ) \right] \right| > t-2 e^{\delta} W^{-1/2} \right)\\
    & \leq 2 |\mathcal{V}| \exp\left(-\frac{2W(t-2e^{\delta} W^{-1/2})^2}{ e^{\delta}} \right)\\
    & \lesssim W^{m/2} \exp\left(-\frac{2W(t-2e^{\delta}/\sqrt{W})^2}{ e^{\delta}} \right).
\end{align*} 
For the third inequality, we use Hoeffding's inequality with 
$\exp(\bv^{\top}\bZ_w) \le e^{\delta}$.
For sufficiently large $W$, and by choosing $t \coloneq e^{\delta} \log W /\sqrt{W}$, we obtain
\begin{align}
    \sup_{\bv \in \delta \mathbb{S}^{m-1}} \left|\frac{1}{W}\sum_{w=1}^W  \exp(\bv^{\top} \bZ_w) - \mathbb{E}\left[ \exp(\bv^{\top} \bZ) \right]\right| = O_p\Big(\frac{\log W}{\sqrt{W}}\Big). \label{tmp_2}
\end{align}
Similarly, we derive
\begin{align}
    \sup_{\bv \in \delta \mathbb{S}^{m-1}} \left\|\frac{1}{W}\sum_{w=1}^W  \exp(\bv^{\top} \bZ_w) \bZ_w - \mathbb{E}\left[ \exp(\bv^{\top} \bZ) \bZ \right]\right\|_2 = O_p\Big(\frac{\log W}{\sqrt{W}}\Big). \label{tmp_3}
\end{align}
By Jensen’s inequality, $\mathbb{E}[ \exp(\bv^{\top} \bZ)] \geq \exp(\mathbb{E}[ \bv^{\top} \bZ]) =1$. 
Using \eqref{tmp_2} and \eqref{tmp_3}, we get the desired result
\begin{align*}
   \kappa_W 
   &= \sup_{\bv \in \delta \mathbb{S}^{m-1}} \left\|\frac{\sum_{w=1}^W  \exp(\bv^{\top} \bZ_w) \bZ_w}{\sum_{w=1}^W  \exp(\bv^{\top} \bZ_w)} - \frac{\mathbb{E}\left[ \exp(\bv^{\top} \bZ) \bZ \right]}{\mathbb{E}\left[ \exp(\bv^{\top} \bZ) \right]}\right\|_2 \\
   &\leq \sup_{\bv \in \delta \mathbb{S}^{m-1}} \left\|
   \left(\frac{1}{\frac{1}{W}\sum_{w=1}^W  \exp(\bv^{\top} \bZ_w)} - \frac{1}{\mathbb{E}\left[ \exp(\bv^{\top} \bZ) \right]} \right) \frac{1}{W} \sum_{w=1}^W  \exp(\bv^{\top} \bZ_w) \bZ_w \right\|_2 \\
   & \quad + 
   \sup_{\bv \in \delta \mathbb{S}^{m-1}} \left\|\frac{\frac{1}{W}\sum_{w=1}^W  \exp(\bv^{\top} \bZ_w) \bZ_w - \mathbb{E}\left[ \exp(\bv^{\top} \bZ) \bZ \right]}{\mathbb{E}\left[ \exp(\bv^{\top} \bZ) \right]}\right\|_2\\
   & = O_p\Big(\frac{\log W}{\sqrt{W}}\Big).
\end{align*}

\section{Proofs for the Stochastic Error}
\label{section_proof_covering_number}

In this section, we prove Theorem~\ref{thm.covering_bd_TC}.
We first establish some basic lemmas that allow us to derive a covering bound by controlling the variation of the model in the individual layers. 
The next lemma shows that a bound on the gradient or a uniform bound on the Lipschitz constant with respect to the parameters implies a bound on the covering number. Recall that $\|\bbf\|_{L^\infty(D)}\coloneq\sup_{\bx\in D} \|\bbf(\bx)\|_\infty$ with  $\|\cdot\|_\infty$ the maximum entry vector/matrix norm.

\begin{lemma}
\label{lem.gen_metric_entropy_bound}
Let $\bbf_{\btheta}=(f_{\btheta}^{(1)},\ldots,f_{\btheta}^{(q)})^\top:D\subset \mathbb{R}^d\to \mathbb{R}^q$ for all $\btheta \in [-B,B]^r.$ If 
\begin{align*}
    K\coloneq 
    \sup_{\btheta, \btheta' \in [-B,B]^r, \btheta \neq \btheta'} \, \frac{ \|\bbf_{\btheta}-\bbf_{\btheta'}\|_{L^\infty(D)}}{\|\btheta-\btheta'\|_\infty}
\end{align*}
exists and is finite, then, 
\begin{align*}
    \log \mN\Big(\delta, \big\{\bbf_{\btheta}: \btheta\in [-B,B]^r \big\}, \|\cdot\|_{L^\infty(D)}\Big)
    \leq r \log\Big(3+\frac{KB}{\delta}\Big). 
\end{align*}
Moreover, if the gradients $\nabla_{\btheta} f_{\btheta}^{(j)}(\bx)$ of the component functions $f_{\btheta}^{(j)}$ exist for all $\btheta \in [-B,B]^r, \, \bx \in D$ and are continuous in $\btheta,$ then, 
\begin{align}
    K= \sup_{\btheta \in [-B,B]^r, \, \bx \in D} \, \, \max_{j=1,\ldots,q} \, \big\|\nabla_{\btheta} f_{\btheta}^{(j)}(\bx)\big\|_1. 
    \label{eq.K_equality}
\end{align}
\end{lemma}

\begin{proof}
Whenever $\btheta, \btheta' \in [-B,B]^r, \btheta \neq \btheta',$ we have $\|\bbf_{\btheta}-\bbf_{\btheta'}\|_{L^\infty(D)}\leq K \|\btheta-\btheta'\|_\infty.$ To construct a covering, take parameters in the discrete set $\Theta\coloneq\{-B,B\}\cup\{\tfrac{2\delta}K \mathbb{Z}\cap (-B,B)\}^r.$ The length of the interval is $2B$ and thus the cardinality of this set is $\leq (3+KB/\delta)^r.$ This set provides the centers of a $\delta$-covering of the function class. 

To prove $\leq$ in \eqref{eq.K_equality}, Taylor expansion with respect to the parameter vector shows that for any $\btheta, \btheta'\in [-B,B]^r,$ there exists a $\bxi\in [-B,B]^r$ such that $f_{\btheta}^{(j)}(\bx)-f_{\btheta'}^{(j)}(\bx)=\big(\nabla_{\btheta} f_{\bxi}^{(j)}(\bx)\big)^\top \big(\btheta-\btheta')$ and therefore
$\max_{j=1,\ldots,q} \, |f_{\btheta}^{(j)}(\bx)-f_{\btheta'}^{(j)}(\bx)|\leq \sup_{\btheta \in [-B,B]^r} \, \max_{j=1,\ldots,q} \,  \|\nabla_{\btheta} f_{\btheta}^{(j)}(\bx)\|_1 \|\btheta-\btheta'\|_\infty.$

To show $\geq$ in \eqref{eq.K_equality}, consider a $\btheta\in (-B,B)^r.$ By choosing $\alpha>0$ small enough, we can ensure that $\btheta_\alpha'(j, \bx)\coloneq\btheta-\alpha \sign(\nabla_{\btheta} f_{\btheta}^{(j)}(\bx)) \in (-B,B)^r,$ where the sign  function $\sign()$ is applied componentwise. Using the definition of the gradient and $\alpha=\|\btheta-\btheta_\alpha'(j, \bx)\|_\infty,$ we have for $\alpha \downarrow 0,$
\begin{align*}
    \big\|\nabla_{\btheta} f_{\btheta}^{(j)}(\bx)\big\|_1
    = \frac 1{\alpha} \nabla_{\btheta} f_{\btheta}^{(j)}(\bx)^\top \big(\btheta-\btheta_\alpha'(j, \bx)\big)
    = \frac{f_{\btheta}^{(j)}(\bx)-f_{\btheta_\alpha'(j, \bx)}^{(j)}(\bx)}{\|\btheta-\btheta_\alpha'(j, \bx)\|_\infty}+o(1).
\end{align*}
Due to the assumed continuity of the gradients in $\btheta,$ we obtain
\begin{align*}
    \sup_{\btheta \in [-B,B]^r, \, \bx \in D} \, \, \max_{j=1,\ldots,q} \, \big\|\nabla_{\btheta} f_{\btheta}^{(j)}(\bx)\big\|_1
    = \sup_{\btheta \in (-B,B)^r, \, \bx \in D} \, \, \max_{j=1,\ldots,q} \, \big\|\nabla_{\btheta} f_{\btheta}^{(j)}(\bx)\big\|_1
    \leq K.
\end{align*}
As we already established the other inequality, we must have equality and \eqref{eq.K_equality} holds. 
\end{proof}

The next lemma provides a generic tool to control the covering number of function compositions. Similarly as in the backpropagation algorithm, it shows that the overall parameter variation of a hierarchical model can be controlled by the layerwise variation with respect to the inputs and the parameters. 
\begin{lemma}
\label{lem.comp_of_entropy}
For $q$ a positive integer, consider dimensions $r_1,\ldots, r_q,$ domains $D_0,\ldots, D_q,$ and functions $\bbf_{1, \btheta_1}, \ldots, \bbf_{q, \btheta_q}$  such that for any $\btheta_j\in [-B,B]^{r_j},$ $\bbf_{j, \btheta_j}: D_{j-1}\to D_j.$ Write
\begin{align*}
    K_j\coloneq 
    \sup_{\btheta_j, \btheta_j' \in [-B,B]^{r_j}, \btheta_j \neq \btheta'_j} \, \frac{ \|\bbf_{j,\btheta_j}-\bbf_{j,\btheta'_j}\|_{L^\infty(D_{j-1})}}{\|\btheta_j-\btheta'_j\|_\infty},
    \quad 
    Q_j\coloneq 
    \sup_{\btheta_j\in [-B,B]^{r_j}, \bx, \bx'\in D_{j-1}, \bx\neq \bx'} \, \frac{\|\bbf_{j,\btheta_j}(\bx)-\bbf_{j,\btheta_j}(\bx')\|_\infty}{\|\bx-\bx'\|_\infty},
\end{align*}
for the Lipschitz constants of the individual functions with respect to the parameters and the inputs. Then, for $\btheta=(\btheta_1,\ldots,\btheta_q)$ and $\btheta'=(\btheta'_1,\ldots,\btheta'_q),$
\begin{align*}
    \sup_{\btheta, \btheta'\in [-B,B]^{r_1+\ldots+r_q}, \, \btheta\neq \btheta'}\frac{\| \bbf_{q,\btheta_q}\circ \ldots \circ \bbf_{1,\btheta_1}
    - \bbf_{q,\btheta_q'}\circ \ldots \circ \bbf_{1,\btheta_1'}\|_{L^\infty(D_0)}}{\|\btheta-\btheta'\|_\infty}
    \leq K_q+\sum_{\ell=1}^{q-1} Q_q\ldots Q_{\ell+1}K_\ell.
\end{align*}
\end{lemma}

Combined with Lemma \ref{lem.gen_metric_entropy_bound} this gives for any $\delta>0,$
\begin{align}
\begin{split}
    &\log \, \mN\bigg(\delta, \Big\{\bbf_{q,\btheta_q}\circ \ldots \circ \bbf_{1,\btheta_1}: \btheta_q\in [-B,B]^{r_q}, \ldots, \btheta_1\in [-B,B]^{r_1}\Big\}, \|\cdot\|_{L^\infty(D_0)}\bigg) \\
    &\leq \sum_{j=1}^q r_j
    \log \bigg(3+\frac{B}{\delta}\Big(K_q+\sum_{\ell=1}^{q-1} Q_q\ldots Q_{\ell+1}K_\ell\Big) \bigg).
\end{split}
\label{eq.entropy_comp}
\end{align}
If the gradients with respect to $\btheta=(\btheta_1,\ldots, \btheta_q)$ exist and are continuous, applying the second claim of Lemma \ref{lem.gen_metric_entropy_bound} and the previous lemma provides a bound on the $\|\cdot\|_1$-norm of the gradients.

\begin{proof}
We prove the inequality by induction on $q.$ For $q=1,$ the inequality holds. The induction step $(q-1) \to q$ follows by applying the triangle inequality from
\begin{align*}
    &\big\| \bbf_{q,\btheta_q}\circ \ldots \circ \bbf_{1,\btheta_1}
    - \bbf_{q,\btheta_q'}\circ \ldots \circ \bbf_{1,\btheta_1'}\big\|_{L^\infty(D_0)} \\
    &\leq 
    \big\| \bbf_{q,\btheta_q}\circ \bbf_{q-1,\btheta_{q-1}} \circ \ldots \circ \bbf_{1,\btheta_1}
    - \bbf_{q,\btheta_q}\circ \bbf_{q-1,\btheta_{q-1}'} \circ \ldots \circ \bbf_{1,\btheta_1'} \big\|_{L^\infty(D_0)}\\
    &\quad \quad + \big\|\bbf_{q,\btheta_q}\circ \bbf_{q-1,\btheta_{q-1}'} \circ \ldots \circ \bbf_{1,\btheta_1'}- \bbf_{q,\btheta_q'}\circ \bbf_{q-1,\btheta_{q-1}'} \circ \ldots \circ \bbf_{1,\btheta_1'}\big\|_{L^\infty(D_0)} \\
    &\leq Q_q \big\| \bbf_{q-1,\btheta_{q-1}} \circ \ldots \circ \bbf_{1,\btheta_1}
    - \bbf_{q-1,\btheta_{q-1}'} \circ \ldots \circ \bbf_{1,\btheta_1'} \big\|_{L^\infty(D_0)}
    + \big\|\bbf_{q,\btheta_q}-\bbf_{q,\btheta_q'}\big\|_{L^\infty(D_{q-1})}
\end{align*}
and the induction hypothesis.
\end{proof}

\begin{lemma}
\begin{itemize}
    \item[(i)] For $A,B\geq 1$ and $d,q \leq N$,
    \begin{align}
        \sup_{\bbf_{\btheta}\in \mathcal{MLP}(d , q, L,N,B)} \, \|\bbf_{\btheta}\|_{L^\infty([-A,A]^d)}&\leq A\big(B(N+1)\big)^{L+1}, \label{eq.MLP_Linfty}\\\sup_{\bbf_{\btheta}\in \mathcal{MLP}(d , q, L,N,B), \, \bx, \bx'\in [-A,A]^d, \, \bx\neq \bx'} \,  \frac{\|\bbf_{\btheta}(\bx)-\bbf_{\btheta}(\bx')\|_\infty}{\|\bx-\bx'\|_\infty} &\leq (BN)^{L+1}, \label{eq.MLP_Lipschitz_x} \\ \sup_{\bbf_{\btheta}, \bbf_{\btheta'}\in \mathcal{MLP}(d , q, L,N,B), \, \btheta \neq \btheta'} \, \frac{ \|\bbf_{\btheta}-\bbf_{\btheta'}\|_{L^\infty([-A,A]^d)}}{\|\btheta-\btheta'\|_\infty}
        &\leq 6A \big(B(N+1)\big)^{2L+1}.
        \label{eq.MLP_Lipschitz_theta}
    \end{align}

    \item[(ii)] For $A, B \geq 1,$
    \begin{align}
        \sup_{\bbf_{\btheta}\in \mathcal{MA}(N,H,B)} \, \|\bbf_{\btheta}\|_{L^\infty([-A,A]^{N\times T}])}&\leq 2HABN, \label{eq.multihead_Linfty}\\
        \sup_{\bbf_{\btheta}\in \mathcal{MA}(N,H,B), \, X, X'\in [-A,A]^{N \times T}, \, X \neq X'} \,  \frac{\|\bbf_{\btheta}(X)-\bbf_{\btheta}(X')\|_\infty}{\|X-X'\|_\infty} &\leq 6HA^2 B^2 N^3, \label{eq.multihead_Lipschitz_x} \\ \sup_{\bbf_{\btheta}, \bbf_{\btheta'}\in \mathcal{MA}(N,H,B), \, \btheta \neq \btheta'} \, \frac{ \|\bbf_{\btheta}-\bbf_{\btheta'}\|_{L^\infty([-A,A]^{N\times T}])}}{\|\btheta-\btheta'\|_\infty}
        &\leq 3HA^3BN^3.
        \label{eq.multihead_Lipschitz_theta}
    \end{align}

    \item[(iii)] For $A,B, L_{\mathrm{mlp}}\geq 1$ and $N \leq N_{\mathrm{mlp}}$,
    \begin{align}
        \sup_{\bbf_{\btheta}\in \mathcal{TB}(N,H,L_{\mathrm{mlp}},N_{\mathrm{mlp}}, B)} \, \|\bbf_{\btheta}\|_{L^\infty([-A,A]^{N\times T}])}&\leq 2HA \big(B(N_{\mathrm{mlp}}+1)\big)^{L_{\mathrm{mlp}}+2}, \label{eq.TB_Linfty}\\\sup_{\bbf_{\btheta}\in \mathcal{TB}(N,H,L_{\mathrm{mlp}},N_{\mathrm{mlp}}, B), \, \bx, \bx'\in [-A,A]^{N\times T}, \, \bx\neq \bx'} \,  \frac{\|\bbf_{\btheta}(\bx)-\bbf_{\btheta}(\bx')\|_\infty}{\|\bx-\bx'\|_\infty} &\leq 6HA^2  \big(BN_{\mathrm{mlp}}\big)^{L_{\mathrm{mlp}}+3}, \label{eq.TB_Lipschitz_x} \\ \sup_{\bbf_{\btheta}, \bbf_{\btheta'}\in \mathcal{TB}(N,H,L_{\mathrm{mlp}},N_{\mathrm{mlp}}, B), \, \btheta \neq \btheta'} \, \frac{ \|\bbf_{\btheta}-\bbf_{\btheta'}\|_{L^\infty([-A,A]^{N\times T}])}}{\|\btheta-\btheta'\|_\infty}
        &\leq 30 H^2 A^5 \big(B(N_{\mathrm{mlp}}+1)\big)^{2L_{\mathrm{mlp}}+4}.
        \label{eq.TB_Lipschitz_theta}
    \end{align}
\end{itemize}    
\end{lemma}

\begin{proof}
\textit{(i):} We apply Lemma \ref{lem.comp_of_entropy}. Let $\btheta_1\coloneq W_0 \in [-B,B]^{N\times d},$ $\bbf_{1, \btheta_1}(\bx)\coloneq W_0\bx,$ for $j=2,\ldots, L,$ $\btheta_j\coloneq (W_{j-1}, \bv_{j-1})\in [-B,B]^{N\times (N+1)},$ $f_{j,\btheta_j}(\bx)\coloneq W_{j-1}\sigma(\bx-\bv_{j-1}),$ $\btheta_{L+1}=(W_L,\bv_L)\in [-B,B]^{q\times N + N},$ and $f_{L+1,\btheta_{L+1}}(\bx)\coloneq W_L\sigma(\bx-\bv_L).$

If $D_0=[-A,A],$ we show that to apply Lemma \ref{lem.comp_of_entropy}, one can take as domains $D_j=[-AB^j (N+1)^{j-1}N,AB^j (N+1)^{j-1}N]$ for $j=1,\ldots,L+1.$  

Since for any $\bx\in [-A,A]^d$ and $W_0\in [-B,B]^{N\times d},$ one has $\|\bbf_{1,\btheta_1}(\bx)\|_\infty=\|W_0\bx\|_\infty\leq ABd,$ the formula holds for $j=1.$ We now show the induction step from $(j-1)$ to $j.$ Suppose that $D_{j-1}=[-AB^{j-1}(N+1)^{j-2}N, AB^{j-1}(N+1)^{j-2} N].$ Using that for $j\geq 2,$ we find
\begin{align*}
    \|\bbf_{j,\btheta_j}(\bx)\|_\infty 
    = \big\|W_{j-1}\sigma(\bx-\bv_{j-1})\big\|_\infty 
    \leq BN \big\|\sigma(\bx-\bv_{j-1})\big\|_\infty
    \leq BN\big(\|\bx\|_\infty+B\big).
\end{align*}
Whenever $\|\bx\|_\infty\leq AB^{j-1}(N+1)^{j-2} N,$ using $A,B, N \geq 1,$ we therefore have $$\big\|\bbf_{j,\btheta_j}(\bx)\big\|_\infty\leq BN\big(AB^{j-1}(N+1)^{j-2}N+B\big)\leq AB^{j}N\big((N+1)^{j-2}N+B\big)\leq AB^j (N+1)^{j-1}N,$$ that is, $\bbf_{j,\btheta_j}(\bx)\in D_j.$ This establishes the induction step and proves \eqref{eq.MLP_Linfty}. 

To show that the Lipschitz constants of each layer function with respect to the inputs are bounded by $BN$, observe that whenever $W_0\in [-B,B]^{N\times d},$ $\|\bbf_{1,\btheta_1}(\bx)-\bbf_{1,\btheta_1}(\bx')\|_\infty\leq \|W_0(\bx-\bx')\|_\infty\leq Bd\|\bx-\bx'\|_\infty.$ Similarly, if all parameters are bounded by $B,$ for any $j=2,\ldots, L+1,$
\begin{align}
    \big\|\bbf_{j,\btheta_j}(\bx)-\bbf_{j,\btheta_j}(\bx')\big\|_\infty
    \leq \big\|W_{j-1}\big(\sigma(\bx-\bv_j)-\sigma(\bx'-\bv_j)\big)\big\|_\infty \leq BN \big\|\sigma(\bx-\bv_j)-\sigma(\bx'-\bv_j)\big\|_\infty \notag\\
    \leq BN\|\bx-\bx'\|_\infty. 
    \label{eq.bdsci}
\end{align}
This shows that the Lipschitz constant of the composed function is bounded by the product of the layerwise Lipschitz constants, thus, $\leq (BN)^{L+1},$ proving \eqref{eq.MLP_Lipschitz_x}. 

We finally derive a bound for the Lipschitz constants of the functions $f_{1,\btheta_1},\ldots, f_{L+1, \btheta_{L+1}}$ with respect to the parameters. Since for matrices $W, W' \in [-B,B]^{n \times m}$ and $\bv, \bv'\in [-B,B]^m,$ we have (for matrices $\|\cdot\|_\infty$ is the entrywise maximum norm)
\begin{align*}
    \big\|W\sigma(\bx-\bv)-W'\sigma(\bx-\bv')\|_\infty
    &\leq \big\|(W-W')\sigma(\bx-\bv)\big\|_\infty
    +\big\| W'\big(\sigma(\bx-\bv)-\sigma(\bx-\bv')\big)\big\|_\infty \\ &\leq m \|W-W'\|_\infty \big(\|\bx\|_\infty+B\big)
    + mB \big\| \sigma(\bx-\bv)-\sigma(\bx-\bv')\big\|_\infty \\
    &\leq m \|W-W'\|_\infty \big(\|\bx\|_\infty+B\big)
    + mB \big\|\bv-\bv'\big\|_\infty.
\end{align*}
Hence, we obtain for any $j=1,\ldots,L+1$ and any $
\bx\in [-AB^{j-1} (N+1)^{j-1},AB^{j-1} (N+1)^{j-1}],$
\begin{align*}
    \big\|\bbf_{j,\btheta_j}(\bx)-\bbf_{j,\btheta_j}(\bx')\big\|_\infty
    &\leq \Big(N\big( AB^{j-1} (N+1)^{j-1} +B\big)+ NB\Big) \big\|\btheta_j-\btheta'_j\big\|_\infty \\
    &\leq 3 A B^j (N+1)^j \big\|\btheta_j-\btheta'_j\big\|_\infty \\
    &\leq 3 A B^{L+1} (N+1)^{L+1} \big\|\btheta_j-\btheta'_j\big\|_\infty.
\end{align*}
We now apply Lemma \ref{lem.comp_of_entropy} with $q=L+1$. \eqref{eq.bdsci} shows that one can take $Q_1, \ldots, Q_{L+1}\leq BN\leq B(N+1)$ to apply Lemma \ref{lem.comp_of_entropy}. For any $y\geq 2,$ $\sum_{r=0}^L y^r\leq 2y^L.$ Since $B(N+1)\geq 2,$ this proves 
\begin{align*}
    \sup_{\bbf_{\btheta}, \bbf_{\btheta'}\in \mathcal{MLP}(d , q, L,N,B), \, \btheta \neq \btheta'} \, \frac{ \|\bbf_{\btheta}-\bbf_{\btheta'}\|_{L^\infty([-A,A]^d)}}{\|\btheta-\btheta'\|_\infty}
        &\leq 6 A \big(B(N+1)\big)^{2L+1},
\end{align*}
establishing \eqref{eq.MLP_Lipschitz_theta}.

\textit{(ii):} 
Consider 
\begin{align*}
        \operatorname{SA}_{U,V}^{(t)}(X)
        \coloneqq \operatorname{SA}_{U,V}\big(\bx_1,\ldots,\bx_T\big)_{:,t} =\sum_{s=1}^T \frac{\exp(\bx^\top_s U \bx_{t})}{\sum_{r=1}^T \exp(\bx^{\top}_r U \bx_t)} V \bx_s
\end{align*}
as a map $D=[-A,A]^{N\times T}\to \mathbb{R}^N.$ 

Since $\|V\bx_s\|_\infty\leq NB\|\bx_s\|_\infty\leq A B N,$ we obtain $\|\operatorname{SA}_{U,V}^{(t)}\|_\infty \leq A B N.$ This shows that any multi-head self-attention function $\operatorname{MA}_{f_1,\ldots,f_H}(X) = X + \sum_{h=1}^H f_h(X)$ in the class is bounded by $A + HABN,$ establishing 
\eqref{eq.multihead_Linfty}. 

We now derive \eqref{eq.multihead_Lipschitz_x}.
For $X = (\bx_t)_{t \in [T]} \in [-A,A]^{N \times T}$ and 
$X' = (\bx'_t)_{t \in [T]} \in [-A,A]^{N \times T}$, we have
\begin{align*}
    \big\|\operatorname{SA}_{U,V}^{(t)}(X) - \operatorname{SA}_{U,V}^{(t)}(X') \big\|_{\infty} 
    &\leq \Big\| \sum_{s=1}^T \frac{\exp(\bx^\top_s U \bx_{t})}{\sum_{r=1}^T \exp(\bx^{\top}_r U \bx_t)} V \bx_s - \sum_{s=1}^T \frac{\exp(\bx^\top_s U \bx_{t})}{\sum_{r=1}^T \exp(\bx^{\top}_r U \bx_t)} V \bx'_s \Big\|_{\infty}\\
& \quad +
\Big\| \sum_{s=1}^T \frac{\exp(\bx^\top_s U \bx_{t})}{\sum_{r=1}^T \exp(\bx^{\top}_r U \bx_t)} V \bx'_s 
-
\sum_{s=1}^T \frac{\exp(\bx'^\top_s U \bx'_{t})}{\sum_{r=1}^T \exp(\bx'^{\top}_r U \bx'_t)} V \bx'_s
\Big\|_{\infty}\\
& \leq \max_{s \in [T]} 
\big\| V \bx_s - V \bx'_s \big\|_{\infty}\\
& \quad + \sum_{s=1}^T \bigg|
\frac{\exp(\bx^\top_s U \bx_{t})}{\sum_{r=1}^T \exp(\bx^{\top}_r U \bx_t)}
-
\frac{\exp(\bx'^\top_s U \bx'_{t})}{\sum_{r=1}^T \exp(\bx'^{\top}_r U \bx'_t)}
\bigg| \, \max_{s \in [T]} \big\| V \bx'_s \big\|_{\infty}.
\end{align*}
Assume that $U,V \in [-B,B]^{N \times N}$ and  $A, B \geq 1$. Using 
the Lipschitz property
$\|\operatorname{softmax}(\bz)-\operatorname{softmax}(\bz')\|_{1}
\leq 2\|\bz-\bz'\|_{\infty}$ that holds for any real-valued vectors $\bz$ and $\bz'$ of the same dimension (see Lemma \ref{lemma_softmax_lip}), we obtain
\begin{align*}
    \big\|\operatorname{SA}_{U,V}^{(t)}(X) - \operatorname{SA}_{U,V}^{(t)}(X') \big\|_{\infty}
& \leq BN \big\|X - X' \big\|_{\infty}
+ 2 ABN \max_{s \in [T]} \big| \bx^\top_s U \bx_{t} - \bx'^\top_s U \bx'_{t} \big|\\
& \leq BN \big\|X - X' \big\|_{\infty}
+ 2 ABN \max_{s \in [T]} \Big(\big| \bx^\top_s U \bx_{t} - \bx'^\top_s U \bx_{t} \big|
+ \big| \bx'^\top_s U \bx_{t} - \bx'^\top_s U \bx'_{t} \big|\Big)\\
& \leq BN \big\|X - X' \big\|_{\infty} + 4ABN \big(ABN \max_{s \in [T]} \|\bx_s - \bx_s'\|_1 \big) \\
& \leq 5 A^2 B^2 N^3 \big\|X - X' \big\|_{\infty}.
\end{align*}
This shows that each of the $H$ self-attention functions has Lipschitz constant bounded by $5A^2B^2 N^3.$ 
Since $X$ has Lipschitz constant $1$ and by assumption $A, B \geq 1$, this shows that any multi-head self-attention function $\operatorname{MA}_{f_1,\ldots,f_H}(X) = X + \sum_{h=1}^H f_h(X)$ in the class has Lipschitz constant bounded by $6HA^2 B^2 N^3,$ proving \eqref{eq.multihead_Lipschitz_x}.

We now show \eqref{eq.multihead_Lipschitz_theta}. We have
\begin{align*}
    \partial_{V_{ij}}\operatorname{SA}_{U,V}^{(t)}(X)
    = \sum_{s=1}^T \frac{\exp(\bx^\top_s U \bx_{t})}{\sum_{r=1}^T \exp(\bx^{\top}_r U \bx_t)} \be_i \be_j^\top \bx_s,
\end{align*}
and thus on $X \in [-A,A]^{N\times T},$ $\|\partial_{V_{ij}}\operatorname{SA}_{U,V}(X)\|_\infty\leq A.$ Moreover,
\begin{align*}
    \partial_{U_{ij}}\operatorname{SA}_{U,V}^{(t)}(X)
    = \sum_{s=1}^T \bigg[\frac{\bx^\top_s \be_i \be_j^\top \bx_{t} \exp(\bx^\top_s U \bx_{t})}{\sum_{r=1}^T \exp(\bx^{\top}_r U \bx_t)} - \frac{\exp(\bx^\top_s U \bx_{t}) \sum_{r=1}^T \bx^\top_r \be_i \be_j^\top \bx_{t} \exp(\bx^{\top}_r U \bx_t)}{\big(\sum_{r=1}^T \exp(\bx^{\top}_r U \bx_t)\big)^2}\bigg]V \bx_s.
\end{align*}
Using that every entry of $X$ is bounded in absolute value by $A$ and any entry of $V$ is bounded in absolute value by $B,$ we find
\begin{align*}
    \big\|\partial_{U_{ij}}\operatorname{SA}_{U,V}(X)\big\|_\infty
    \leq 2A^2 \max_{s=1,\ldots,T}\big\|V\bx_s\big\|_\infty\leq 2A^2 N B A=2A^3 NB.
\end{align*}

The multi-head self-attention class consists of functions of the form $\operatorname{MA}_{f_1,\ldots,f_H}(X) = X + \sum_{h=1}^H f_h(X) \in \mathbb{R}^{N \times T}$ with self-attention layers $f_1,\ldots,f_H$. 
Let $\btheta=(V_1,U_1,\ldots,V_H,U_H)$ and denote the $(n,t)$-th entry of $\operatorname{MA}_{f_1,\ldots,f_H}(X)$ by $\operatorname{MA}_{f_1,\ldots,f_H}(X)_{n,t}.$ Then, for any $X\in [A,A]^{N\times T},$
\begin{align*}
    \big\| \nabla_{\btheta} \operatorname{MA}_{W, f_1,\ldots,f_H}(X)_{n,t}\big\|_1
    \leq \underbrace{HN^2A}_{\text{due to } V_1,\ldots,V_H}+ \underbrace{HN^2 (2A^3 NB)}_{\text{due to } U_1,\ldots,U_H}\leq 3HA^3N^3 B,
\end{align*}
proving \eqref{eq.multihead_Lipschitz_theta}.

\textit{(iii):} Recall that $\mathcal{TB}(N,H,L_{\mathrm{mlp}},N_{\mathrm{mlp}}, B) = \{ \bh_{\btheta'} \circ \bg_{\btheta}: \bg_{\btheta}  \in \mathcal{MA}(N,H,B), \, \bh_{\btheta'} \in \mathcal{MLP}(N,N,L_{\mathrm{mlp}},N_{\mathrm{mlp}},B) \}.$ 
By \eqref{eq.multihead_Linfty}, we find $$\sup_{\bg_{\btheta}\in \mathcal{MA}(N,H,B)} \, \|\bg_{\btheta}\|_{L^\infty([-A, A]^{N\times T})} \leq 2H ABN\eqcolon \ol A.$$ 
By \eqref{eq.MLP_Linfty} and using $N \leq N_{\mathrm{mlp}},$
$$\sup_{\bh_{\btheta'}\in \mathcal{MLP}(N ,N, L_{\mathrm{mlp}},N_{\mathrm{mlp}},B)} \, \|\bh_{\btheta'}\|_{L^\infty([-\ol A,\ol A]^N])}\leq \ol A \big(B(N_{\mathrm{mlp}}+1)\big)^{L_{\mathrm{mlp}}+1}\leq  2HA\big(B(N_{\mathrm{mlp}}+1)\big)^{L_{\mathrm{mlp}}+2} ,$$  establishing \eqref{eq.TB_Linfty}.

We now prove \eqref{eq.TB_Lipschitz_x}. The Lipschitz constant with respect to the input of the composition is the product of the Lipschitz constants with respect to the inputs of the inner function $\bg_{\btheta}$ and the outer function $\bh_{\btheta'}$. By \eqref{eq.multihead_Lipschitz_x}, the Lipschitz constant of the inner function is $6HA^2B^2 N^3$ and by \eqref{eq.MLP_Lipschitz_x}, the Lipschitz of the outer function is $(BN_{\mathrm{mlp}})^{L_{\mathrm{mlp}}}.$ The product is the upper bound in \eqref{eq.TB_Lipschitz_x}.

To derive \eqref{eq.TB_Lipschitz_theta}, we apply Lemma \ref{lem.comp_of_entropy} with $q=2,$ $K_1\leq 3HA^3BN^3$ (by \eqref{eq.multihead_Lipschitz_theta}), $$K_2\leq 6\ol A \big(B(N_{\mathrm{mlp}}+1)\big)^{2L_{\mathrm{mlp}}+1}\leq  12 H A \big(B(N_{\mathrm{mlp}}+1)\big)^{2L_{\mathrm{mlp}}+2}$$ (using that by assumption $N\leq N_{\mathrm{mlp}}$), and $Q_1\leq 6HA^2B^2N^3$ (by \eqref{eq.multihead_Lipschitz_x}). Because of $L_{\mathrm{mlp}}\geq 1,$ the bound \eqref{eq.TB_Lipschitz_theta} follows now from 
\begin{align*}
    K_2+ Q_1K_1
    &\leq  12 H A \big(B(N_{\mathrm{mlp}}+1)\big)^{2L_{\mathrm{mlp}}+2}
    + (6HA^2B^2N^3) (3HA^3BN^3) \\
    &\leq 30 H^2 A^5 \big(B(N_{\mathrm{mlp}}+1)\big)^{2L_{\mathrm{mlp}}+4}.
\end{align*}

\end{proof}

The class of all maps $\mF_0\coloneq\{\bbf: (\mathcal{E}_W)^{T} \to \Delta^{W-1}\}$ has $W^T$ possible inputs and maps inputs to a $W$-dimensional probability vector. If $\mN(\delta)$ denotes the covering number of $\Delta^{W-1}$ with sup-norm balls of radius $\delta,$ we have that $\mN(\delta,\mF_0,\|\cdot \|_{L^\infty((\mathcal{E}_W)^{T})})=\mN(\delta)^{W^T}$ and $\log \mN(\delta,\mF_0,\|\cdot \|_{L^\infty((\mathcal{E}_W)^{T})})=W^T\log \mN(\delta).$ 

Compared to this bound, the following log covering bound for the transformer class depends only linearly on the vocabulary size $W$ and is independent of the
document length $T$.
While transformers are still a flexible class, the comparison shows that they are a small subset within all possible maps from the $T$ tokens to probability vectors.

\begin{thm}[Theorem \ref{thm.covering_bd_TC}]
Let $A,B, L, L_{\mathrm{mlp}}\geq 1$ and $N \leq N_{\mathrm{mlp}}$. For the transformer class  introduced in Definition \ref{defi.transfomer_class} and any $\delta>0,$
    \begin{align*}
    &\log \, \mathcal{N}\Big(\delta, \mathcal{T}\big(W,T,L,N,H,L_{\mathrm{mlp}},N_{\mathrm{mlp}},B\big), \|\cdot\|_{L^\infty((\mE_W)^T)}\Big) \\
    &\leq \Big(2WN+ 4HLN^2+3LL_{\mathrm{mlp}}N_{\mathrm{mlp}}^2\Big)  \log\bigg(3+\frac{(6H)^{10 L^2} \big(B(N_{\mathrm{mlp}}+1)\big)^{38L^2 L_{\mathrm{mlp}}}B}{\delta}\bigg).
\end{align*}
\end{thm}

\begin{proof}
We aim to apply Lemma \ref{lem.comp_of_entropy}. Recall that any element in this class is of the form
\begin{align*}
    \bbf(X)= \Big( D \cdot \operatorname{TB}_L \circ \cdots \circ \operatorname{TB}_1  (P + E X) \Big)_{:,T}
\end{align*}
with parameter matrices $E, D^
\top\in [-B,B]^{N \times W},$
$P \in [-1,1]^{N \times T}$ a fixed matrix introduced in Definition \ref{def_PE}, and for each $\ell=1,\ldots, L,$
$\operatorname{TB}_\ell \in \mathcal{TB}(N,H,L_{\mathrm{mlp}},N_{\mathrm{mlp}}, B).$ 
We apply Lemma \ref{lem.comp_of_entropy} with $q=L+2,$ by interpreting the structure as a composition of $L+2$ functions, namely $P+EX, \operatorname{TB}_1, \ldots, \operatorname{TB}_L, D.$ 
We ignore the $(\,\cdot\,)_{:,T}$ operator, as it does not increase the covering number.
For all of these layers, we need to bound the Lipschitz constant with respect to the parameters. This will then provide us with bounds for the quantities $K_1,\ldots, K_{L+2}$ defined in Lemma \ref{lem.comp_of_entropy}. For all layers except for the first layer, we also need to bound the Lipschitz constant with respect to the inputs. This will then provide us with bounds for the quantities $Q_2,\ldots,Q_{L+2}$ defined in Lemma \ref{lem.comp_of_entropy}. The result follows then by applying \eqref{eq.entropy_comp}. 

Let $X\in (\mathcal{E}_W)^T \subset [0,1]^{W\times T}.$ Then, $P+EX\in [-B-1,B+1]^{N \times T}$ and by using \eqref{eq.TB_Linfty}, for any $\ell=1,\ldots,L$
\begin{align*}
    \operatorname{TB}_\ell \circ \cdots \circ \operatorname{TB}_1  (P + E X)
    : (\mathcal{E}_W)^T \to [-A_{\ell+1},A_{\ell+1}]^{N \times T}
\end{align*}
with 
\begin{align}
    A_{\ell+1}\coloneq (B+1)\Big(2H\big(B(N_{\mathrm{mlp}}+1)\big)^{L_{\mathrm{mlp}}+2}\Big)^\ell.
\end{align}

For any $X\in (\mathcal{E}_W)^T,$ and any matrices $E,E'\in [-B,B]^{N\times W},$ we have $\|P+EX-(P+E'X)\|_\infty\leq \|E-E'\|_\infty.$ Thus, $K_1\leq 1.$ 

Replacing $A$ in  \eqref{eq.TB_Lipschitz_theta} by $A_{\ell-1},$ we obtain that 
\begin{align*}
    K_{\ell}\leq 
    30 H^2 A_{L+1}^5 \big(B(N_{\mathrm{mlp}}+1)\big)^{6L_{\mathrm{mlp}}}, \quad \ell=2, \ldots, L+1,
\end{align*}
and by \eqref{eq.TB_Lipschitz_x},
\begin{align*}
    Q_{\ell}\leq 6HA_{L+1}^2  \big(BN_{\mathrm{mlp}}\big)^{4L_{\mathrm{mlp}}}, \quad \ell=2, \ldots, L+1.
\end{align*}
Whenever $D,D'\in [-B,B]^{W\times N}$ and $U\in \mathbb{R}^{N \times T},$ we obtain $\|DU-D'U\|_\infty \leq N \|U\|_\infty \|D-D'\|_\infty.$ Thus, $K_{L+2}\leq N A_{L+1}\leq N_{\mathrm{mlp}} A_{L+1}.$ 

For the Lipschitz constant of the output layer with respect to the inputs, observe that for $D\in [-B,B]^{W\times N}$ and $U, U'\in \mathbb{R}^{N \times T},$ $\|DU-DU'\|_\infty\leq N \|D\|_\infty\|U-U'\|_\infty \leq NB \|U-U'\|_\infty.$ Hence $Q_{L+2}\leq BN\leq B N_{\mathrm{mlp}}.$

Using that for any $y\geq 2,$ $\sum_{r=0}^L y^r \leq 2y^L,$ $A_{L+1}=(B+1)(2H(B(N_{\mathrm{mlp}}+1))^{L_{\mathrm{mlp}}+2})^L,$ and $L, L_{\mathrm{mlp}}\geq 1,$
\begin{align*}
    &K_{L+2}+\sum_{\ell=1}^{L+1} Q_{L+2}\ldots Q_{\ell+1}K_\ell \\
    &\leq 2\Big(6HA_{L+1}^2  \big(BN_{\mathrm{mlp}}\big)^{4L_{\mathrm{mlp}}} \Big)^L 30 H^2 A_{L+1}^5 \big(B(N_{\mathrm{mlp}}+1)\big)^{6L_{\mathrm{mlp}}}  \\
    &\leq 2(6H)^{L+2}\big(B(N_{\mathrm{mlp}}+1)\big)^{10L_{\mathrm{mlp}}L} A_{L+1}^{7L} \\
    &\leq (6H)^{10 L^2} \big(B(N_{\mathrm{mlp}}+1)\big)^{38L^2 L_{\mathrm{mlp}}}.
\end{align*}
This is an upper bound for the Lipschitz constant with respect to the parameters. To apply Lemma \ref{lem.gen_metric_entropy_bound}, the number of learnable parameters in the transformer class is shown in \eqref{eq.nr_params_TC} to be
\begin{align*}
    2WN+ L\Big((2H+1)N^2+ 2NN_{\mathrm{mlp}}+(L_{\mathrm{mlp}}-1)N_{\mathrm{mlp}}^2+L_{\mathrm{mlp}}N_{\mathrm{mlp}} + N\Big).
\end{align*}
Using that by assumption $N\leq N_{\mathrm{mlp}},$ the number of parameters can be bounded by
\begin{align*}
    \leq 2WN+ 4HLN^2+3LL_{\mathrm{mlp}}N_{\mathrm{mlp}}^2.
\end{align*}
Applying now Lemma \ref{lem.gen_metric_entropy_bound} yields the claimed bound for the metric entropy.
\end{proof}

\section{Proofs for the Approximation Error}
\label{section_proof_theorem_upper_app}

We now prove Theorem~\ref{theorem_upper_app}.
Unless otherwise specified, the implicit constants in the notation $\lesssim$ and $\tilde{O}$ depend only on $(m_1, m_2, r_1, r_2, \delta_1, \delta_2,
\beta_1, \beta_2,
\| \bg_0 \|_{\mathcal{C}^{\beta_1}}, \sup_{\bh \in \mathcal{H}} \| \bh \|_{\mathcal{C}^{\beta_2}})$.
We recall Theorem \ref{theorem_upper_app} here.

\begin{thm}[Theorem \ref{theorem_upper_app}]
      Suppose Assumptions~\ref{assumption_gh_smoothness} and \ref{assumption_transformer_size_lower} hold. For sufficiently large $W$ and $T$ and for any $\lambda \in [T^{-1},\log^{-1} (W \land T)]$, there exists a transformer $\bbf^* \in \mathcal{T}_M (W,T,L,N,H,L_{\mathrm{mlp}},N_{\mathrm{mlp}},B)$ such that
    \[
    \sup_{\bh \in \mathcal{H}} \E \Big[ \|\bbf^{*}(\bX_{1}, \ldots, \bX_{T}) - \bbf_0(\bh, \bX_{1}, \ldots, \bX_{T})\|_{\infty}^2 \, \Big| \, \bh
    \Big] 
    = \tilde{O} \bigg(N_{\mathrm{mlp}}^{-\frac{4 \beta_1}{m_1 r_1}} +  \lambda^{2 \beta_2} + \frac{M^2}{T \varphi_W(\lambda)^{r_2}} + \kappa_W^2\bigg).
    \]
\end{thm}

As prerequisites of the proof, we recall definitions and key lemmas from Section \ref{sec_approximation_sketch}. The random vectors $\bZ_{1,t} \in \mathbb{R}^{r_1 m_1}$, 
$\bZ_{2,t} \in \mathbb{R}^{r_2 m_2}$ and 
$\bY_{t} \in \mathbb{R}^{m_2}$ are defined by  
$$
\bZ_{1,t}
=
\begin{pmatrix}
\bm{\phi}_1(\bX_{t-r_1}) \\
\bm{\phi}_1(\bX_{t-r_1+1}) \\
\vdots \\
\bm{\phi}_1(\bX_{t-1})
\end{pmatrix},
\qquad
\bZ_{2,t}
=
\begin{pmatrix}
\bm{\phi}_2(\bX_{t-r_2}) \\
\bm{\phi}_2(\bX_{t-r_2+1}) \\
\vdots \\
\bm{\phi}_2(\bX_{t-1})
\end{pmatrix},
\qquad
\bY_{t} \coloneqq 
\frac{\bm{\phi}_2(\bX_t)}{\exp[\bg_0(\bZ_{1,t} )^{\top} \bm{\phi}_1(\bX_t)]}.
$$
In addition, 
recall that for $\bz \in \mathbb{R}^{r_2 m_2}$ and $K_\lambda(\bz,\bz')\coloneq \exp(-\Vert \bz - \bz' \Vert^2/(2\lambda^2))$,
\begin{align*}
    \wh{\bbm}_{\lambda}(\bz) \coloneqq
\sum_{s \in [\log^2 T, T - \log^2 T]}
    \frac{K_\lambda\big(\bZ_{2,s}, \bz\big)}{\sum_{i \in [\log^2 T, T - \log^2 T]} K_\lambda\big(\bZ_{2,i}, \bz\big) } \bY_{s}.
\end{align*}
Lemma \ref{lemma_key_NW_version2} states that
\begin{align}
    \mathbb{E}\left[\left\|\frac{\delta_2 }{\|\wh{\bbm}_{\lambda}(\bZ_{2,T+1})\|_2} \wh{\bbm}_{\lambda}(\bZ_{2,T+1})
- \bh(\bZ_{2,T}) \right\|^2_2 \, \bigg| \, \bh \right] = \tilde{O}\Big(\lambda^{2 \beta_2} + \frac{1}{T \varphi_W(\lambda)^{r_2}} + \kappa_W^2\Big),
\label{eq.Lemma6_recall}
\end{align}   
and that there exists a constant $c>0$  such that $\mathbb{P}\big( \| \wh{\bbm}_{\lambda}(\bZ_{2,T+1})  \|_2 < c \, \big| \, \bh \big)
    \lesssim 1/(T \varphi_W(\lambda)^{r_2}).$ All implicit constants and the constant $c$ depend only on $(m_1, m_2, r_1, r_2, \delta_1, \delta_2,
\beta_1, \beta_2,
\| \bg_0 \|_{\mathcal{C}^{\beta_1}}, \sup_{\bh \in \mathcal{H}} \| \bh \|_{\mathcal{C}^{\beta_2}})$.

\begin{proof}[Proof of Theorem~\ref{theorem_upper_app}]
We construct a transformer 
$\bbf^* \in \mathcal{T}_M(W,T,L,N,H,L_{\mathrm{mlp}},N_{\mathrm{mlp}},B)$ defined by  
$$
\bbf^*(X) = \operatorname{clip}_{M} \circ \, \Big(D \cdot \bm{\nu}_2 \circ \bm{\mu}_2 \circ \bm{\nu}_1 \circ \bm{\mu}_1 (P + E X)\Big)_{:,T},
$$
where $E, D^{\top} \in \mathbb{R}^{N \times W}$, $\bm{\mu}_1, \bm{\mu}_2 \in \mathcal{MA}(N,H,B)$ and $\bm{\nu}_1, \bm{\nu}_2 \in \mathcal{MLP}(N, L_{\mathrm{mlp}},N_{\mathrm{mlp}}, B)$.
For the encoding matrix
$E$,
we use
\[
    E \coloneq \left[\begin{matrix}
\bm{\phi}_1(\be_1) & \bm{\phi}_1(\be_2) & \cdots & \bm{\phi}_1(\be_W) \\
\bm{\phi}_2(\be_1) & \bm{\phi}_2(\be_2) & \cdots & \bm{\phi}_2(\be_W) \\
\bm{0}_{N-m_1 - m_2} & \bm{0}_{N-m_1 - m_2} & \cdots & \bm{0}_{N-m_1 - m_2} 
\end{matrix}\right], \quad
D \coloneq E^{\top}.
\]
Also, we use $\bm{\mu}_1, \bm{\mu}_2 \in \mathcal{MA}(N,H,B)$ and
$\bm{\nu}_1 \in \mathcal{MLP}(N, L_{\mathrm{mlp}}, N_{\mathrm{mlp}}, B)$ defined in Lemma~\ref{lemma_approx_firstpart},  such that
    \begin{align*}
        \Big( \bm{\mu}_2 \circ \bm{\nu}_1 \circ \bm{\mu}_1 \big(P + E X \big) \Big)_{:,T}
        =
            \left( \begin{matrix} 
        \bk_1\\
        \bullet_{N-m_1-m_2}\\
        \bk_2
        \end{matrix}\right)
            \begin{array}{l}
        \in \mathbb{R}^{m_1}\\
        \left. \right.\\   
        \in \mathbb{R}^{m_2}
    \end{array}
    \end{align*}
    with
    \begin{align*}
        \big\| \bk_1 - \bg_0(\bZ_{1,T+1}) \big\|_{\infty} \vee \,
        \left\| \bk_2 - \wh{\bbm}_{\lambda}(\bZ_{2,T+1}) \right\|_{\infty} 
    \lesssim (N_{\mathrm{mlp}})^{-\frac{2\beta_1}{m_1 r_1}} + \frac{1}{T}.
    \end{align*}
    Applying Lemma~\ref{lemma_normalize} with $M=\lceil \log_2 T\rceil^2$, there exists a multilayer perceptron $\bm{\nu}_2 \in \mathcal{MLP}(N, L_{\mathrm{mlp}},N_{\mathrm{mlp}}, B)$
    satisfying
    \[
        \left\|\bm{\nu}_2 
        \left( \begin{matrix} 
        \ba_1\\
        \ba_2\\
        \ba_3
        \end{matrix}\right) 
        - 
        \left( \begin{matrix} 
        \ba_1\\
        \delta_2 \ba_3 / \|\ba_3\|_2 \\
        \bm{0}_{N-m_1-m_2}
        \end{matrix}\right)
        \right\|_{\infty} \lesssim \frac{1}{T},
    \]
    for all $\ba_1 \in \mathbb{R}^{m_1}, \ba_2 \in \mathbb{R}^{N - m_1 - m_2}$, and $\ba_3 \in \mathbb{R}^{m_2}$ with $c/2 \leq \|\ba_3\|_2 \leq \exp(\delta_1 m_1) + 2$.
    Since $\|\bY_t\|_2  \leq \exp(\delta_1 m_1)$ for every $t \in [T]$, we have
    $\|\wh{\bbm}_{\lambda}(\bZ_{2,T+1})\|_2 \leq \exp(\delta_1 m_1)$ and hence
    $\|\bk_2\|_2 \leq \exp(\delta_1 m_1)+2$, provided that $N_{mlp}$ and $T$ are sufficiently large. Combining these inequalities with \eqref{eq.Lemma6_recall} and $\varphi_W(\lambda)\leq 1$, we obtain for sufficiently large $N_{\mathrm{mlp}}$ and $T$  
    \begin{align*}        &\mathbb{E}\left[\left\|\Big( \bm{\nu}_2 \circ \bm{\mu}_2 \circ \bm{\nu}_1 \circ \bm{\mu}_1 \big(P + E X \big) \Big)_{:,T} - 
        \left( \begin{matrix}
    \bg_0(\bZ_{1,T+1}) \\
    \bh(\bZ_{2,T+1})  \\
    \bm{0}_{N-m_1 - m_2}\\
    \end{matrix} \right)
    \right\|_{\infty}^2 \mathbb{I}\big( \|\wh{\bbm}_{\lambda}(\bZ_{2,T+1}) \|_{2} \geq c\big) \, \middle| \, \bh
        \right]\\
& \leq \mathbb{E}\left[\left\|\bm{\nu}_2 
        \left( \begin{matrix} 
        \bk_1\\
        \bullet_{N-m_1-m_2}\\
        \bk_2
        \end{matrix}\right) - 
        \left( \begin{matrix}
    \bg_0(\bZ_{1,T+1}) \\
    \bh(\bZ_{2,T+1}) \\
    \bm{0}_{N-m_1 - m_2}\\
    \end{matrix} \right)
    \right\|_{\infty}^2 \mathbb{I}\big( \|\bk_2 \|_{2} \geq c/2\big)
        \, \, \middle| \, \, \bh \right]\\
        &\leq 3\mathbb{E}\left[\left\|\bm{\nu}_2 
        \left( \begin{matrix} 
        \bk_1\\
        \bullet_{N-m_1-m_2}\\
        \bk_2
        \end{matrix}\right) - 
        \left( \begin{matrix} 
        \bk_1\\
        \delta_2 \bk_2 / \|\bk_2\|\\
        \bm{0}_{N-m_1-m_2}
        \end{matrix}\right)
    \right\|_{\infty}^2 \mathbb{I}\big( \|\bk_2 \|_{2} \geq c/2\big)
        \, \, \middle| \, \, \bh \right]\\
        & \quad + 3\mathbb{E}\left[\left\|\left( \begin{matrix} 
        \bk_1\\
        \delta_2 \bk_2 / \|\bk_2\|\\
        \bm{0}_{N-m_1-m_2}
        \end{matrix}\right) - 
        \left( \begin{matrix}
    \bg_0(\bZ_{1,T+1}) \\
    \delta_2 \wh{\bbm}_{\lambda}(\bZ_{2,T+1})/\|\wh{\bbm}_{\lambda}(\bZ_{2,T+1})\|_2  \\
    \bm{0}_{N-m_1 - m_2}\\
    \end{matrix} \right)
    \right\|_{\infty}^2 \mathbb{I}\big( \|\bk_2 \|_{2} \geq c/2\big)
        \, \, \middle| \, \, \bh \right]\\
    & \quad + 3\mathbb{E}\left[\left\|\left( \begin{matrix}
    \bg_0(\bZ_{1,T+1}) \\
    \delta_2 \wh{\bbm}_{\lambda}(\bZ_{2,T+1})/\|\wh{\bbm}_{\lambda}(\bZ_{2,T+1})\|_2  \\
    \bm{0}_{N-m_1 - m_2}\\
    \end{matrix} \right) - 
        \left( \begin{matrix}
    \bg_0(\bZ_{1,T+1}) \\
    \bh(\bZ_{2,T+1})  \\
    \bm{0}_{N-m_1 - m_2}\\
    \end{matrix} \right)
    \right\|_{\infty}^2
       \, \, \middle| \, \, \bh \right]\\
    & = \tilde{O}\Big((N_{\mathrm{mlp}})^{-\frac{2\beta_1}{m_1 r_1}} + \lambda^{2 \beta_2} + \frac{1}{T \varphi_W(\lambda)^{r_2}} + \kappa_W^2\Big).
    \end{align*}
By considering the decoding matrix
\[
    D \coloneq \left[\begin{matrix}
\bm{\phi}_1(\be_1) & \bm{\phi}_1(\be_2) & \cdots & \bm{\phi}_1(\be_W) \\
\bm{\phi}_2(\be_1) & \bm{\phi}_2(\be_2) & \cdots & \bm{\phi}_2(\be_W) \\
\bm{0}_{N-m_1 - m_2} & \bm{0}_{N-m_1 - m_2} & \cdots & \bm{0}_{N-m_1 - m_2} 
\end{matrix}\right]^{\top}
=
\left[\begin{matrix}
\bm{\phi}_1(\be_1)^{\top} & \bm{\phi}_2(\be_1)^{\top}  &  \bm{0}_{N-m_1 - m_2}^{\top}  \\
\bm{\phi}_1(\be_2)^{\top} & \bm{\phi}_2(\be_2)^{\top} &  \bm{0}_{N-m_1 - m_2}^{\top}  \\
\vdots & \vdots & \vdots \\
\bm{\phi}_1(\be_W)^{\top} & \bm{\phi}_2(\be_W)^{\top} &  \bm{0}_{N-m_1 - m_2}^{\top} 
\end{matrix}\right]
\]
and noticing that $M \geq \delta_1 m_1 + \delta_2 \geq \max_{w \in [W]} | \bg_0( \bZ_{1,T})^{\top} \bm{\phi}_1(\be_w)
  + \bh( \bZ_{2,T})^{\top}  \bm{\phi}_2(\be_w)| $, we obtain the assertion of Theorem~\ref{theorem_upper_app} by
\begin{align*}
    &\E \Big[ \|\bbf^{*}(\bX) - \bbf_0(\bh, \bX)\|_{\infty}^2
    \, \Big| \, \bh \Big]\\ 
    &= \E \left[ \left\|\operatorname{clip}_{M} \circ \, D \cdot \Big( \bm{\nu}_2 \circ \bm{\mu}_2 \circ \bm{\nu}_1 \circ \bm{\mu}_1 (P + E X)\Big)_{:,T} - \left( \begin{matrix}
    \bg_0( \bZ_{1,T})^{\top} \bm{\phi}_1(\be_1)
  + \bh( \bZ_{2,T})^{\top}  \bm{\phi}_2(\be_1)\\
\bg_0( \bZ_{1,T})^{\top} \bm{\phi}_1(\be_2)
  + \bh( \bZ_{2,T})^{\top}  \bm{\phi}_2(\be_2)\\
  \vdots\\
\bg_0( \bZ_{1,T})^{\top} \bm{\phi}_1(\be_W)
  + \bh( \bZ_{2,T})^{\top}  \bm{\phi}_2(\be_W)
    \end{matrix} \right)\right\|_{\infty}^2
    \, \, \middle| \, \, \bh
    \right]\\
    &= \E \left[  \left\|
    \operatorname{clip}_{M} \circ \, D \cdot \left(\Big( \bm{\nu}_2 \circ \bm{\mu}_2 \circ \bm{\nu}_1 \circ \bm{\mu}_1 (P + E X)\Big)_{:,T} -         
    \left( \begin{matrix}
    \bg_0(\bZ_{1,T+1}) \\
    \bh(\bZ_{2,T+1})  \\
    \bm{0}_{N-m_1 - m_2}\\
    \end{matrix} \right) \right)\right\|_{\infty}^2
    \, \, \middle| \, \, \bh
    \right] \\
    & \leq \E \left[  \left\|
    D \cdot \left(\Big( \bm{\nu}_2 \circ \bm{\mu}_2 \circ \bm{\nu}_1 \circ \bm{\mu}_1 (P + E X)\Big)_{:,T} -         
    \left( \begin{matrix}
    \bg_0(\bZ_{1,T+1}) \\
    \bh(\bZ_{2,T+1})  \\
    \bm{0}_{N-m_1 - m_2}\\
    \end{matrix} \right) \right)\right\|_{\infty}^2 
    \mathbb{I}\big( \|\wh{\bbm}_{\lambda}(\bZ_{2,T+1}) \|_{2} \geq c\big)
    \, \, \middle| \, \, \bh
    \right]\\
    & \quad + M^2 \P \big( \|\wh{\bbm}_{\lambda}(\bZ_{2,T+1}) \|_{2} < c \, \big| \, \bh \big)\\
    & = \E \left[  \max_w \left|
    \left( \begin{matrix}
    \bm{\phi}_1(\be_w) \\
    \bm{\phi}_2(\be_w)  \\
    \bm{0}_{N-m_1 - m_2}
    \end{matrix} \right)^{\top}
     \left(\Big( \bm{\nu}_2 \circ \bm{\mu}_2 \circ \bm{\nu}_1 \circ \bm{\mu}_1 (P + E X)\Big)_{:,T} -         
    \left( \begin{matrix}
    \bg_0(\bZ_{1,T+1}) \\
    \bh(\bZ_{2,T+1})  \\
    \bm{0}_{N-m_1 - m_2}\\
    \end{matrix} \right) \right)\right|^2 
    \mathbb{I}\big( \|\wh{\bbm}_{\lambda}(\bZ_{2,T+1}) \|_{2} \geq c\big)
    \, \, \middle| \, \, \bh \right]\\
    & \quad + M^2 \P \big( \|\wh{\bbm}_{\lambda}(\bZ_{2,T+1}) \|_{2} < c \, \big| \, \bh \big) \\
    & \leq (m_2 + 1) \E \left[   \left\|
     \left(\Big( \bm{\nu}_2 \circ \bm{\mu}_2 \circ \bm{\nu}_1 \circ \bm{\mu}_1 (P + E X)\Big)_{:,T} -         
    \left( \begin{matrix}
    \bg_0(\bZ_{1,T+1}) \\
    \bh(\bZ_{2,T+1})  \\
    \bm{0}_{N-m_1 - m_2}\\
    \end{matrix} \right) \right)\right\|_2^2 
    \mathbb{I}\big( \|\wh{\bbm}_{\lambda}(\bZ_{2,T+1}) \|_{2} \geq c\big) \, \, \middle| \, \, \bh 
    \right]\\
    & \quad + M^2 \P \big( \|\wh{\bbm}_{\lambda}(\bZ_{2,T+1}) \|_{2} < c \, \big| \, \bh \big) \\
    & = \tilde{O}\Big((N_{\mathrm{mlp}})^{-\frac{2\beta_1}{m_1 r_1}} + \lambda^{2 \beta_2} + \frac{M^2}{T \varphi_W(\lambda)^{r_2}} + \kappa_W^2\Big). 
\end{align*}
\end{proof}

\section{Construction of the Transformer}

In this section, we construct $\bm{\mu}_{1}, \bm{\mu}_{2} \in \mathcal{MA}(N,H,B)$ and $\bm{\nu}_{1}, \bm{\nu}_{2} \in \mathcal{MLP}(N, L_{\mathrm{mlp}}, N_{\mathrm{mlp}}, B)$, which are used in the proof of Theorem~\ref{theorem_upper_app}.
We construct multi-head self-attention layers in Appendix \ref{sec_pos_enc}, and multilayer perceptron layers in Appendix \ref{sec_construct_firstmlp}.
Using these results, we provide a formal proof of Lemma \ref{lemma_approx_firstpart} in Appendix~\ref{prove_lemma_approx_firstpart}.
Unless otherwise specified, the implicit constants in the notation $\lesssim$ depend only on $(m_1, m_2, r_1, r_2, \delta_1, \delta_2, \beta_1, \| \bg_0 \|_{\mathcal{C}^{\beta_1}})$.

\subsection{Construction of multi-head self-attention layers} \label{sec_pos_enc}

Recall that for $A \coloneqq N/4$, the positional encoding matrix $P \in \mathbb{R}^{N \times T}$ is defined as 
    $$P \coloneq \left[\begin{matrix}
        \bm{0}_{N-2A}& \ldots & \bm{0}_{N-2A} \\
        \bq_{1}& \ldots& \bq_{T} 
    \end{matrix}\right], \quad \bq_{t} \coloneq \begin{pmatrix}
        \bq_{1,t}\\
        \bq_{2,t}\\
        \vdots\\
        \bq_{A,t}
    \end{pmatrix},  \quad 
    \bq_{a,t} \coloneq 
    \left(
    \begin{matrix}
        \sin\left(\frac{t}{T^{a/A}}\right) \\ \cos\left(\frac{t}{T^{a/A}}\right)
    \end{matrix}
    \right).
   $$
The following lemma is key to the use of multi-scale sinusoidal positional encoding. 
Although it seems likely that this result has appeared before, we were unable to find it in the literature.
\begin{lemma} \label{lemma_posi_uplo}
    For any $t_1, t_2 \in [T]$,
    \begin{enumerate}
        \item [(i)] If $t_1 = t_2$, then $\bq_{t_1}^{\top} \bq_{t_2} = \sum_{a=1}^A \bq_{a,t_1}^{\top} \bq_{a,t_2} = A$.
        \item [(ii)] If $t_1 \neq t_2$, then $\bq_{t_1}^{\top} \bq_{t_2} = \sum_{a=1}^A \bq_{a,t_1}^{\top} \bq_{a,t_2} \leq A - \frac{1}{3 T^{2/A}}$.        
    \end{enumerate}
\end{lemma}

\begin{proof}[Proof of Lemma~\ref{lemma_posi_uplo}]
    Claim (i) follows from  $\|\bq_{a,t}\|_2 = 1$ for any $a \in [A]$ and $t \in [T]$. 

    To prove (ii), the addition theorem and the fact that $\cos(\cdot)$ is $2\pi$-periodic and an even function gives
    \begin{align*}
        \sum_{a=1}^A \bq_{a,t_1}^{\top} \bq_{a,t_2} 
        = \sum_{a=1}^A \cos\left( \frac{t_1 - t_2}{T^{a/A}} \right) = \sum_{a=1}^A \cos\left( s_a \right)
    \end{align*}
    with 
    $$s_a\coloneq\left( \frac{|t_1 - t_2|}{T^{a/A}} \right) - 2 \pi \left\lfloor \frac{1}{2 \pi} \left( \frac{|t_1 - t_2|}{T^{a/A}} \right) \right\rfloor \in [0, 2\pi), \quad a \in [A].$$
    It remains to show that there exists some $a' \in [A]$ such that $T^{-1/A} \leq s_{a'} \leq 2 \pi- T^{-1/A}$.
    Combining this with the inequality 
    $\cos(x) \leq 1 - x^2/2 + x^4 / 24 \leq 1 - x^2/3$ for $0\leq x \leq 1$,
    we obtain $\cos( s_{a'} ) \leq \cos( T^{-1/A} ) \leq 1 - 1/(3 T^{2/A})$, thereby completing the proof from
    $\sum_{a=1}^A \cos\left( s_a \right) \leq (A-1) + \cos (s_{a'}) \leq A - 1/(3 T^{2/A})$.
    
To derive a contradiction, assume that either $s_a < T^{-1/A}$ or $s_a > 2\pi - T^{-1/A}$ for all $a \in [A]$. From $|t_1 - t_2| < T$, we have 
$$s_A = \frac{|t_1 - t_2|}{T} - 2 \pi \left\lfloor \frac{1}{2 \pi} \left( \frac{|t_1 - t_2|}{T} \right) \right\rfloor = \frac{|t_1 - t_2|}{T} \in (0,1).$$
This excludes the possibility that $s_A>2\pi-T^{-1/A}$ and by assumption, we must have $s_A < T^{-1/A}.$ 
Now, assume that 
$$s_{a+1} = \frac{|t_1 - t_2|}{T^{(a+1)/A}} < T^{-1/A}, \quad \text{for some} \ a \in \{A-1,A-2,\ldots, 1\}.$$
It follows
$$\frac{|t_1 - t_2|}{T^{a/A}} < 1 < 2 \pi,$$
and hence
\begin{align*}
    s_{a} &= \frac{|t_1 - t_2|}{T^{a/A}} - 2 \pi \left\lfloor \frac{1}{2 \pi} \left( \frac{|t_1 - t_2|}{T^{a/A}} \right) \right\rfloor
    = \frac{|t_1 - t_2|}{T^{a/A}}
    < 1.
\end{align*}
From the assumption, it follows that $s_{a} < T^{-1/A}$. 
By backwards induction on $a$ from $A-1$ down to $1$,
we obtain 
$$s_{a} = \frac{|t_1 - t_2|}{T^{a/A}} < T^{-1/A}$$
for every $a \in [A]$. For $a=1,$
$$T^{-1/A}\leq \frac{|t_1 - t_2|}{T^{1/A}} =s_1<T^{-1/A},$$
we arrive at a contradiction. Therefore, the claimed result follows.
\end{proof}

\begin{rem} \label{remark_benefit_multi_sinusodial}
There is a clear benefit in choosing $A\gg 1.$ For single-scale sinusoidal positional encoding (i.e., $A=1$), the difference between two adjacent  positional encoding vectors, $\bq_t$ and $\bq_{t+1}$, scales with $O(1/T)$.
If the self-attention layer is defined using the hardmax function, that is,
$$ \Big(VX \operatorname{hardmax}\big(X^{\top} U X\big)\Big)_{:,t} = V \bx_{s^*} \ \text{ with }  \  s^* \coloneqq \argmax_{s \in [T]} \big(\bx^\top_s U \bx_{t}\big),$$
this small difference is sufficient to distinguish the two columns.  
However, when using the softmax function, this difference needs to be multiplied by extremely large values to become non-negligible.
In contrast, in Theorem~\ref{theorem_upper_app}, we employ the multi-scale sinusoidal positional encoding with $A = N/4 \gtrsim \log T$, which gives $(3T^{2/A})^{-1} \gtrsim 1$.
By applying Lemma~\ref{lemma_posi_uplo}, we therefore do not need to use such large parameters.
\end{rem}

For fixed $r$ and each $a \in [A]$, introducing the $2 \times 2$ rotation matrix
\begin{align}
M_a \coloneq \left[\begin{matrix}
    \cos\left(\frac{r}{T^{a/A}}\right) & -\sin\left(\frac{r}{T^{a/A}}\right)\\
    \sin\left(\frac{r}{T^{a/A}}\right) & \cos\left(\frac{r}{T^{a/A}}\right)    
\end{matrix}\right],
\label{eq.M_def}
\end{align}
and applying the addition theorem yields
\begin{align} \label{eq_sincos_add}
    M_a \bq_{a,t}= 
\left[\begin{matrix}
    \cos\left(\frac{r}{T^{a/A}}\right) & -\sin\left(\frac{r}{T^{a/A}}\right)\\
    \sin\left(\frac{r}{T^{a/A}}\right) & \cos\left(\frac{r}{T^{a/A}}\right)    
\end{matrix}\right]
\left(\begin{matrix}
    \sin\Big(\frac{t}{T^{a/A}}\Big)\\
    \cos\Big(\frac{t}{T^{a/A}}\Big)    
\end{matrix}\right) = \left(\begin{matrix}
    \sin\Big(\frac{t-r}{T^{a/A}}\Big)\\
    \cos\Big(\frac{t-r}{T^{a/A}}\Big)    
\end{matrix}\right)
=\bq_{a,t-r}
.
\end{align}
This is the key argument in the next lemma showing how the self-attention layer can approximately implement a delay operation with respect to its sequential input. 

\begin{lemma} \label{lemma_positional_one}
Under Assumption~\ref{assumption_transformer_size_lower}, and for any $d_1, d_2, d_3, d_4 \in \mathbb{N}_0$ with $d_1 + d_2 + d_3 = N/2$ and $d_2 + d_4 \leq N$, and for any $r \in \{0,1,\ldots,T-1\}$, 
there exists a self-attention $\operatorname{SA}_{U,V} \in \mathcal{SA}(N,B)$ such that
$$
\operatorname{SA}_{U,V} :
\left(\begin{matrix}
\bx_{1,t}  \\
\bx_{2,t}  \\
\bx_{3,t}  \\
\bq_{t} 
\end{matrix}\right)_{t \in [T]}
\begin{array}{l}
            \in [-1,1]^{d_1}\\
            \in [-1,1]^{d_2}\\
            \in [-1,1]^{d_3}\\
            \in [-1,1]^{N/2}
        \end{array}
\longmapsto \quad \left(\begin{matrix}
\bm{0}_{d_4} \\
\by_{t} \\
\bm{0}_{N - d_2 - d_4}
\end{matrix}\right)_{t \in [T]}
$$
with $\by_{t} \in \mathbb{R}^{d_2}$  
satisfying
\begin{align}
    \big\|\by_t - \bx_{2,t-r}\big\|_{\infty} \leq 2 T^{1 - \log T} \quad \text{ for } t \in \{r + 1, \ldots, T\}, \qquad \|\by_{t}\|_{\infty} \leq 1  \quad \text{ for } t \in [T]. \label{tmp_27}
\end{align}
\end{lemma}

\begin{proof}
For $t \in [T]$, we write $\bx_t \coloneq \stack(\bx_{1,t}, \bx_{2,t}, \bx_{3,t}, \bq_t)$.
It is enough to show that there exist matrices $U,V \in [-B,B]^{N \times N}$ such that for any $t \in [T]$,
$$\sum_{s=1}^T \frac{\exp(\bx^{\top}_s U \bx_{t})}{\sum_{j=1}^T \exp(\bx^{\top}_j U \bx_t)} V\bx_s = 
\left(\begin{matrix}
\bm{0}_{d_4}\\
\by_{t}\\
\bm{0}_{N -d_2 - d_4},
\end{matrix}\right)$$ 
and \eqref{tmp_27} holds.
Recall that $A \coloneq N/4$ and that $M_a$ is defined in \eqref{eq.M_def}. For the matrix $U \in \mathbb{R}^{N \times N}$, we choose 
$$
U \coloneq 
B \cdot \left[\begin{matrix}
O_{2A \times 2A} & O_{2A \times 2A}  \\[1em] 
O_{2A \times 2A} & 
\left[\begin{matrix}
M_1 & 0 & \cdots & 0 \\
0 & M_2 & \cdots & 0 \\
\vdots & \vdots & \ddots & \vdots \\
0 & 0 & \cdots & M_A
\end{matrix}\right]
\end{matrix}\right] \, ,
$$
recalling that $O_{p\times q}$ denotes the $p\times q$ zero-matrix. By \eqref{eq_sincos_add}, we have
$$
\bx_s^{\top} U \bx_t
= B \bx_s^{\top} 
\left(\begin{matrix}
\bm{0}_{2A}\\
M_1 \bq_{1,t}\\
M_2 \bq_{2,t}\\
\vdots\\
M_A \bq_{A,t}
\end{matrix}\right)
=
B \bx_s^{\top} 
\left(\begin{matrix}
\bm{0}_{2A}\\
\bq_{1,t-r}\\
\bq_{2,t-r}\\
\vdots\\
\bq_{A,t-r}
\end{matrix}\right)
= B \sum_{a=1}^A \bq_{a,s}^{\top} \bq_{a,t-r}
$$
for all $t,s \in [T]$.
For the matrix $V \in \mathbb{R}^{N \times N}$, we choose 
$$
V \coloneq \left[\begin{matrix}
    O_{d_4 \times d_1} & O_{d_4 \times d_2} & O_{d_4 \times (N-d_1-d_2)} \\
    O_{d_2 \times d_1} & I_{d_2} & O_{d_2 \times (N-d_1-d_2)}  \\ 
    O_{(N-d_2-d_4) \times d_1} & O_{(N-d_2-d_4) \times d_2} & O_{(N-d_2-d_4) \times (N-d_1-d_2)}
\end{matrix}
\right],
$$
which yields
$$V\bx_s = \left(\begin{matrix}
\bm{0}_{d_4} \\
\bx_{2,s}\\
\bm{0}_{N - d_2 - d_4}
\end{matrix}\right).$$
First, observe that
\begin{align*}
    \sum_{s=1}^T \frac{\exp(\bx^{\top}_s U \bx_{t})}{\sum_{j=1}^T \exp(\bx^{\top}_j U \bx_t)} V\bx_s = \sum_{s=1}^T \frac{\exp(\bx^{\top}_s U \bx_{t})}{\sum_{j=1}^T \exp(\bx^{\top}_j U \bx_t)} \left(\begin{matrix}
\bm{0}_{d_4} \\
\bx_{2,s}\\
\bm{0}_{N - d_2 - d_4}
\end{matrix}\right)
= 
\left(\begin{matrix}
\bm{0}_{d_4} \\
\sum_{s=1}^T \frac{\exp(\bx^{\top}_s U \bx_{t})}{\sum_{j=1}^T \exp(\bx^{\top}_j U \bx_t)}   \bx_{2,s}\\
\bm{0}_{N - d_2 - d_4}
\end{matrix}\right).
\end{align*}
For every $t \in [T]$, we have 
\begin{align*}
    \left\| \sum_{s=1}^T \frac{\exp(\bx^{\top}_s U \bx_{t})}{\sum_{j=1}^T \exp(\bx^{\top}_j U \bx_t)}   \bx_{2,s} \right\|_{\infty}
    \leq \sum_{s=1}^T \frac{\exp(\bx^{\top}_s U \bx_{t})}{\sum_{j=1}^T \exp(\bx^{\top}_j U \bx_t)} \|\bx_{2,s}\|_{\infty} \leq 1,
\end{align*}
which yields the second assertion.
Now we assume $t \in \{r+1,\ldots,T\}$ and derive the first assertion.
If $s = t-r$, we have by Lemma~\ref{lemma_posi_uplo},
\begin{align}
    \frac{\exp(\bx^{\top}_s U \bx_{t})}{\sum_{j=1}^T \exp(\bx^{\top}_j U \bx_t)} &= \frac{\exp\big(B\sum_{a=1}^A \bq_{a,s}^{\top} \bq_{a,t-r}\big)}{\sum_{j=1}^T \exp\big(B \sum_{a=1}^A \bq_{a,j}^{\top} \bq_{a,t-r}\big)} \nonumber \\
    & \geq \frac{\exp\big(B A\big)}{\exp\big(B A\big) + (T-1) \exp\Big(B \big(A - \frac{1}{3 T^{2/A}}\big) \Big)} \nonumber\\
    & \geq \frac{1}{1+T \exp\big(-B/(3 T^{2/A})\big) } \nonumber\\
    & \geq 1 - T \exp\Big(- \frac{B}{3 T^{2/A}}\Big)  
    \nonumber\\
    &\geq 1- T \exp\left( - \frac{B}{3T^{1/\log T}} \right) \nonumber\\
    &= 1- T \exp\left( - \frac{B}{3 e} \right) \nonumber\\
    &\geq 1- T \exp\left( - \log^2 T \right) \nonumber\\
    & = 1 - T^{1 - \log T}, \label{eq_lower_po}
\end{align}
provided that $A = N/4 \geq 2 \log T$ and $B \geq 9 \log^2 T > 3e \log^2 T$. Similarly, if $s \neq t-r$, we have
\begin{align}
    \frac{\exp(\bx^{\top}_s U \bx_{t})}{\sum_{j=1}^T \exp(\bx^{\top}_j U \bx_t)} = \frac{\exp\big(B \sum_{a=1}^A \bq_{a,s}^{\top} \bq_{a,t-r}\big)}{\sum_{j=1}^T \exp\big(B \sum_{a=1}^A \bq_{a,j}^{\top} \bq_{a,t-r}\big)} 
    \leq \frac{\exp\Big(B \big(A - \frac{1}{3 T^{2/A}}\big) \Big)}{ \exp(B A)}
    \leq T^{- \log T}.  \label{eq_upper_po}
\end{align}
Together, \eqref{eq_lower_po} and \eqref{eq_upper_po} yield
\begin{align*}
     \big\|\by_t - \bx_{2,t-r}\big\|_{\infty} & = \left\|\sum_{s=1}^T \frac{\exp(\bx^{\top}_s U \bx_{t})}{\sum_{j=1}^T \exp(\bx^{\top}_j U \bx_t)} \bx_{2,s} \, - \,
\bx_{2,t-r}
\right\|_{\infty}\\
&= \left\|\frac{\exp(\bx^{\top}_{t-r} U \bx_{t})}{\sum_{j=1}^T \exp(\bx^{\top}_j U \bx_t)} \bx_{2,t-r} 
\, + \, \sum_{s \neq t-r} 
\frac{\exp(\bx^{\top}_s U \bx_{t})}{\sum_{j=1}^T \exp(\bx^{\top}_j U \bx_t)} \bx_{2,s}
\, - \,
\bx_{2,t-r}
\right\|_{\infty}\\
& =  \left\|\left(\frac{\exp(\bx^{\top}_{t-r} U \bx_{t})}{\sum_{j=1}^T \exp(\bx^{\top}_j U \bx_t)} - 1\right) \bx_{2,t-r} \right\|_{\infty}
+ \sum_{s \neq t-r} \left\|  
\frac{\exp(\bx^{\top}_s U \bx_{t})}{\sum_{j=1}^T \exp(\bx^{\top}_j U \bx_t)} \bx_{2,s}
\right\|_{\infty}\\
& \leq 2 T^{1 - \log T}.
\end{align*}
The bound $\|\by_t\|_{\infty} \leq 1$ for $t \in [T]$ follows since $\by_t$ is a weighted
average of $\{\bx_{2,s}\}_{s \in [T]}$.
\end{proof}

Based on Lemma \ref{lemma_positional_one}, we can construct a multi-head self-attention that approximately implements a delay operation. We apply this lemma by setting 
$\bx_{1,t} = \bm{\phi}_1(\bX_t)$ and $\bx_{2,t} = \bm{\phi}_2(\bX_t)$. 

\begin{lemma} \label{lemma_approx_rgram}
    Under Assumption~\ref{assumption_transformer_size_lower}, there exists $\bm{\mu}_{1} \in \mathcal{MA}(N,H,B)$ such that 
    $$\bm{\mu}_{1} : \left(\begin{matrix}
\bx_{1,t} \\
\bx_{2,t} \\
\bm{0}_{\tfrac{N}{2} -m_1-m_2} \\
\bq_t 
\end{matrix}\right)_{t \in [T]}           \begin{array}{l}
            \in [-1,1]^{m_1}\\
            \in [-1,1]^{m_2}\\
            \left. \right.\\
            \in [-1,1]^{N/2}
        \end{array}
        \longmapsto \quad \left(\begin{matrix}
\bx_{1,t} \\
\bx_{2,t} \\
\by_{1,t} \\
\by_{2,t} \\
\bm{0}_{\tfrac{N}{2}- (r_1+1)m_1 -  (r_2+1)m_2}\\
\bq_t
\end{matrix}\right)_{t \in [T]}
\begin{array}{l}
            \left. \right.\\
            \left. \right.\\
            \in \mathbb{R}^{r_1 m_1}\\
            \in \mathbb{R}^{r_2 m_2}\\
            \left. \right.\\
            \left. \right.
        \end{array}
        $$
where for $a=1,2,$
\begin{align*}
    &\left\|\by_{a,t} - 
    \begin{pmatrix}
        \bx_{a,t-r_a} \\
        \bx_{a,t-r_a+1}
        \\
        \vdots\\
        \bx_{a,t-1}
    \end{pmatrix} \right\|_{\infty} \lesssim T^{1 - \log T}, \ \text{ for all } t \in \{r_a + 1, \ldots, T\}, \ \text{and} \ \  \|\by_{a,t}\|_{\infty} \leq 1,  \ \text{ for all } t \in [T].
\end{align*}
\end{lemma}

\begin{proof}
    By applying Lemma \ref{lemma_positional_one} $(r_1 + r_2)$-times and adding the identity, we obtain the assertion.
\end{proof}  

The construction of the next multi-head self-attention layer follows from the definition of the self-attention. Later, we will set 
\begin{align*}
    &\bx_{1,t} = \bg_0(\bZ_{1,t+1}), \quad \bx_{2,t} = \bY_t, \quad 
    \bx_{3,t} = \frac{1}{\lambda} \bm{\phi}_2(\bX_{t- r_2}), \quad 
    \bx_{4,t} = \frac{1}{\lambda} \stack\big(\bm{\phi}_2(\bX_{t- r_2+1}), \ldots, \bm{\phi}_2(\bX_{t- 1}) \big), \\
    &\bx_{5,t} = \frac{1}{\lambda} \bm{\phi}_2(\bX_{t}), \quad x_{6,t} = \frac{\sqrt{r_2}}{\lambda} \quad \quad x_{7,t} =  T^3 \, \mathbb{I}(t \in [\log^2 T, T - \log^2 T]).
\end{align*}

\begin{lemma} \label{lemma_represent_NW_similar}
    There exists a multi-head self-attention $\bm{\mu}_2 \in \mathcal{MA}(N,H,B)$ such that 
\begin{align*} 
    \bm{\mu}_2 : \left(
\begin{matrix}
\bx_{1,t}\\
\bx_{2,t}\\
\bx_{3,t}\\
\bx_{4,t}\\
\bx_{5,t}\\
x_{6,t}\\
x_{7,t}\\
\bm{0}_{N - m_1 - (r_2+2)m_2-2} \\
\end{matrix}\right)_{t \in [T]}
        \begin{array}{l}
            \in \mathbb{R}^{m_1}\\
            \in \mathbb{R}^{m_2}\\
            \in \mathbb{R}^{m_2}\\
            \in \mathbb{R}^{(r_2-1) m_2}\\
            \in \mathbb{R}^{m_2}\\
            \in \mathbb{R}\\
            \in \mathbb{R}\\
            \left. \right.
        \end{array}
        \longmapsto \quad
        \left(
\begin{matrix}
\bx_{1,t}\\
\bx_{2,t}\\
\bx_{3,t}\\
\bx_{4,t}\\
\bx_{5,t}\\
x_{6,t}\\
x_{7,t}\\
\bm{0}_{N - m_1 - (r_2+3)m_2-2} \\
\by_t
\end{matrix}\right)_{t \in [T]}
        \begin{array}{l}
            \left. \right.\\
            \left. \right.\\
            \left. \right.\\
            \left. \right.\\
            \left. \right.\\
            \left. \right.\\
            \left. \right.\\
            \left. \right.\\
            \in \mathbb{R}^{m_2}
        \end{array}
\end{align*}
with
\begin{align}
     \by_t \coloneqq 
     \ol{\sum}_{s \in [T]} 
     \Big(\exp \big( \big\langle \stack(\bx_{3,s}, \bx_{4,s}), \stack(\bx_{4,t}, \bx_{5,t})\big\rangle - x_{6,s} x_{6,t} + \tfrac{\lambda}{\sqrt{r_2}}x_{7,s} x_{6,t}  \big) \, ; \, \bx_{2,s}  \Big),
     \label{eq.38fgew}
\end{align}
where $\ol{\sum}$ is defined in \eqref{def_weighted_average}.
\end{lemma}

\begin{proof}
    By the definition of the self-attention in \eqref{eq_sa_t}, the $t$-th column of $\operatorname{SA}_{U,V}\big((\bx_t)_{t \in [T]}\big)$ for given matrices $U, V \in \mathbb{R}^{N \times N}$ can be written as 
    $$\sum_{s=1}^T \frac{\exp(\bx^\top_s U \bx_{t})}{\sum_{i=1}^T \exp(\bx^{\top}_i U \bx_t)} V \bx_s = \ol{\sum}_{s \in [T]} \big( \exp(\bx^\top_s U \bx_{t}) \, ; \, V \bx_s \big).$$
    Choosing matrices $U,V \in [-B,B]^N$ such that
    $$\bx_s^{\top} U \bx_t
    = \big\langle \stack(\bx_{3,s}, \bx_{4,s}), \stack(\bx_{4,t}, \bx_{5,t})\big\rangle - x_{6,s} x_{6,t} + \tfrac{\lambda}{\sqrt{r_2}}x_{7,s} x_{6,t}, \qquad V \bx_s = 
    \left( \begin{matrix}
        \bm{0}_{N-m_2}\\
        \bx_{2,s}
    \end{matrix} \right),$$
    we obtain 
    $$\operatorname{SA}_{U,V}\big((\bx_t)_{t \in [T]}\big) = 
    \left( \begin{matrix}
        \bm{0}_{N-m_2}\\
        \by_t
    \end{matrix} \right)_{t \in [T]}.$$
    By considering $f_O \in \mathcal{SA}(N,B)$ with $f_O(\cdot) = O_{N \times T}$ and 
    $\bm{\mu}_2 \coloneqq \operatorname{MA}_{\operatorname{SA}_{U,V}, f_O, f_O, \ldots, f_O} \in \mathcal{MA}(N,H,B)$,
    we get the assertion by adding the identity to $\operatorname{SA}_{U,V}$.
\end{proof}

\subsection{Construction of multilayer perceptron layers} \label{sec_construct_firstmlp}

Given input and output dimensions $d_{\mathrm{in}}, d_{\mathrm{out}} \in \mathbb{N}$ and depth and width $L_{\mathrm{mlp}}, N_{\mathrm{mlp}} \in \mathbb{N}$, let 
$p_0 \coloneqq d_{\mathrm{in}}$, $p_{\ell} \coloneqq N_{\mathrm{mlp}}$ for $\ell \in [L_{\mathrm{mlp}}]$ and $p_{L_{\mathrm{mlp}}+1} \coloneqq d_{\mathrm{out}}$. 
A multilayer perceptron is parameterized by 
$\boldsymbol{\theta} \coloneqq \{W_{\ell}, \bm{b}_{\ell}\}_{\ell \in \{0,1,\ldots,L_{\mathrm{mlp}}\}}$
with the weight matrices $\{W_{\ell} \in \mathbb{R}^{p_{\ell+1} \times p_{\ell}}\}_{\ell \in \{0,1,\ldots,L_{\mathrm{mlp}}\}}$ and the bias vectors $\{\bm{b}_{\ell} \in \mathbb{R}^{p_{\ell+1}}\}_{\ell \in \{0,1,\ldots,L_{\mathrm{mlp}}\}}$.  
For input $\bx  \in \mathbb{R}^{d_{\mathrm{in}}},$ the multilayer perceptron output $\operatorname{MLP}_{\boldsymbol{\theta}}(\bx)$ is given by
\begin{align*} \label{DNN}
\operatorname{MLP}_{\boldsymbol{\theta}}(\bx) \coloneq 
W_{L_{\mathrm{mlp}}} \bm{f}^{(L_{\mathrm{mlp}})}(\bx) - \bm{b}_{L_{\mathrm{mlp}}},
\end{align*}
where $\bm{f}^{(\ell)}$ for $\ell \in [L_{\mathrm{mlp}}]$ is recursively defined via
\begin{align*}
\bm{f}^{(\ell)}(\bx) \coloneq
\sigma\left(W_{\ell-1} \bm{f}^{(\ell-1)}(\bx) - \bm{b}_{\ell-1}\right), 
\end{align*}
and $\bm{f}^{(0)}(\bx) \coloneq \bx.$ 
Here $\sigma$ denotes the componentwise application of the ReLU activation function in the sense that for any positive integer $p,$
\begin{align*}
    \sigma
\left( (x_1, \dots, x_p)^{\top }\right)=
\big((x_{1})_+, \dots, (x_p)_+ \big)^{\top} .
\end{align*}
It is well known that multilayer perceptrons can approximate polynomial and analytic functions at exponential rates, whereas they approximate H\"older-smooth functions at polynomial rates.
The results in this section follow from slight modifications of these known results.

The following two lemmas show that multilayer perceptrons can approximate
multiplication at exponential rates. Similar results can be found in the
literature \cite{kohler2021rate, lu2021deep, 10.1214/19-AOS1875, YAROTSKY2017103}.
We adapt Lemmas A.2 and A.3 of \cite{10.1214/19-AOS1875} to general bounded
domains, while keeping the network parameters bounded.

\begin{lemma} \label{lemma_mlp_multiply}
For any $c \geq 1$ and $M = 5,6,\ldots$, there exists a multilayer perceptron $f: \mathbb{R}^2 \to \mathbb{R}$ with depth $M$, width $7$, and parameters bounded in absolute value by $4c^2$, such that
$$\sup_{x_1, x_2 \in [-c, c]} |f(x_1, x_2)- x_1 x_2| \leq c^2 2^{-M+6}.$$    
\end{lemma}

\begin{proof}
    By Lemma A.2 of \cite{10.1214/19-AOS1875}, for any $m \in \mathbb{N}$, there exists a multilayer perceptron $g : \mathbb{R}^2 \to \mathbb{R}$ with depth $m+4$, width $6$, and all parameters bounded in absolute value by $1$, such that
    $$\sup_{y_1, y_2 \in [0, 1]} |g(y_1, y_2)- y_1 y_2| \leq 2^{-m}.$$
    Let $\ell : x \mapsto (x+c)/(2c)$ and define
    \[
    f(x_1, x_2) = c^2 \Big( 4 g\big(\ell(x_1), \ell(x_2)\big) - 2 \ell(x_1) - 2 \ell(x_2) + 1\Big)
    \]
    Since $x_j=2c \, \ell(x_j)-c$ for $j=1,2$, we have, for all $x_1,x_2 \in [-c,c]$,
    \begin{align*}
        |f(x_1, x_2) - x_1 x_2|
        &=
        \Big|c^2 \Big( 4 g\big(\ell(x_1), \ell(x_2)\big) - 2 \ell(x_1) - 2 \ell(x_2) + 1\Big) - \big(2 c \ell(x_1) - c\big)\big(2 c \ell(x_2) - c\big) \Big|\\
        &= 4 c^2 \big| g\big(\ell(x_1), \ell(x_2)\big) - \ell(x_1) \ell(x_2) \big| \leq 4 c^2 2^{-m}.
    \end{align*}
    Moreover, the affine transformations before and after $g$ can be absorbed into
    the first and last layers.
    Hence $f$ can be represented by a multilayer perceptron with depth $m+4$, width $7$ and all parameters bounded in absolute value by $4c^2$, where one of the seven neurons is used to store $\ell(x_1) + \ell(x_2)$.
    Setting $m = M - 4$ proves the assertion. 
\end{proof}

\begin{lemma}\label{lemma_mlp_multiply_multi}
    For any $d = 2,3, \ldots$ and $M = 5,6, \ldots$, there exists a multilayer perceptron $f: \mathbb{R} \to \mathbb{R}$ with depth $\lceil \log_2 d \rceil M$, width $7d$, and  universally bounded parameters such that
$$\sup_{x \in [-1, 1]} \big|f(x)- x^d \big| \leq  d^2 2^{-M+6}.$$
\begin{proof}
    Let $\tilde g : \mathbb{R}^2 \to \mathbb{R}$ be the multilayer perceptron defined in Lemma~\ref{lemma_mlp_multiply} with depth $M$, width $7$ and all parameters bounded in absolute value by $4$, such that $\sup_{y_1, y_2 \in [-1, 1]} |\tilde g(y_1, y_2)- y_1 y_2| \leq 2^{-M+6}.$
    To ensure that the output lies in the range $[-1,1]$, define 
    \begin{align*}
        g(x_1, x_2) \coloneqq \big( \tilde{g}(x_1, x_2) + 1 \big)_+ - \big( \tilde{g}(x_1, x_2) - 1 \big)_+ - 1 = \begin{cases}
            -1, & \quad \tilde{g}(x_1, x_2) < -1 \\
            \tilde{g}(x_1, x_2), &  \quad \tilde{g}(x_1, x_2) \in [-1,1] \\ 
            1, &  \quad \tilde{g}(x_1, x_2) > 1.
        \end{cases}
    \end{align*}
    Then,     
    $g : \mathbb{R}^2 \to [-1,1]$ is a multilayer perceptron with depth $M+1$, width $7$ and all parameters bounded in absolute value by $4$, satisfying
        \begin{align} 
        \sup_{y_1, y_2 \in [-1, 1]} | g(y_1, y_2)- y_1 y_2| \leq 2^{-M+6}.\label{tmp_62}
    \end{align}    

    We construct the multilayer perceptron $f$ as follows.
    Let $q \coloneqq \lceil \log_2 d \rceil$ and define
    \[
    \bm{\ell} : x \mapsto (\underbrace{x, \ldots, x}_{d}, \underbrace{1, \ldots, 1}_{2^q - d})^{\top}.
    \]  
    Let $\ell(x) = (x_1, \ldots, x_{2^q})^{\top}$.
    Starting from $(\ell(x)$, 
    we apply $g$ to the $2^{q-1}$ pairs $(x_1, x_2), \ldots, (x_{2^q - 1}, x_{2^q})$ to compute  $g(x_1, x_2), \ldots, g(x_{2^q - 1}, x_{2^q})$. 
    Then, we pair neighboring entries and apply $g$ again to the resulting $2^{q-2}$ pairs.
    This procedure is continued until only one entry remains.
    Since affine transformations before and after a multilayer perceptron can
    be absorbed into the first and last layers, the resulting network has depth
    $qM$, width $7\cdot 2^{q-1} \leq 7d$, and all parameters are bounded by a
    constant independent of $d$ and $M$.

    If $y_1,y_2,\tilde y_1,\tilde y_2 \in [-1,1]$, then by \eqref{tmp_62}
    and the triangle inequality,
    $| g(\tilde y_1,\tilde y_2)-y_1y_2 |
    \leq
    2^{-M+6}
    +
    |\tilde y_1-y_1|
    +
    |\tilde y_2-y_2|.$
    By induction on the number of iterated multiplication steps $q$, we obtain
    \[
    \left|
    f(x)- x^d
    \right|
    \leq
    3^{q-1}2^{-M+6}
    \leq
    3^{\log_2 d} 2^{-M+6}
    \leq
    d^2 2^{-M+6}
    \]
    for any $x \in [-1,1]$.
\end{proof}
\end{lemma}

The next lemma states that multilayer perceptrons can approximate analytic functions at exponential rates.
We adapt Theorem 6 of \cite{wang2018exponential} to the case of bounded parameters.
Compared with their result, our result achieves the same error rate using a more shallow but wider architecture.

\begin{lemma} \label{lemma_analytic}
Assume that the power series $f_0(x) = \sum_{k=0}^\infty a_k x^k$ is absolutely convergent in $[-1,1]$ and $A \coloneqq \sum_{k =0}^\infty |a_k| < \infty$.
Given fixed $\delta \in (0,1)$, for any sufficiently large $M \in \mathbb{N}$, there exists a multilayer perceptron $f: \mathbb{R} \to \mathbb{R}$ with depth $\lceil \log_2(M/\delta) \rceil M$, width $7 \lceil M/\delta \rceil^2$, and parameters bounded in absolute value by a constant depending only on $A$, such that
$$\sup_{x \in [-1+\delta, 1-\delta]}\, \big|f(x)-f_0(x)\big| \leq \frac{A (M+1)^2 2^{-M+6}}{\delta^2}.$$
\end{lemma}

\begin{proof}
    For given $q \in \mathbb{N}$, define
    \[
    P_q(x) \coloneqq \sum_{k=0}^{q-1} a_k x^k, \quad 
    R_q(x) \coloneqq \sum_{k=q}^{\infty} a_k x^k.
    \]
    For $x \in [-1+\delta, 1-\delta]$, we have
    $|R_q(x)| \leq (1-\delta)^q \sum_{k =q}^\infty |a_k| \leq A(1-\delta)^q$ and hence
    \begin{align}
        \sup_{x \in [-1+\delta, 1-\delta]} |P_q(x) - f_0(x)| \leq A(1-\delta)^q. \label{tmp_63}
    \end{align}
    In the following proof, we approximate $P_q(x)$, with $q$ to be chosen later.

For each $d=2,\ldots,q-1$, we apply Lemma~\ref{lemma_mlp_multiply_multi}
to approximate $x^d$.
By running these subnetworks in parallel, together with the outputs $1$ and $x$,
we construct a multilayer perceptron $\bg : \mathbb{R} \to \mathbb{R}^q$ with depth $\lceil \log_2(q-1) \rceil M$, width $1 + \sum_{d=2}^{q-1} 7d \leq 7q^2$ and bounded parameters such that    
    \[
        \sup_{x \in [-1, 1]} \Big\|\bg(x)- (x^{\alpha})_{\alpha \in \{0,1,\ldots,q-1\}}\Big\|_{\infty} \leq (q-1)^2 2^{-M+6}.
    \]
Define $f: \mathbb{R} \to \mathbb{R}$ by
\begin{align*}
    f(x) \coloneqq \big\langle (a_0, a_1, \ldots, a_{q-1})^{\top},  \bg(x) \big\rangle.
\end{align*}
Then $f$ can be represented by a multilayer perceptron with depth $\lceil \log_2(q-1) \rceil M$, width $7q^2$ and parameters bounded in absolute value by a constant depending only on $A$.
Moreover, for any $x \in [-1, 1]$,
\begin{align*}
    \big|f(x)- P_q(x)\big| \leq \sum_{k=0}^{q-1} |a_k| \big| \big(\bg \circ \bm{\ell}(x)\big)_{(k+1)} - x^k  \big|  \leq A (q-1)^2 2^{-M+6}.
\end{align*}
Combining with \eqref{tmp_63}, we have
\begin{align*}
    \sup_{x \in [-1+\delta, 1-\delta]}\, \big|f(x)-f_0(x)\big| \leq A (1 - \delta)^q + A (q-1)^2 2^{-M+6}. 
\end{align*}
Choosing $q = \lceil M/\delta \rceil$,
we obtain
\[
A(1-\delta)^q + A (q-1)^2 2^{-M+6} 
\leq A e^{-M} + \frac{A M^2 2^{-M+6}}{\delta^2} 
\leq \frac{A (M+1)^2 2^{-M+6}}{\delta^2}.
\]    
This proves the assertion.
\end{proof}

\begin{lemma}[Theorem 1 of \cite{kong2025posterior}]
\label{lemma_approx_smooth}
Given a $\beta$-H\"older smooth function $\bbf_0 : [-c,c]^{d_1} \to \mathbb{R}^{d_2}$, for any sufficiently large $M \in \mathbb{N}$, there exists a multilayer perceptron $\bbf: \mathbb{R}^{d_1} \to \mathbb{R}^{d_2}$
with depth $\lceil \log_2 M \rceil^2$, width $M$, and parameters bounded in absolute value by a constant depending only on $(c, d_1, \beta, \|\bbf_0\|_{\mathcal{C}^{\beta}})$ such that
$$\sup_{\bx \in [-c, c]^d} \|\bbf(\bx)-\bbf_0(\bx)\|_{\infty} \lesssim M^{-\frac{2 \beta}{d_1}},$$
where the implicit constant depends only on $(c, d_1, \beta, \|\bbf_0\|_{\mathcal{C}^{\beta}})$.
\end{lemma}

\begin{lemma} \label{lemma_mlp_approx_exp}
    Given a fixed $c > 0$, for any sufficiently large $M$, there exists a multilayer perceptron $f: \mathbb{R} \to \mathbb{R}$
with depth $M^2$, width $33 M^2$, and parameters bounded in absolute value by a constant
depending only on $c$ such that     
    $$\sup_{x \in [-c, c]}|f(x)-\exp(x)| \lesssim M^2 2^{-M},$$
where the implicit constant depends only on $c$.
\end{lemma}
\begin{proof}
Without loss of generality, assume that $c \in \mathbb{N}$. Otherwise, we replace $c$ by $\lceil c \rceil$ and apply the same argument.
    Define 
    $$f_1 : x \mapsto \frac{x}{2c}.$$
    By Lemma~\ref{lemma_analytic}, there exists a multilayer perceptron $f_2:\mathbb{R} \to \mathbb{R}$ with depth $\lceil \log_2(2M) \rceil M$, width $32 M^2$, and bounded parameters such that   
    $$\sup_{x \in [-1/2, 1/2]}|f_2(x)-\exp(x)| \lesssim M^2 2^{-M}.$$
    Also, by Lemma \ref{lemma_mlp_multiply_multi}, there exists a multilayer perceptron $f_3:\mathbb{R} \to \mathbb{R}$ with depth $\lceil \log_2(2c) \rceil M$, width $16c$ and bounded parameters such that   
    $$\sup_{x \in [0, 1]} |f_3(x) - x^{2c}| \lesssim 2^{-M}.$$  
    We obtain the assertion from the fact that 
    \[
        f(x) \coloneq e^{2c} \, f_3 \Big(\frac{f_2 \circ f_1(x)}{e}\Big)
    \]
    can be represented as a multilayer perceptron with depth $\lceil \log_2(2M) \rceil M + \lceil \log_2(2c) \rceil M  \leq M^2$, width $32 M^2 + 16c \leq 33 M^2$, and parameters bounded in absolute value by a constant depending only on $c$, and 
    \begin{align*}
        \sup_{x \in [-c,c]} \big|f(x) - \exp(x)\big|
        &= \sup_{x \in [-1/2, 1/2]} \Big|e^{2c} \, f_3 \Big(\frac{f_2(x)}{e}\Big)-  \exp(x)^{2c}\Big|\\
        &\leq e^{2c} \sup_{x \in [-1/2, 1/2]} \Big| f_3 \Big(\frac{f_2(x)}{e}\Big)- \Big(\frac{f_2(x)}{e}\Big)^{2c}\Big| 
        + \sup_{x \in [-1/2, 1/2]} \big|  f_2(x)^{2c} - \exp(x)^{2c}\big|\\ &\lesssim M^2 2^{-M}.
    \end{align*}
\end{proof}

\begin{lemma} \label{lemma_1oversqrtx}
    Given $0<a<b$, there exist constants $c=c(a,b), c'=c'(a,b),$ such that for any sufficiently large $M$, there exists a multilayer perceptron $f: \mathbb{R} \to \mathbb{R}$
with depth $\lceil \log_2(Mb/a) \rceil M$, width $7 \lceil Mb/a \rceil^2$, and parameters bounded in absolute value by $c,$ satisfying
    $$\sup_{x \in [a, b]}\left|f(x)-\frac{1}{\sqrt{x}}\right| \leq c' M^2 2^{-M}.$$
\end{lemma}

\begin{proof} 
    We rewrite
    $$\frac{1}{\sqrt{x}} = h_0 \circ g_0(x), \text{ with } g_0(x)\coloneq \frac{2x - a - b}{b} \text{ and } h_0(z) \coloneq \frac{\sqrt 2}{\sqrt{a+b}} \, \frac{1}{\sqrt{1+\frac{z}{1+(a/b)}}},$$  
    where $g_0$ maps $[a,b]$ to $[-1 + (a/b), 1 - (a/b)]$.    
    By the generalized binomial theorem, 
    $$h_0(z) = \frac{\sqrt 2}{\sqrt{a+b}} \sum_{k \in \mathbb{N}_0} \binom{-\frac{1}{2}}{k} \left(\frac{z}{1+(a/b)}\right)^k \quad \text{for all} \  z \in \big( -1 - (a/b), 1 + (a/b)\big)$$
with generalized binomial coefficients
\[
\binom{-\frac{1}{2}}{k}
= \frac{(-\frac{1}{2}) (-\frac{1}{2} -1) \ldots (-\frac{1}{2} - k + 1) }{k!} =
(-1)^k \frac{(2k)!}{4^k (k!)^2}.
\]
Observe that
    \begin{align*}
        \sum_{k \in \mathbb{N}_0} \left|\binom{-\frac{1}{2}}{k} \left(\frac{1}{1+(a/b)}\right)^k\right| 
        = \sum_{k \in \mathbb{N}_0} \binom{-\frac{1}{2}}{k} \left(-\frac{1}{1+(a/b)}\right)^k  
        = \frac{1}{\sqrt{1-\frac{1}{1+(a/b)}}} = \sqrt{1+(b/a)}.
    \end{align*}
    By Lemma \ref{lemma_analytic}, there exists a multilayer perceptron $h : \mathbb{R} \to \mathbb{R}$ with depth
    $\lceil \log_2(Mb/a) \rceil M$, width $7 \lceil Mb/a \rceil^2$, and parameters bounded in absolute value by a constant depending only on $a$ and $b$ such that
    $$\sup_{z \in [-1+(a/b), 1-(a/b)]}|h(z)-h_0(z)| \lesssim M^2 2^{-M}.$$
    It follows that there exists a multilayer perceptron $f = h \circ g_0 : \mathbb{R} \to \mathbb{R}$ with the same architecture such that 
    $$\sup_{x \in [a, b]}\left|f(x)-\frac{1}{\sqrt{x}}\right| \lesssim M^2 2^{-M}.$$
\end{proof}

\begin{lemma}\label{lemma_mlp_scaleup}
    For any $d, c, L, M \in \mathbb{N}$ with $c \leq (MB)^L$, there exists $\bbf \in \mathcal{MLP}(d, d, L, 2dM, B)$ such that  
    $\bbf(\bx) = c \bx$
    for every $\bx \in \mathbb{R}^d$.
\end{lemma}
\begin{proof}
    It suffices to consider the case $d=1,$ as then, for arbitrary $d$, one can construct $d$ networks computing $x \mapsto c x$ and concatenate them into one large network computing $\bx \mapsto \bbf(\bx) = c \bx$.
        
    The first hidden layer consists of $2M$ neurons, with $M$ neurons computing
    $(Bx)_+ = Bx \, \mathbb{I}(x \ge 0)$ and the remaining $M$ neurons computing
    $(-Bx)_+ = -Bx \, \mathbb{I}(x \le 0)$.
    For each $\ell \in [L-1]$, suppose that the $\ell$-th hidden layer consists of
    $M$ neurons computing
    $M^{\ell-1} B^{\ell} x \, \mathbb{I}(x \ge 0)$ and the remaining $M$ neurons
    computing
    $-M^{\ell-1} B^{\ell} x \, \mathbb{I}(x \le 0)$.
    Then, in the $(\ell+1)$-st hidden layer, we construct $M$ neurons computing
    $$
    \bigl(MB \cdot M^{\ell-1} B^{\ell} x \, \mathbb{I}(x \ge 0)\bigr)_+
    = M^{\ell} B^{\ell+1} x \, \mathbb{I}(x \ge 0),
    $$
    and the remaining $M$ neurons computing
    $$
    \bigl(MB \cdot (-M^{\ell-1} B^{\ell} x \, \mathbb{I}(x \le 0))\bigr)_+
    = -M^{\ell} B^{\ell+1} x \, \mathbb{I}(x \le 0).
    $$
    By induction on $\ell$ from $1$ to $L-1$, the last hidden layer consists of
    $2M$ neurons, with $M$ neurons computing
    $M^{L-1} B^{L} x \, \mathbb{I}(x \ge 0)$ and the remaining $M$ neurons
    computing
    $-M^{L-1} B^{L} x \, \mathbb{I}(x \le 0)$.
    We obtain the assertion from 
    $$M \cdot \frac{c}{(MB)^L} \cdot M^{L-1} B^{L} x \, \mathbb{I}(x \ge 0) + M \cdot \frac{-c}{(MB)^L} \cdot \big( -M^{L-1} B^{L} x \, \mathbb{I}(x \le 0) \big) = c x.$$    
\end{proof}

\begin{lemma}\label{lemma_pos_mlp}  
    Under Assumption~\ref{assumption_transformer_size_lower} and any $M \geq 1 + 2\log^2 T$, $J \geq 2$ and $c \leq (MB)^{J-2}$,
    there exists $f \in \mathcal{MLP}(\tfrac{N}{2},1,J,M,B)$ such that for every $t \in [T]$,
    $$f(\bq_t) = c \, \mathbb{I}\big(t \in [\log^2 T, T-\log^2 T]\big).$$
\end{lemma}
\begin{proof}  
    By Assumption~\ref{assumption_transformer_size_lower}, we have
    $
    3T^{8/N} \leq 3T^{1/\log T} \leq B.
    $
    By Lemma~\ref{lemma_posi_uplo} with $A=N/4$ as in
    Definition~\ref{def_PE}, for every $t_1,t_2 \in [T]$,
    \begin{align*}
        \left(\bq_{t_1}^{\top} \bq_{t_2} - \frac{N}{4} + \frac{1}{B}\right)_+ = 
        \frac{1}{B} \, \mathbb{I}(t_1 = t_2),
    \end{align*}
    where we used $(- (3 T^{1/\log T})^{-1} + B^{-1})_+ = 0$ when $t_1 \neq t_2$.
    Summing up, we get
    \begin{align*}
        \sum_{s \in [1, \log^2 T)} \left(\bq_{s}^{\top} \bq_{t} - \frac{N}{4} + \frac{1}{B}\right)_+ +
        \sum_{s \in (T - \log^2 T, T]}  \left(\bq_{s}^{\top} \bq_{t} - \frac{N}{4} + \frac{1}{B}\right)_+
        = \frac{1}{B} \, \mathbb{I}(t \notin [\log^2 T, T-\log^2 T]).
    \end{align*}
    In the first hidden layer, 
    we construct $ \leq 1 + 2 \log^2 T$ neurons computing
    $$\bq_{t}\mapsto \left\{ \left(\bq_{s}^{\top} \bq_{t} - \frac{N}{4} + \frac{1}{B}\right)_+ \right\}_{s \in [T] \setminus [\log^2 T, T - \log^2 T]}.$$
    In the second hidden layer, we construct one neuron computing
    \begin{align*}
        \bq_{t}\mapsto &\left(1 
        - \sum_{s \in [1, \log^2 T)} B\left(\bq_{s}^{\top} \bq_{t} - \frac{N}{4} + \frac{1}{B}\right)_+ -
        \sum_{s \in (T - \log^2 T, T]}  B \left(\bq_{s}^{\top} \bq_{t} - \frac{N}{4} + \frac{1}{B}\right)_+\right)_+\\
        &= \mathbb{I}(t \in [\log^2 T, T-\log^2 T]).
    \end{align*}
    By applying Lemma~\ref{lemma_mlp_scaleup} to the remaining $J-2$ hidden layers, we obtain the desired network output.
\end{proof}

Recall that, given vectors $\bv_1 \in \mathbb{R}^{d_1}$ and
$\bv_2 \in \mathbb{R}^{d_2}$ and a function
$\bbf : \mathbb{R}^{d_1+d_2} \to \mathbb{R}^m$, we write
$\bbf(\bv_1,\bv_2)$ for $\bbf(\operatorname{stack}(\bv_1,\bv_2))$.

\begin{lemma} \label{lemma_construct_firstmlp}
    Suppose Assumption~\ref{assumption_transformer_size_lower} holds.
    For sufficiently large $T$ and for any $\lambda \geq T^{-1}$,
    there exists a multi-head self-attention $\bm{\nu}_1 \in \mathcal{MLP}(N, N, L_{\mathrm{mlp}},N_{\mathrm{mlp}}, B)$ such that for any $t \in [T]$,
\begin{align*}
\bm{\nu}_1 : \left( \begin{matrix}
    \bx_1 \\
    \bx_2 \\
    \bx_3 \\
    \bx_4 \\
    \bx_5\\
    \bm{0}_{\tfrac{N}{2}- (r_1+1)m_1 -  (r_2+1)m_2}\\
    \bq_t
    \end{matrix}\right)
    \begin{array}{l}
            \in [-1,1]^{m_1}\\
            \in [-1,1]^{m_2}\\
            \in [-1,1]^{m_1}\\
            \in [-1,1]^{(r_1 - 1)m_1}\\
            \in [-1,1]^{r_2 m_2}\\
            \left. \right. \\        
            \in [-1,1]^{N/2}
        \end{array}
        \longmapsto \quad  
        \left(\begin{matrix}
            \by_1 \\
            \by_2 \\
            \frac{1}{\lambda} \bx_{5} \\
            \frac{1}{\lambda} \bx_{2} \\
            \frac{\sqrt{r_2}}{\lambda} \\
            T^3 \, \mathbb{I}(t \in [\log^2 T, T - \log^2 T])\\
        \bm{0}_{N - m_1 - (r_2+2)m_2-2}
        \end{matrix} \right)
    \begin{array}{l}
            \in \mathbb{R}^{m_1}\\
            \in \mathbb{R}^{m_2}\\
            \left. \right.\\
            \left. \right.\\
            \in \mathbb{R}\\
            \in \mathbb{R}\\
            \left. \right. 
        \end{array}
        ,
\end{align*}
with
\begin{align*}
    \big\|\by_{1} - \bg_0(\bx_4, \bx_1) \big\|_{\infty} \vee 
    \left\|\by_{2} - \exp\big( -  \bx_1^{\top} \, \bg_0(\bx_3, \bx_4)    \big) \bx_2 \right\|_{\infty} \lesssim (N_{\mathrm{mlp}})^{-\frac{2\beta_1}{m_1 r_1}}.
\end{align*}
\end{lemma}

\begin{proof}
    To construct $\by_1, \by_2$, we construct  $\bm{\Phi}_{1}, \bm{\Phi}_{2,1} \in \mathcal{MLP}(m_1 r_1, m_1, \lceil \log_2 N_{\mathrm{mlp}} \rceil^2, \lceil N_{\mathrm{mlp}}/4 \rceil, B)$ using Lemma~\ref{lemma_approx_smooth} such that
    for any $\bx_4 \in [-1,1]^{(r_1 - 1)m_1}$, any $\bx_1 \in [-1,1]^{m_1}$, and any $\bx_3 \in [-1,1]^{m_1}$,   
    \begin{align}
        \big\| \bm{\Phi}_{1}(\bx_4, \bx_1) - \bg_0(\bx_4, \bx_1) \big\|_{\infty} \lesssim (N_{\mathrm{mlp}})^{-\frac{2\beta_1}{m_1 r_1}},  \label{eq_nu_1}
    \end{align}
    and    $$\big\|\bm{\Phi}_{2,1}(\bx_3, \bx_4) - \bg_0(\bx_3, \bx_4) \big\|_{\infty} \lesssim (N_{\mathrm{mlp}})^{-\frac{2\beta_1}{m_1 r_1}}.$$
    Since $\|\bg_0(\bx_3, \bx_4)\|_{\infty} \leq \delta_1$, we have 
    $\|\bm{\Phi}_{2,1}(\bx_3, \bx_4)\|_{\infty} \leq c_1 \coloneq 1+\delta_1$.   
    By applying Lemma \ref{lemma_mlp_multiply} $m_1$-times and summing the results, we construct $\bm{\Phi}_{2,2} \in \mathcal{MLP}(2 m_1 , 1, \lceil \log_2 N_{\mathrm{mlp}} \rceil^2, 7 m_1, B)$ such that for any $\bx_1 \in [-1,1]^{m_1}$ and $\by \in [-c_1, c_1]^{m_1}$,
    $$
    \big|\bm{\Phi}_{2,2}(\bx_1, \by) - \bx_1^{\top} \by \big| \lesssim (N_{\mathrm{mlp}})^{-\log_2  N_{\mathrm{mlp}}}.$$
    From $|\bx_1^{\top} \by| \leq m_1 c_1$,
    we have $|\bm{\Phi}_{2,2}\big(\bx_1, \bm{\Phi}_{2,1}(\bx_3, \bx_4)\big)| \leq c_2 \coloneq  1 + m_1 c_1$.
    By applying Lemma \ref{lemma_mlp_approx_exp},
    we construct 
    $\bm{\Phi}_{2,3} \in \mathcal{MLP}(1, 1, \lceil 2 \log_2 N_{\mathrm{mlp}}\rceil^2, 33 \lceil 2 \log_2 N_{\mathrm{mlp}}\rceil^2, B)$   
    such that for any $y \in [-c_2, c_2]$,
    $$\Big|\bm{\Phi}_{2,3}(y) - \exp(-y) \Big| \lesssim 
    \lceil 2 \log_2 N_{\mathrm{mlp}}\rceil^4
    2^{- \lceil 2 \log_2 N_{\mathrm{mlp}}\rceil^2}
    \lesssim
    \lceil 2 \log_2 N_{\mathrm{mlp}}\rceil^4 (N_{\mathrm{mlp}})^{- 4 \log_2  N_{\mathrm{mlp}}}
    \lesssim (N_{\mathrm{mlp}})^{- \log_2  N_{\mathrm{mlp}}}.$$
    We have $|\bm{\Phi}_{2,3} \circ \bm{\Phi}_{2,2}\big(\bx_1, \bm{\Phi}_{2,1}(\bx_3, \bx_4)\big)| \leq c_3 \coloneq  1 + \exp(c_2)$.
    By applying Lemma \ref{lemma_mlp_multiply} $m_2$-times and parallelizing the results,
    we construct $\bm{\Phi}_{2,4} \in \mathcal{MLP}(m_2+1 , m_2, \lceil \log_2 N_{\mathrm{mlp}}\rceil^2, 7 m_2, B)$ such that for any $y \in [-c_3, c_3]$ and $\bx_2 \in [-1,1]^{m_2}$,
    $$\big\|\bm{\Phi}_{2,4}(y, \bx_2) - y \bx_2 \big\|_{\infty} \lesssim (N_{\mathrm{mlp}})^{-\log_2  N_{\mathrm{mlp}}}.$$
    To sum up, we can construct 
    $\bm{\Phi}_{2} \in \mathcal{MLP}((r_1+1)m_1 + m_2, m_2, 7 \lceil \log_2 N_{\mathrm{mlp}}\rceil^2, \lceil N_{\mathrm{mlp}}/2 \rceil, B)$ with
    $$\bm{\Phi}_{2}(\bx_1, \bx_2, \bx_3, \bx_4) \coloneq \bm{\Phi}_{2,4}\Big(  \bm{\Phi}_{2,3} \circ \bm{\Phi}_{2,2}\big( \bx_1, \bm{\Phi}_{2,1}(\bx_3, \bx_4)  \big), \bx_2 \Big),$$
    such that for any $\bx_1, \bx_3 \in [-1,1]^{m_1}$, $\bx_2 \in [-1,1]^{m_2}$, and $\bx_4 \in [-1,1]^{(r_1 - 1)m_1}$,
    \begin{align}
        \Big| \bm{\Phi}_{2}(\bx_1, \bx_2, \bx_3, \bx_4) - \exp\big( -  \bx_1^{\top} \, \bg_0(\bx_3, \bx_4)
      \big) \bx_2 \Big| \lesssim (N_{\mathrm{mlp}})^{-\frac{2\beta_1}{m_1 r_1}}.
     \label{eq_nu_2}
    \end{align}

 By Assumption~\ref{assumption_transformer_size_lower}, we have 
    $$B^{L_{\mathrm{mlp}}-1} \geq (\log^2 T)^{\log^3 T} \geq T \cdot T \geq \frac{\sqrt{r_2}}{\lambda} \geq \frac{1}{\lambda}.$$
Using Lemma \ref{lemma_mlp_scaleup}, 
    we can construct  
    $\bm{\Phi}_{3} \in \mathcal{MLP}((r_2+1)m_2+1, (r_2+1)m_2+1, L_{\mathrm{mlp}}, 2(r_2+1)m_2+2, B)$
    such that
    \begin{align}
        \bm{\Phi}_{3}(\bx_5, \bx_2, 1) = 
            \stack\Big(\frac{\bx_{5}}{\lambda},
            \frac{\bx_{2}}{\lambda},
            \frac{\sqrt{r_2}}{\lambda}\Big). \label{eq_nu_35}
    \end{align}
    Similarly, notice that 
    $(\lceil \tfrac{N_{\mathrm{mlp}}}{\log T}\rceil B)^{(\log N_{\mathrm{mlp}})} \geq N_{\mathrm{mlp}}^{(\log N_{\mathrm{mlp}})} \geq T^3$ since $N_{\mathrm{mlp}}\geq T^{\alpha}$ by Assumption~\ref{assumption_transformer_size_lower}.
    Using Lemma~\ref{lemma_pos_mlp},
    we can construct $\bm{\Phi}_{4} \in \mathcal{MLP}(\tfrac{N}{2},1,2+\lceil \log(N_{\mathrm{mlp}})\rceil,\lceil \tfrac{N_{\mathrm{mlp}}}{\log T}\rceil ,B)$
    such that
    \begin{align}
        \bm{\Phi}_{4}(\bq_t) =  T^3 \, \mathbb{I}\big(t \in [\log^2 T, T - \log^2 T]\big).
        \label{eq_nu_6}
    \end{align}
    By stacking the MLPs $(\bm{\Phi}_{1}, \bm{\Phi}_{2}, \bm{\Phi}_{3}, \bm{\Phi}_{4})$ and using \eqref{eq_nu_1}, \eqref{eq_nu_2}, \eqref{eq_nu_35} and \eqref{eq_nu_6}, we obtain the assertion. 
\end{proof}

Later, we apply Lemma \ref{lemma_construct_firstmlp} by setting
\[
\bx_1 = \bm{\phi}_1(\bX_t), \quad
\bx_2 = \bm{\phi}_2(\bX_t), \quad
\bx_3 = \bm{\phi}_1(\bX_{t-r_1}), \quad
\bx_4 = \stack(\bm{\phi}_1(\bX_{t-r_1+1}), \ldots, \bm{\phi}_1(\bX_{t-1})), \quad
\bx_5 = \bZ_{2,t}.
\]

\begin{lemma} \label{lemma_normalize}
    Given  $0<a<b$, $\delta>0,$ and input dimension $d$, there exist constants $c=c(a,b,\delta,d),$ $c'=c'(a,b,\delta,d)$ such that  for any sufficiently large $M$, there exists a multilayer perceptron $\bbf: \mathbb{R}^d \to \mathbb{R}^d$
with depth $\lceil \log_2 (4b^2 M / a^2) \rceil M + 2M$, width $M^3$, and parameters bounded in absolute value by $c,$ satisfying
    $$\sup_{\bx \in \mathbb{R}^d :  a \leq \|\bx\|_2 \leq b}\left\|\bbf(\bx)-\frac{\delta}{\|\bx\|_2} \bx \right\|_{\infty}\leq c'  M^2 2^{-M}.$$
\end{lemma}

\begin{proof}
    First, applying Lemma \ref{lemma_mlp_multiply} $d$ times and summing the results shows the existence of 
    a multilayer perceptron 
    $g_1 : \mathbb{R}^d \to \mathbb{R}$ with depth $M$, width $7d$, and parameters bounded in absolute value by a constant
depending only on $b$ such that
    $$\sup_{\bx \in [-b,b]^d} 
    \big|g_1(\bx)-\|\bx\|_2^2 \big| \lesssim 2^{-M}.$$    
    Second, by Lemma~\ref{lemma_1oversqrtx}, there exists 
    a multilayer perceptron 
    $g_2 : \mathbb{R} \to \mathbb{R}$ with depth $\lceil \log_2 (4b^2 M / a^2) \rceil M$, width $7 \lceil 4 b^2 M / a^2 \rceil^2$, and parameters bounded in absolute value by a constant
depending only on $(a, b, d)$ such that
    $$\sup_{x \in [d a^2/2, 2db^2]}\left|g_2(x)-\frac{1}{\sqrt{x}}\right| \lesssim M^2 2^{-M}.$$
    Third, applying Lemma \ref{lemma_mlp_multiply} $d$ times, multiplying by $\delta$ at each step, and parallelizing the results, shows the existence of 
    a multilayer perceptron 
    $g_3 : \mathbb{R}^{d+1} \to \mathbb{R}^d$ with depth $M$, width $7d$, and parameters bounded in absolute value by a constant
depending only on $(a, b, \delta, d)$ such that
    $$\sup_{\bx \in [-b,b]^d, y \in [(3db^2)^{-1/2}, (d a^2/3)^{-1/2}] } 
    \big\|g_3(\bx,y)- (\delta y) \cdot \bx \big\|_{\infty} \lesssim 2^{-M}.$$ 
    Finally, we consider a multilayer perceptron $\bbf :\mathbb{R}^d \to \mathbb{R}^d$ defined by
    $$\bbf( \bx ) \coloneq g_3\big( \bx, g_2 \circ g_1(\bx)\big),$$
    which has depth $\lceil \log_2 (4b^2 M / a^2) \rceil M + 2M$, width $\max(7d, 7 \lceil 4 b^2 M / a^2 \rceil^2)+d \leq M^3$, and parameters bounded in absolute value by a constant
depending only on $(a, b,  \delta, d)$. 
    Here, the additional $d$ neurons in each layer are used to preserve the input $\bx \in [-b,b]^d$ via the equality $(x+b)_+ - b = x$ for $x \in [-b,b]$.

    For sufficiently large $M$, if $\bx \in \mathbb{B}^d(b) \setminus \mathbb{B}^d(a)$, then we have $g_1(\bx) \in [da^2/2, 2db^2]$ and $g_2 \circ g_1(\bx) \in [(3db^2)^{-1/2}, (da^2/3)^{-1/2}]$.  
    Thus, the upper bound of the approximation error is given by
    \begin{align*}
        \sup_{\bx \in \mathbb{B}^d(b) \setminus \mathbb{B}^d(a)}\left\|\bbf(\bx)-\frac{\delta}{\|\bx\|_2} \bx \right\|_{\infty}  
        &=
        \sup_{\bx \in \mathbb{B}^d(b) \setminus \mathbb{B}^d(a)}\left\|\bg_3\Big( \stack \big( \bx,  g_2 \circ g_1(\bx)\big)\Big)-\frac{\delta}{\|\bx\|_2} \bx \right\|_{\infty} \\
        &\leq  \sup_{\bx \in \mathbb{B}^d(b) \setminus \mathbb{B}^d(a)}\left\|\bg_3\Big( \stack \big( \bx, g_2 \circ g_1(\bx)\big)\Big) -  \delta \big(g_2 \circ g_1(\bx)\big) \bx \right\|_{\infty}\\
        & \quad +
        \sup_{\bx \in \mathbb{B}^d(b) \setminus \mathbb{B}^d(a)}\left\| \delta \big(g_2 \circ g_1(\bx)\big) \bx - \delta \big(g_1(\bx)\big)^{-1/2}  \bx \right\|_{\infty}\\
        & \quad +
        \sup_{\bx \in \mathbb{B}^d(b) \setminus \mathbb{B}^d(a)}\left\|\delta \big(g_1(\bx)\big)^{-1/2}\bx - \delta \big(\|\bx\|_2^2\big)^{-1/2} \cdot \bx  \right\|_{\infty}\\
        &\lesssim 2^{-M} + M^2 2^{-M} + 2^{-M} \lesssim M^2 2^{-M}.
    \end{align*}
\end{proof}

\subsection{Proof of Lemma \ref{lemma_approx_firstpart}}
\label{prove_lemma_approx_firstpart}

Let
$$
M \coloneq  \left[\begin{matrix}
\bm{\phi}_1(\bX_1) & \bm{\phi}_1(\bX_2) & \cdots & \bm{\phi}_1(\bX_T) \\
\bm{\phi}_2(\bX_1) & \bm{\phi}_2(\bX_2) & \cdots & \bm{\phi}_2(\bX_T) \\
\bm{0}_{\tfrac{N}{2} -m_1-m_2} & \bm{0}_{\tfrac{N}{2} -m_1-m_2} & \cdots & \bm{0}_{\tfrac{N}{2} -m_1-m_2} \\
\bq_1 & \bq_2 & \cdots & \bq_T 
\end{matrix}\right] = \left( 
\begin{matrix}
    \bm{\phi}_1(\bX_t)\\
    \bm{\phi}_2(\bX_t)\\
    \bm{0}_{\tfrac{N}{2} - m_1 - m_2}\\
    \bq_t
\end{matrix} \right)_{t \in [T]}
\begin{array}{l}
            \in \mathbb{R}^{m_1}\\
            \in \mathbb{R}^{m_2}\\
            \left. \right.\\
            \in \mathbb{R}^{N/2}
        \end{array}
        $$
and introduce
$\bm{\mu}_{1,0}: \mathbb{R}^{N \times T} \to \mathbb{R}^{N \times T}$, 
    $$\bm{\mu}_{1,0} : \left(\begin{matrix}
\bx_{1,t} \\
\bx_{2,t} \\
\bm{0}_{\tfrac{N}{2} -m_1-m_2} \\
\bq_t 
\end{matrix}\right)_{t \in [T]} 
        \begin{array}{l}
            \in \mathbb{R}^{m_1}\\
            \in \mathbb{R}^{m_2}\\
            \left. \right.\\
            \in \mathbb{R}^{N/2}
        \end{array}
        \longmapsto \quad \left(\begin{matrix}
    \bx_{1,t} \\
    \bx_{2,t} \\
    \bx_{1,t-r_1} \\
    \bx_{1,t-r_1+1} \\
    \vdots \\
    \bx_{1,t-1} \\
    \bx_{2,t-r_2} \\
    \bx_{2,t-r_2+1} \\
    \vdots \\
    \bx_{2,t-1} \\
    \bm{0}_{\tfrac{N}{2}- (r_1+1)m_1 -  (r_2+1)m_2}  \\
    \bq_t 
\end{matrix}\right)_{t \in [T]}
        \begin{array}{l}
            \in \mathbb{R}^{m_1}\\
            \in \mathbb{R}^{m_2}\\
            \in \mathbb{R}^{m_1}\\
            \in \mathbb{R}^{m_1}\\
            \vdots\\
            \in \mathbb{R}^{m_1}\\
            \in \mathbb{R}^{m_2}\\
            \in \mathbb{R}^{m_2}\\
            \vdots\\
            \in \mathbb{R}^{m_2}\\
            \left. \right.\\
            \in \mathbb{R}^{N/2}\\          
        \end{array}.
        $$
Using the definition of $\bZ_{1,t}$ and $\bZ_{2,t}$ in \eqref{def_Z}, the map $\bm{\mu}_{1,0}$ applied to $M$ has the form
\begin{align}
    \bm{\mu}_{1,0} (M) = \left(
    \begin{matrix}
        \bm{\phi}_1(\bX_t)\\
        \bm{\phi}_2(\bX_t)\\ 
        \bZ_{1,t}\\
        \bZ_{2,t}\\
        \bm{0}_{\tfrac{N}{2}- (r_1+1)m_1 -  (r_2+1)m_2}\\
        \bq_t
    \end{matrix}
    \right)_{t \in [T]}
    =
    \left(
    \begin{matrix}
        \bm{\phi}_1(\bX_t)\\
        \bm{\phi}_2(\bX_t)\\ 
        \bm{\phi}_1(\bX_{t-r_1})\\
        \stack\big( \bm{\phi}_1(\bX_{t-r_1+1}), \ldots, \bm{\phi}_1(\bX_{t-1}) \big)\\
        \bZ_{2,t}\\
        \bm{0}_{\tfrac{N}{2}- (r_1+1)m_1 -  (r_2+1)m_2}\\
        \bq_t
    \end{matrix}
    \right)_{t \in [T]} 
     \begin{array}{l}
            \in \mathbb{R}^{m_1}\\
            \in \mathbb{R}^{m_2}\\
            \in \mathbb{R}^{m_1}\\
            \in \mathbb{R}^{(r_1 - 1)m_1}\\
            \in \mathbb{R}^{r_2 m_2}\\
            \left. \right. \\        
            \in \mathbb{R}^{N/2}
    \end{array}
    .
    \label{eq.bcidsvd}
\end{align}
Introducing moreover the map $\bm{\nu}_{1,0} : \mathbb{R}^{N \times T} \to \mathbb{R}^{N \times T},$
$$
\bm{\nu}_{1,0} : \left( \begin{matrix}
    \bx_{1,t} \\
    \bx_{2,t} \\
    \bx_{3,t} \\
    \bx_{4,t} \\
    \bx_{5,t}\\
    \bm{0}_{\tfrac{N}{2}- (r_1+1)m_1 -  (r_2+1)m_2}\\
    \bq_t
    \end{matrix}\right)_{t \in [T]} 
            \begin{array}{l}
            \in \mathbb{R}^{m_1}\\
            \in \mathbb{R}^{m_2}\\
            \in \mathbb{R}^{m_1}\\
            \in \mathbb{R}^{(r_1 - 1)m_1}\\
            \in \mathbb{R}^{r_2 m_2}\\
            \left. \right. \\        
            \in \mathbb{R}^{N/2}
        \end{array}
        \longmapsto \quad  
    \left(\begin{matrix}
        \bg_0(\bx_{4,t}, \bx_{1,t}) \\
        \exp\big(-\bx_{1,t}^{\top}  \bg_0(\bx_{3,t}, \bx_{4,t}) \big) \bx_{2,t} \\
        \frac{1}{\lambda} \bx_{5,t} \\
        \frac{1}{\lambda} \bx_{2,t} \\
         \frac{\sqrt{r_2}}{\lambda}\\
        T^3 \, \mathbb{I}(t \in [\log^2 T, T - \log^2 T])\\
        \bm{0}_{N - m_1 - (r_2+2)m_2-2} \end{matrix}\right)_{t \in [T]}
        \begin{array}{l}
            \in \mathbb{R}^{m_1}\\
            \in \mathbb{R}^{m_2}\\
            \in \mathbb{R}^{r_2 m_2}\\
            \in \mathbb{R}^{m_2}\\
            \in \mathbb{R}\\
            \in \mathbb{R}\\
            \left. \right.
        \end{array}
$$
yields
\begin{align*}
\bm{\nu}_{1,0} \circ \bm{\mu}_{1,0} \left(M\right) &= \left(\begin{matrix}
    \bg_0(\bZ_{1,t+1}) \\
    \exp\big[-\bg_0(\bZ_{1,t} )^{\top} \bm{\phi}_1(\bX_t)   \big] \bm{\phi}_2(\bX_t)\\
    \frac{1}{\lambda} \bZ_{2, t}\\
    \frac{1}{\lambda} \bm{\phi}_2(\bX_t) \\
    \frac{\sqrt{r_2}}{\lambda} \\
    T^3 \, \mathbb{I}(t \in [\log^2 T, T - \log^2 T])\\
    \bm{0}_{N - m_1 - (r_2+1)m_2-2}
    \end{matrix}\right)_{t \in [T]}\\
    &=
    \left(\begin{matrix}
    \bg_0(\bZ_{1,t+1}) \\
    \exp\big[-\bg_0(\bZ_{1,t} )^{\top} \bm{\phi}_1(\bX_t)   \big] \bm{\phi}_2(\bX_t)\\
    \frac{1}{\lambda} \bm{\phi}_2(\bX_{t-r_2})\\
    \frac{1}{\lambda} \stack\big(\bm{\phi}_2(\bX_{t-r_2+1}), \ldots, \bm{\phi}_2(\bX_{t-1}) \big)\\
    \frac{1}{\lambda} \bm{\phi}_2(\bX_t) \\
    \frac{\sqrt{r_2}}{\lambda} \\
    T^3 \, \mathbb{I}(t \in [\log^2 T, T - \log^2 T])\\
    \bm{0}_{N - m_1 - (r_2+1)m_2-2}
    \end{matrix}\right)_{t \in [T]}
    \begin{array}{l}
            \in \mathbb{R}^{m_1}\\
            \in \mathbb{R}^{m_2}\\
            \in \mathbb{R}^{m_2}\\
            \in \mathbb{R}^{(r_2-1) m_2}\\
            \in \mathbb{R}^{m_2}\\
            \in \mathbb{R}\\
            \in \mathbb{R}\\
            \left. \right.
        \end{array}
\end{align*}
Finally, we define the map $\bm{\mu}_{2,0}: \mathbb{R}^{N \times T} \to \mathbb{R}^{N}$,
$$
\bm{\mu}_{2,0} : \left(
\begin{matrix}
\bx_{1,t}\\
\bx_{2,t}\\
\bx_{3,t}\\
\bx_{4,t}\\
\bx_{5,t}\\
x_{6,t}\\
x_{7,t}\\
\bm{0}_{N - m_1 - (r_2+2)m_2-2} \\
\end{matrix}\right)_{t \in [T]}
        \begin{array}{l}
            \in \mathbb{R}^{m_1}\\
            \in \mathbb{R}^{m_2}\\
            \in \mathbb{R}^{m_2}\\
            \in \mathbb{R}^{(r_2-1) m_2}\\
            \in \mathbb{R}^{m_2}\\
            \in \mathbb{R}\\
            \in \mathbb{R}\\
            \left. \right.
        \end{array}
        \longmapsto \quad
        \left(
\begin{matrix}
\bx_{1,T}\\
\bx_{2,T}\\
\bx_{3,T}\\
\bx_{4,T}\\
\bx_{5,T}\\
x_{6,T}\\
x_{7,T}\\
\bm{0}_{N - m_1 - (r_2+3)m_2-2} \\
\by
\end{matrix}\right)
        \begin{array}{l}
            \left. \right.\\
            \left. \right.\\
            \left. \right.\\
            \left. \right.\\
            \left. \right.\\
            \left. \right.\\
            \left. \right.\\
            \left. \right.\\
            \in \mathbb{R}^{m_2}
        \end{array}
$$
with
$$
     \by \coloneqq    
     \ol{\sum}_{s \in [\log^2 T, T-\log^2 T]} \Big( \exp \big( \big\langle \stack(\bx_{3,s}, \bx_{4,s}),  \stack(\bx_{4,T}, \bx_{5,T})\big\rangle - x_{6,s} x_{6,T}\big) \,;\, 
\bx_{2,s}\Big).
$$
Then we have
\begin{align*}
    \bm{\mu}_{2,0} \circ \bm{\nu}_{1,0} \circ \bm{\mu}_{1,0} \left(M\right) 
    &= 
    \left(\begin{matrix}
    \bg_0(\bZ_{1,t+1}) \\
   \bullet_{N-m_1-m_2}\\
        \ol{\sum}_{s \in [\log^2 T, T-\log^2 T]}
        \Big( \exp \big(\frac{-r_2+\bZ_{2,s}^{\top} \bZ_{2,T+1}}{\lambda^2}\big) \, ; \, \exp\big(-\bg_0(\bZ_{1,s})^{\top} \bm{\phi}_1(\bX_{s})\big) \, \bm{\phi}_2(\bX_{s}) \Big)
    \end{matrix}\right)\\
    &= \left( \begin{matrix}
    \bg_0(\bZ_{1,t+1}) \\
    \bullet_{N-m_1-m_2}\\
    \ol{\sum}_{s \in [\log^2 T, T-\log^2 T]}
        \Big( K_\lambda(\bZ_{2,s}, \bZ_{2,T+1}) \, ; \, \exp\big(-\bg_0(\bZ_{1,s})^{\top} \bm{\phi}_1(\bX_{s})\big) \, \bm{\phi}_2(\bX_{s}) \Big)
    \end{matrix}\right)
\end{align*}
with $K_\lambda(\bz,\bz')\coloneq \exp(-\Vert \bz - \bz' \Vert^2/(2\lambda^2))$,
provided that $\|\bZ_{2,t}\|_2 = 1$ for $t \in \{r_2+1,\ldots,T+1\}$.
See \eqref{def_weighted_average} for the definition of $\ol{\sum}$.
With these new definitions, we can restate Lemma~\ref{lemma_approx_firstpart}.

\begin{lemma}[Restatement of Lemma \ref{lemma_approx_firstpart}] \label{lemma_approx_firstpart_rewrite}
    Suppose Assumptions \ref{assumption_gh_smoothness} and \ref{assumption_transformer_size_lower} hold. 
    For sufficiently large $T$ and 
    for any $\lambda \geq T^{-1}$, there exist multi-head self-attentions $\bm{\mu}_{1}, \bm{\mu}_2 \in \mathcal{MA}(N,H,B)$ and
    a multilayer perceptron $\bm{\nu}_1 \in \mathcal{MLP}(N, L_{\mathrm{mlp}}, N_{\mathrm{mlp}}, B)$ such that
    $$\left\| \big( \bm{\mu}_2 \circ \bm{\nu}_1 \circ \bm{\mu}_{1}(M) \big)_{:,T} 
    -
    \bm{\mu}_{2,0} \circ \bm{\nu}_{1,0} \circ \bm{\mu}_{1,0}(M)
    \right\|_{\infty, m_1, m_2} \lesssim (N_{\mathrm{mlp}})^{-\frac{2\beta_1}{m_1 r_1}} + \frac{1}{T},$$
    where $\| (x_1,x_2,\ldots,x_N)^{\top} \|_{\infty, m_1, m_2} \coloneqq \max_{i \in ([N] \setminus \{m_1+1, \ldots, N-m_2\})} |x_i|$. 
\end{lemma}

The proof invokes the following lemma, whose proof is deferred to the end of this section.

\begin{lemma} \label{lemma_ignorefirstlast}
    For any $\alpha_- < \alpha_+$, $a_1,\ldots, a_T \in [\alpha_-, \alpha_+]$, $\beta>0$,  $\by_1,\ldots,\by_T \in [-\beta,\beta]^{m_2}$ and $c>0$, 
    \begin{align*}
    \left\|
    \ol{\sum}_{s \in [T]} \Big( \exp \big(a_s + c \, \mathbb{I}(s \in [\log^2 T, T - \log^2 T])\big) \, ; \,  \by_s \Big) 
    -
    \ol{\sum}_{s \in [\log^2 T, T - \log^2 T]} \Big( \exp \big(a_s\big) \, ; \,  \by_s \Big) 
    \right\|_{\infty} \\
    \leq \frac{8 \beta \log^2 T}{ T \exp(c + \alpha_- - \alpha_+)}.
    \end{align*}
\end{lemma}

\begin{proof}[Proof of Lemma~\ref{lemma_approx_firstpart_rewrite}]
Choosing $\bm{\mu}_{1}$, $\bm{\nu}_1$ and $\bm{\mu}_2$ as in Lemmas \ref{lemma_approx_rgram}, \ref{lemma_construct_firstmlp} and \ref{lemma_represent_NW_similar}, the error can be decomposed via the triangle inequality into three terms,
\begin{align}    &\left\|\big(\bm{\mu}_2 \circ \bm{\nu}_1 \circ \bm{\mu}_{1} ( M )\big)_{:,T} - \bm{\mu}_{2,0} \circ \bm{\nu}_{1,0} \circ \bm{\mu}_{1,0} (M) \right\|_{\infty, m_1, m_2}\notag \\
    & \leq \left\|\big(\bm{\mu}_2 \circ \bm{\nu}_1 \circ \bm{\mu}_{1} (M)\big)_{:,T} - \bm{\mu}_{2,0} \circ \bm{\nu}_1 \circ \bm{\mu}_{1} (M) \right\|_{\infty, m_1, m_2}\label{tmp_15} \\ 
    &\quad +
    \left\|\bm{\mu}_{2,0} \circ \bm{\nu}_1 \circ \bm{\mu}_{1} \left(M\right)
    -
    \bm{\mu}_{2,0} \circ \bm{\nu}_{1,0} \circ \bm{\mu}_{1} (M)
    \right\|_{\infty, m_1, m_2} \label{tmp_16}\\
    &\quad + \left\|\bm{\mu}_{2,0} \circ \bm{\nu}_{1,0} \circ \bm{\mu}_{1} (M)
    -
    \bm{\mu}_{2,0} \circ \bm{\nu}_{1,0} \circ \bm{\mu}_{1,0} (M)
    \right\|_{\infty, m_1, m_2}. \label{tmp_17}
\end{align}
To derive some properties of the vectors $\bm{\mu}_{1}(M)$ and
$\bm{\nu}_1 \circ \bm{\mu}_{1}(M)$, set
\begin{align} \label{tmp_18}
\bm{\mu}_{1} (M)
\eqcolon 
\left(
\begin{matrix}
\ba_{1,t}\\
\ba_{2,t}\\
\ba_{3,t}\\
\ba_{4,t}\\
\ba_{5,t} \\
\ba_{6,t}\\
\ba_{7,t}\\
\end{matrix}\right)_{t \in [T]}
\begin{array}{l}
     \in \mathbb{R}^{m_1}  \\
      \in \mathbb{R}^{m_2}  \\
     \in \mathbb{R}^{m_1}  \\
      \in \mathbb{R}^{(r_1-1)m_1}  \\
     \in \mathbb{R}^{r_2 m_2}  \\
      \in \mathbb{R}^{N/2 - (r_1+1)m_1 - (r_2 + 1)m_2}  \\
     \in \mathbb{R}^{N/2}
\end{array},
\,
\bm{\nu}_1 \circ \bm{\mu}_{1} (M)
\eqcolon
\left(
\begin{matrix}
\bb_{1,t}\\
\bb_{2,t}\\
\bb_{3,t}\\
\bb_{4,t}\\
\bb_{5,t}\\
b_{6,t} \\
b_{7,t}\\
\bb_{8,t}\\
\end{matrix}\right)_{t \in [T]}
\begin{array}{l}
     \in \mathbb{R}^{m_1}  \\
      \in \mathbb{R}^{m_2}  \\
     \in \mathbb{R}^{m_2}  \\
     \in \mathbb{R}^{(r_2-1) m_2}  \\
      \in \mathbb{R}^{m_2}  \\
     \in \mathbb{R}  \\
      \in \mathbb{R}\\
     \in \mathbb{R}^{N - m_1 - (r_2+2)m_2 - 2}
\end{array}.
\end{align}
By Lemma~\ref{lemma_approx_rgram} and the definition of $M$, we have 
\begin{align}
    \| \stack(\ba_{3,t}, \ba_{4,t}) - \bZ_{1,t}\|_{\infty} \lesssim T^{1 - \log T}, \qquad \|\ba_{5,t} - \bZ_{2,t}\|_{\infty} \lesssim T^{1 - \log T} \label{tmp_22}
\end{align}
for all $t \in \{(r_1 \vee r_2) + 1, \ldots, T\}$. Moreover, using that $\|\by_{a,t}\|_\infty\leq 1$ in Lemma~\ref{lemma_approx_rgram},
\begin{align}
    \ba_{1,t} = \bm{\phi}_1(X_t), \,
    \ba_{2,t} = \bm{\phi}_2(X_t), \,
    \| \stack(\ba_{3,t}, \ba_{4,t}, \ba_{5,t} )\|_{\infty} \leq 1, \, 
    \ba_{6,t} = \bm{0}_{\tfrac{N}{2}- (r_1+1)m_1 -  (r_2+1)m_2}, \,
    \ba_{7,t}  = \bq_t \label{tmp_4}
\end{align}
for all $t \in [T]$.
It follows from  Lemma~\ref{lemma_construct_firstmlp} and  $\|\bg_0\|_{\infty} \leq \delta_1$ that for all sufficiently large $N_{\mathrm{mlp}}$ and any $t\in [T],$
\begin{align}
    \|\bb_{1,t}\|_{\infty} &\leq 1 + \delta_1, \quad \|\bb_{2,t}\|_{\infty} \leq 1+\exp(\delta_1 m_1), \quad \|\bb_{3,t}\|_{\infty} \leq \frac{1}{\lambda}, \quad \|\bb_{4,t}\|_{\infty} \leq \frac{1}{\lambda}, \quad \|\bb_{5,t}\|_{\infty} \leq \frac{1}{\lambda},
    \label{tmp_20} \\
b_{6,t} &= \frac{\sqrt{r_2}}{\lambda}, \quad
b_{7,t} = T^3 \, \mathbb{I}(t \in [\log^2 T, T - \log^2 T]), \quad
\bb_{8,t} = \bm{0}_{N - m_1 - (r_2+2)m_2-2}.
\label{tmp_21}
\end{align}

We now derive an upper bound for \eqref{tmp_15}. The vector semi-norm $\|\cdot\|_{\infty, m_1, m_2}$ only depends on the first $m_1$ components and on the last $m_2$ components. By the definitions of $\bm{\mu}_2$ and $\bm{\mu}_{2,0}$, the first $m_1$ components of the last column of $\bm{\mu}_2$ and the first $m_1$ components of $\bm{\mu}_{2,0}$ are both  $\bx_{1,T}$ such that an error can occur only in the last $m_2$ components. By \eqref{tmp_18} and the expression for the last $m_s$ components in \eqref{eq.38fgew}, we find
\begin{align*}
\eqref{tmp_15} = &\Bigg\| 
\ol{\sum}_{s \in [T]} \Big( \exp \big( \big\langle \stack(\bb_{3,s}, \bb_{4,s}),  \stack(\bb_{4,T}, \bb_{5,T})\big\rangle - b_{6,s} b_{6,T} + \frac{\lambda}{\sqrt{r_2}} b_{7,s} b_{6,T} \big) \, ; \, \bb_{2,s} \Big)
\\
& \quad -
\ol{\sum}_{s \in [\log^2 T, T-\log^2 T]} \Big( \exp \big( \big\langle \stack(\bb_{3,s}, \bb_{4,s}),  \stack(\bb_{4,T}, \bb_{5,T})\big\rangle - b_{6,s} b_{6,T} + \frac{\lambda}{\sqrt{r_2}} b_{7,s} b_{6,T} \big) \, ; \, \bb_{2,s} \Big)
\Bigg\|_{\infty}.    
\end{align*} 
By assumption, $\lambda \geq T^{-1},$ which is equivalent to $-T^2 \leq -1/\lambda^2.$ Combining this with \eqref{tmp_20} and \eqref{tmp_21}, for any $t \in [T],$
\[
-(m_2 + 1) r_2 T^2 \leq \frac{-r_2 m_2 - r_2}{\lambda^2} \leq \big\langle \stack(\bb_{3,t}, \bb_{4,t}),  \stack(\bb_{4,T}, \bb_{5,T})\big\rangle - b_{6,t} b_{6,T} \leq \frac{r_2 m_2 - r_2}{\lambda^2} \leq (m_2 - 1) r_2 T^2
\] 
and $\frac{\lambda}{\sqrt{r_2}} b_{7,t} b_{6,T} = T^3 \, \mathbb{I}(t \in [\log^2 T, T - \log^2 T])$.
By \eqref{tmp_20} and \eqref{tmp_21}, we can apply Lemma \ref{lemma_ignorefirstlast} with $(\alpha_-, \alpha_+, \beta, c)=(-(m_2+1) r_2 T^2, (m_2-1) r_2 T^2, 1 + \exp(\delta_1 m_1), T^3)$, which yields
\begin{align}
    \eqref{tmp_15} \leq \frac{8 \big(1 + \exp(\delta_1 m_1) \big) \log^2 T}{T \exp(T^3 - 2 m_2 r_2 T^2)} \lesssim \frac{1}{T}. \label{tmp_23}
\end{align} 
To derive now an upper bound of \eqref{tmp_16}, we analyze $\bm{\nu}_{1,0} \circ \bm{\mu}_{1} (M) - \bm{\nu}_{1} \circ \bm{\mu}_{1} (M)$. 
Set
$$\bm{\nu}_{1,0} \circ \bm{\mu}_{1} (M)
\eqcolon
\left(
\begin{matrix}
\tilde \bb_{1,t}\\
\tilde \bb_{2,t}\\
\tilde \bb_{3,t}\\
\tilde \bb_{4,t}\\
\tilde \bb_{5,t}\\
\tilde b_{6,t} \\
\tilde b_{7,t}\\
\tilde \bb_{8,t}\\
\end{matrix}\right)_{t \in [T]}
\begin{array}{l}
     \in \mathbb{R}^{m_1}  \\
      \in \mathbb{R}^{m_2}  \\
     \in \mathbb{R}^{m_2}  \\
     \in \mathbb{R}^{(r_2-1) m_2}  \\
      \in \mathbb{R}^{m_2}  \\
     \in \mathbb{R}  \\
      \in \mathbb{R}\\
     \in \mathbb{R}^{N - m_1 - (r_2+2)m_2 - 2}
\end{array}.$$
By \eqref{tmp_4}, every component of $\bm{\mu}_{1}(M)$ is bounded in absolute value by $1$.
By the definition of $\bm{\nu}_{1,0}$ and the definition of $\bm{\nu}_{1}$ in Lemma~\ref{lemma_construct_firstmlp}, only the first $m_1+m_2$ entries are mapped to different values and 
$$\stack\big(\bb_{3,t}, \bb_{4,t}, \bb_{5,t}, b_{6,t}, b_{7,t}, \bb_{8,t}\big) = \stack\big(\tilde \bb_{3,t}, \tilde \bb_{4,t}, \tilde \bb_{5,t}, \tilde b_{6,t}, \tilde b_{7,t}, \tilde \bb_{8,t}\big).$$ Moreover, Lemma~\ref{lemma_construct_firstmlp} yields the bounds
\begin{align*}
    &\|\bb_{1,t} - \tilde{\bb}_{1,t}\|_{\infty} \vee 
    \|\bb_{2,t} - \tilde{\bb}_{2,t} \|_{\infty} \lesssim (N_{\mathrm{mlp}})^{-\frac{2\beta_1}{m_1 r_1}}.
\end{align*}
Hence, we get
\begin{align}
    \eqref{tmp_16} & = \left\| \left(\begin{matrix}
        \bb_{1,T}\\
          \bullet_{N -m_1 - m_2}\\
    \ol{\sum}_{s \in [\log^2 T, T-\log^2 T]} 
    \Big( \exp \big( \big\langle \stack(\bb_{3,s}, \bb_{4,s}),  \stack(\bb_{4,T}, \bb_{5,T})\big\rangle - b_{6,s} b_{6,T} \big) \, ; \, \bb_{2,s} \Big)    
     \end{matrix}\right)\right. \quad \,  \nonumber\\ 
      & \qquad - 
     \left.\left(\begin{matrix}
        \tilde \bb_{1,T}\\
          \bullet_{N -m_1 - m_2}\\
    \ol{\sum}_{s \in [\log^2 T, T-\log^2 T]}
    \Big( \exp \big( \big\langle \stack(\tilde \bb_{3,s}, \tilde \bb_{4,s}),  \stack(\tilde \bb_{4,T}, \tilde \bb_{5,T})\big\rangle - \tilde b_{6,s} \tilde b_{6,T} \big) \, ; \,  \tilde \bb_{2,s} \Big)
     \end{matrix}\right) 
     \right\|_{\infty, m_1, m_2}  \nonumber \\
     &\lesssim (N_{\mathrm{mlp}})^{-\frac{2\beta_1}{m_1 r_1}}. \label{tmp_24}
\end{align}

Finally, we derive an upper bound of \eqref{tmp_17}.
For any $N \times T$ matrix of the form
$$
(\bx_t)_{t \in [T]} 
\coloneqq
\left( \begin{matrix}
    \bx_{1,t} \\
    \bx_{2,t} \\
    \bx_{3,t} \\
    \bx_{4,t} \\
    \bx_{5,t}\\
    \bx_{6,t}\\
    \bm{0}_{\tfrac{N}{2}- (r_1+1)m_1 -  (r_2+1)m_2}\\
    \bq_t
    \end{matrix}\right)_{t \in [T]}
    \begin{array}{l}
     \in [-1,1]^{m_1}  \\
      \in [-1,1]^{m_2}  \\
     \in [-1,1]^{m_1}  \\
     \in [-1,1]^{(r_1-1) m_1}  \\
     \in [-1,1]^{m_2}  \\
     \in [-1,1]^{(r_2-1) m_2}  \\
        \left. \right. \\
     \in [-1,1]^{N/2}
\end{array},
$$
recall that 
\[
    \bm{\nu}_{1,0}\big((\bx_t)_{t \in [T]}\big) 
    = 
    \left(\begin{matrix}
        \bg_0(\bx_{4,t}, \bx_{1,t}) \\
        \exp\big(-\bx_{1,t}^{\top}  \bg_0(\bx_{3,t}, \bx_{4,t}) \big) \bx_{2,t} \\
        \frac{1}{\lambda} \bx_{5,t} \\
        \frac{1}{\lambda} \bx_{6,t} \\
        \frac{1}{\lambda} \bx_{2,t} \\
         \frac{\sqrt{r_2}}{\lambda}\\
        T^3 \, \mathbb{I}(t \in [\log^2 T, T - \log^2 T])\\
        \bm{0}_{N - m_1 - (r_2+2)m_2-2} \end{matrix}\right)_{t \in [T]}
        \begin{array}{l}
            \in \mathbb{R}^{m_1}\\
            \in \mathbb{R}^{m_2}\\
            \in \mathbb{R}^{m_2}\\
            \in \mathbb{R}^{(r_2-1) m_2}\\
            \in \mathbb{R}^{m_2}\\
            \in \mathbb{R}\\
            \in \mathbb{R}\\
            \left. \right.
        \end{array}.
\]
and $\bm{\mu}_{2,0} \circ \bm{\nu}_{1,0}((\bx_t)_{t \in [T]}) = \stack(\bg_0(\bx_{4,T}, \bx_{1,T}), \bullet_{N - m_1 - m_2}, \bv( (\bx_t)_{t \in [T]}) )$ with
\begin{align*}
    &\bv\big( (\bx_t)_{t \in [T]} \big)\\ 
    &\coloneqq \ol{\sum}_{s \in [\log^2 T, T-\log^2 T]}
    \Bigg( \exp\bigg(\frac{\langle \stack(\bx_{5,s}, \bx_{6,s}), \stack(\bx_{6,T}, \bx_{2,T})\rangle -  r_2}{\lambda^2}\bigg) \, ; \,  \frac{\bx_{2,s}}{\exp\big(\bx_{1,s}^{\top}  \bg_0(\bx_{3,s}, \bx_{4,s})\big)}
    \Bigg).
\end{align*}
Considering another $N \times T$ matrix
$$
(\bx'_t)_{t \in [T]} 
\coloneqq
\left( \begin{matrix}
    \bx'_{1,t} \\
    \bx'_{2,t} \\
    \bx'_{3,t} \\
    \bx'_{4,t} \\
    \bx'_{5,t}\\
    \bx'_{6,t}\\
    \bm{0}_{\tfrac{N}{2}- (r_1+1)m_1 -  (r_2+1)m_2}\\
    \bq_t
    \end{matrix}\right)_{t \in [T]}
    \begin{array}{l}
     \in [-1,1]^{m_1}  \\
      \in [-1,1]^{m_2}  \\
     \in [-1,1]^{m_1}  \\
     \in [-1,1]^{(r_1-1) m_1}  \\
     \in [-1,1]^{m_2}  \\
     \in [-1,1]^{(r_2-1) m_2}  \\
        \left. \right. \\
     \in [-1,1]^{N/2}
\end{array},
$$
we have
\begin{align*}
    &\left\| \bm{\mu}_{2,0} \circ \bm{\nu}_{1,0}\big( (\bx_t)_{t \in [T]} \big) - \bm{\mu}_{2,0} \circ \bm{\nu}_{1,0}\big( (\bx'_t)_{t \in [T]} \big) \right\|_{\infty, m_1, m_2}\\
&= \big\| \bg_0(\bx_{4,T}, \bx_{1,T}) - \bg_0(\bx'_{4,T}, \bx'_{1,T}) \big\|_{\infty} 
\vee \big\| \bv\big( (\bx_t)_{t \in [T]} \big) - \bv\big( (\bx'_t)_{t \in [T]} \big) \big\|_{\infty}  \\
&\lesssim
\big\|\stack(\bx_{4,T}, \bx_{1,T}) - \stack(\bx_{4,T}', \bx_{1,T}') \big\|_{\infty}^{\beta_1 \land 1}\\
& \quad + \frac{1}{\lambda^2} \max_{t \in [\log^2 T, T-\log^2 T] \cup \{T\}} \big\|\stack(\bx_{2,t}, \bx_{5,t}, \bx_{6,t}) - \stack(\bx'_{2,t}, \bx'_{5,t}, \bx'_{6,t})  \big\|_{\infty}\\
& \quad +
\max_{t \in [\log^2 T, T-\log^2 T]} \big\| \stack(\bx_{1,t}, \bx_{2,t}) - \stack(\bx'_{1,t}, \bx'_{2,t}) \big\|_{\infty}\\
& \quad +
\Big( \max_{t \in [\log^2 T, T-\log^2 T]} \big\|\stack(\bx_{3,t}, \bx_{4,t}) - \stack(\bx'_{3,t}, \bx'_{4,t}) \big\|_{\infty} \Big)^{\beta_1 \land 1},
\end{align*}
using the $\beta_1$-H\"older smoothness of $\bg_0$ and the Lipschitz property
$\|\operatorname{softmax}(\bz)-\operatorname{softmax}(\bz')\|_{\infty}
\leq 2\|\bz-\bz'\|_{\infty}$ that holds for any real-valued vectors $\bz$ and $\bz'$ of the same dimension (see Lemma \ref{lemma_softmax_lip}).
Combining this with \eqref{eq.bcidsvd}, \eqref{tmp_22}, and \eqref{tmp_4}, we obtain
\begin{align}
    \eqref{tmp_17} \lesssim T^{(\beta_1 \land 1)(1 - \log T)} + \frac{T^{1 - \log T}}{\lambda^2} \lesssim \frac{1}{T} \label{tmp_25}
\end{align}
for all sufficiently large $T$.
From \eqref{tmp_23}, \eqref{tmp_24} and \eqref{tmp_25}, we get the assertion. 
\end{proof}

\begin{proof}[Proof of Lemma~\ref{lemma_ignorefirstlast}]
    With 
    \begin{align*}
        P &\coloneq \sum_{i=1}^T \exp \big(a_i + c \, \mathbb{I}(i \in [\log^2 T, T - \log^2 T])\big)\\
        Q &\coloneq \sum_{i \in [\log^2 T,T - \log^2 T ]} \exp \big(a_i + c\big),
    \end{align*}
    we get
    $$\frac{P-Q}{P} = \frac{\sum_{i \in ([T] \setminus [\log^2 T, T - \log^2 T])} \exp (a_i)}{\sum_{i=1}^T \exp \big(a_i + c \, \mathbb{I}(i \in [\log^2 T, T - \log^2 T])\big)} \leq \frac{\exp(\alpha_+) \cdot 2\log^2 T}{\exp(\alpha_- + c) \cdot T/2},$$
    and
    $$\left|\frac{1}{Q} - \frac{1}{P}\right| \leq \frac{4\log^2 T}{T \exp(c + \alpha_- - \alpha_+) Q}.$$
    Hence, for every $s \in [\log^2 T, T - \log^2 T]$,
    we have  
    \begin{align*}
        \left|\frac{   \exp \big(a_s + c\big) }{\sum_{i=1}^T \exp \big(a_i + c \, \mathbb{I}(i \in [\log^2 T, T - \log^2 T])\big)} - 
        \frac{ \exp (a_s)}{\sum_{i \in [\log^2 T, T - \log^2 T]} \exp (a_i)}\right|
        =  \exp \big(a_s + c\big) \left|\left(\frac{1}{P} - \frac{1}{Q}\right)\right|\\
        \leq \frac{4 \log^2 T }{T \exp(c + \alpha_- - \alpha_+)} \frac{\exp \big(a_s + c\big)}{Q},
    \end{align*}
    and summing over $s$,
    \begin{align}
        \sum_{s \in [\log^2 T, T - \log^2 T]} \left|\frac{   \exp \big(a_s + c \, \mathbb{I}(s \in [\log^2 T, T - \log^2 T])\big) }{\sum_{i=1}^T \exp \big(a_i + c \, \mathbb{I}(i \in [\log^2 T, T - \log^2 T])\big)} - 
        \frac{ \exp (a_s)}{\sum_{i \in [\log^2 T, T - \log^2 T]} \exp (a_i)}\right| \notag \\
        \leq \frac{4 \log^2 T }{T \exp(c + \alpha_- - \alpha_+) }. \label{tmp_10}
    \end{align}
    On the other hand, for every $s \notin [\log^2 T, T - \log^2 T]$, we have
    \begin{align*}
        \frac{   \exp (a_s) }{\sum_{i=1}^T \exp \big(a_i + c \, \mathbb{I}(i \in [\log^2 T, T - \log^2 T])\big)}
    & \leq \frac{\exp(\alpha_+)}{ \exp(\alpha_-+c) \cdot T/2},
    \end{align*}
    and hence
    \begin{align}
        \sum_{s \in ([T] \setminus [\log^2 T, T - \log^2 T])} \frac{   \exp \big(a_s + c \, \mathbb{I}(s \in [\log^2 T, T - \log^2 T])\big) }{\sum_{i=1}^T \exp \big(a_i + c \, \mathbb{I}(i \in [\log^2 T, T - \log^2 T])\big)}
    \leq \frac{4 \log^2 T}{ T \exp(c + \alpha_- - \alpha_+)}. \label{tmp_11}
    \end{align}
    Using \eqref{tmp_10}, \eqref{tmp_11} and $\|\by_s\|_{\infty} \leq \beta$, we obtain
\begin{align*}
& \left\|
    \ol{\sum}_{s \in [T]} \Big( \exp \big(a_s + c \, \mathbb{I}(s \in [\log^2 T, T - \log^2 T])\big) \, ; \,  \by_s \Big) 
    -
    \ol{\sum}_{s \in [\log^2 T, T - \log^2 T]} \Big( \exp \big(a_s\big) \, ; \,  \by_s \Big) 
    \right\|_{\infty}\\
&= \left\|\sum_{s=1}^T \frac{   \exp \big(a_s + c \, \mathbb{I}(s \in [\log^2 T, T - \log^2 T])\big) }{\sum_{i=1}^T \exp \big(a_i + c \, \mathbb{I}(i \in [\log^2 T, T - \log^2 T])\big)} \by_s   
    -
    \sum_{s \in [\log^2 T, T - \log^2 T]} \frac{ \exp (a_s)}{\sum_{i \in [\log^2 T, T - \log^2 T]} \exp (a_i)}  \by_s \right\|_{\infty} \\
    &\leq 
    \sum_{s \in [\log^2 T, T - \log^2 T]}   \left\|\frac{   \exp \big(a_s + c\big) }{\sum_{i=1}^T \exp \big(a_i + c \, \mathbb{I}(i \in [\log^2 T, T - \log^2 T])\big)} \by_s   
    -
    \frac{ \exp (a_s)}{\sum_{i \in [\log^2 T, T - \log^2 T]} \exp (a_i)}  \by_s \right\|_{\infty}\\
    & \quad +
    \sum_{s \in ([T] \setminus [\log^2 T, T - \log^2 T])}  \left\|\frac{   \exp \big(a_s + c \, \mathbb{I}(s \in [\log^2 T, T - \log^2 T])\big) }{\sum_{i=1}^T \exp \big(a_i + c \, \mathbb{I}(i \in [\log^2 T, T - \log^2 T])\big)} \by_s   \right\|_{\infty}\\
    & \leq \frac{8 \beta \log^2 T}{ T \exp(c + \alpha_- - \alpha_+)}.
    \end{align*}    
\end{proof}

\section{Directional Convergence of the Nadaraya--Watson Estimator}
\label{section_proof_lemma_key_NW}

In this section, we prove Lemma~\ref{lemma_key_NW_version2}. 
Unless  specified otherwise, the implicit constants in the notation $\lesssim$ and $\tilde{O}$ depend only on $(m_1, m_2, r_1, r_2, \delta_1, \delta_2, \beta_2, \sup_{\bh \in \mathcal{H}} \| \bh \|_{\mathcal{C}^{\beta_2}})$.

\subsection{Proof of Lemma~\ref{lemma_key_NW_version2}}

Recall that $\mathcal{E}_W \coloneqq \{\be_1, \ldots, \be_W\}$ denotes the set of tokens. The following lemma characterizes the Markovian dynamics of the token generation process.
Its proof is deferred to Appendix~\ref{appendix_proof_markovian}.

\begin{lemma}\label{lemma: markov}
For $r\in\mathbb{N}$, let $\mathbf{Z}=(\bZ_n)_{n\in\mathbb{N}_0}$ be a homogeneous $\mathcal{E}_W^r$-valued Markov chain with transition probability kernel
\begin{align*}
P(\bx,\bz)&\coloneq \mathbb{P}(\mathbf{Z}_{n+1}=\bz\vert \mathbf{Z}_n= \bx)
\\
&\,=c(\bx) \exp(g(\bx,\bz))\prod_{i=1}^{r-1} \mathbb{I}(\bz_i=\bx_{i+1}),\quad \bx = (\bx_1, \ldots, \bx_r) \in \mathcal{E}_W^r, \quad \bz = (\bz_1, \ldots, \bz_r) \in \mathcal{E}_W^r,
\end{align*}
where $g\colon \mathcal{E}_W^{r}\times \mathcal{E}_W^{r} \to \mathbb{R} $ is a bounded function and $c(\cdot)$ is the normalizing constant. Then $\bZ$ admits a unique stationary distribution $\rho$ satisfying
    \begin{align*}
 (e^{2\Vert g\Vert_\infty} W)^{-r}\leq    \rho(\bz)\leq (e^{-2\Vert g\Vert_\infty} W)^{-r},\quad \text{for all} \ \bz \in \mathcal{E}_W^r.
\end{align*}
Furthermore, for any $n=1,2,\ldots,$
\begin{align*}
    \sup_{\bx \in \mathcal{E}_W^r}\Vert P^n(\bx,\cdot)-\rho\Vert_{\operatorname{TV}}&\leq \Big(1-e^{-2\Vert g\Vert_\infty r}\Big)^{n/r}.
\end{align*}
\end{lemma}

The next lemma shows that, for a uniformly exponentially ergodic Markov chain on a countable state space, if the direction of the conditional expectation of a vector-valued function satisfies a smoothness condition, then the direction of the Nadaraya--Watson estimator based on a truncated trajectory converges to that of the corresponding conditional expectation.  
For a unit vector $\bm{\theta} \in \mathbb{S}^{d-1}$, define the projection matrices $P_{\bm{\theta}} \coloneq \bm{\theta} \bm{\theta}^{\top}$
and
$P^{\perp}_{\bm{\theta}} \coloneq I_{d} - \bm{\theta} \bm{\theta}^{\top}$, where
$P_{\bm{\theta}}$ projects onto the direction of $\bm{\theta}$ and $P^{\perp}_{\bm{\theta}}$ projects onto the subspace orthogonal to $\bm{\theta}$.
See \eqref{def_weighted_average} for the definition of $\ol{\sum}$.

\begin{lemma} \label{lemma_key_NW} 
Let $\mathbf{M}\coloneq 
(\bm{\xi}_t)_{t\in\mathbb{N}_0}$ be a homogeneous Markov chain on a countable set $\mathcal{D},$ which admits a unique stationary distribution $\rho.$ Assume that $\mathbf{M}$ is uniformly exponentially ergodic, i.e.\ there exists a constant $c_{\mathbf{M}}\in(0,1)$ such that for any $n=0,1,\ldots$
\[\sup_{\bx \in\mathcal{D}}\Vert P^n(\bx,\cdot)-\rho\Vert_{\TV}\leq c_{\mathbf{M}}^n,\]
 where $P^n$ denotes the $n$-step transition probability of $\mathbf{M}$. 
 We write $\mathbb{P}^{\bx}$ and $\mathbb{E}^{\bx}$ for the probability and expectation of the Markov chain $\mathbf{M}$ started at $\bm{\xi}_0 = \bx$.
 For a bounded function $\bbf\colon \mathcal{D}\to \mathbb{R}^d$ define 
\[
\bbm(\bx) \coloneqq \mathbb{E}^{\bx}[\bbf(\bm{\xi}_1)], \qquad \bm{\theta}(\bx) \coloneq \frac{\bbm(\bx)}{\big\|\bbm(\bx)\big\|_2}, \quad \bx\in\mathcal{D},
\]
where we assume that $\inf_{\bx \in \mathcal{D}}\|\bbm(\bx)\|_{2} >0.$ In addition, assume that for some bounded function $\bbs : \mathcal{D} \to \mathbb{R}^{m}$,  $\beta \in (0,1]$ and $c_{\btheta}>0$
  it holds
  \begin{align}
      \Vert \btheta(\bx)-\btheta(\bx')\Vert_2\leq c_{\btheta}\Vert \bbs(\bx)-\bbs(\bx')\Vert_2^\beta,\quad \quad \text{for all} \ \bx,\bx'\in\mathcal{D}.
      \label{cond_lip_theta_s}
  \end{align}
Define 
    $$\wh{\bbm}_{\lambda}( \cdot ) \coloneqq
    \ol{\sum}_{t \in [\log^2 T, T - \log^2 T]} \Big( K_\lambda \big(\bbs(\bm{\xi}_t), \bbs(\cdot)\big) \, ; \, \bbf(\bm{\xi}_{t+1}) \Big)   
    $$
    with $K_\lambda(\bz,\bz')\coloneq \exp(-\Vert \bz - \bz'\Vert_2^2/(2\lambda^2))$.
    Then, for any $\bx\in\mathcal{D}$,
    \begin{align}
        \mathbb{E}^{\bx}\left[ \big\| P^{\perp}_{\bm{\theta}(\bm{\xi}_{T})} \wh{\bbm}_{\lambda}(\bm{\xi}_{T})  \big\|^2_2 \right] = \tilde{O} \Big( \lambda^{2 \beta} 
    + \frac{1}{T \inf_{\bx \in \mathcal{D}} \mathbb{E}_{\rho}[K_{\lambda}\big(\bbs(\bm{\xi}),\bbs(\bx)\big)]} \Big) \label{lemma_key_first_argument}
    \end{align}
    and
    \begin{align}
    &\mathbb{P}^{\bx}\bigg( \big\| P_{\bm{\theta}(\bm{\xi}_{T})} \wh{\bbm}_{\lambda}(\bm{\xi}_{T})  \big\|_2 < \inf_{\bx \in \mathcal{D}} \frac{\big\Vert\E_{\rho}\left[ K_\lambda \big(\bbs(\bm{\xi}), \bbs(\bx)\big) P_{\bm{\theta}(\bx)} \bbm(\bm{\xi})\right]\big\Vert_2}{4\E_{\rho}\big[ K_\lambda \big(\bbs(\bm{\xi}), \bbs(\bx)\big) \big]} \bigg) \notag \\
    & \lesssim \frac{1}{T \inf_{\bx \in \mathcal{D}} \mathbb{E}_{\rho}[K_{\lambda}\big(\bbs(\bm{\xi}),\bbs(\bx)\big)]} + 
     \sup_{\bx \in \mathcal{D}} \frac{\E_{\rho}\big[ K_\lambda \big(\bbs(\bm{\xi}), \bbs(\bx)\big) \big]}{T \big\Vert\E_{\rho}\left[ K_\lambda \big(\bbs(\bm{\xi}), \bbs(\bx)\big) P_{\bm{\theta}(\bx)} \bbm(\bm{\xi})\right]\big\Vert_2^2},    \label{lemma_key_second_argument}
    \end{align}    
    where the implicit constants depend only on $c_{\mathbf{M}}$, $c_{\btheta}$, $\beta$, $\Vert \bbf \Vert_\infty \coloneqq \max_{\bx \in \mathcal{D}} \|\bbf(\bx) \|_2$ and $\Vert \bbs \Vert_\infty \coloneqq \max_{\bx \in \mathcal{D}} \|\bbs(\bx) \|_2$.     
    \end{lemma}

The proof is deferred to Appendix~\ref{appendix_proof_direction_conv_NW}.
The first assertion controls the component of $\wh{\bbm}_{\lambda}(\bm{\xi}_T)$ that is orthogonal to $\bbm(\bm{\xi}_T)$, while the second assertion ensures that the component along $\bbm(\bm{\xi}_T)$ is sufficiently large with high probability.   
The estimator $\wh{\bbm}_{\lambda}$ in the lemma excludes indices $t$ near the boundaries $t < \log^2 T$ and $t > T - \log^2 T$. 
This ensures that the estimator only uses samples which are similarly distributed as the invariant distribution and are approximately independent of $\bm{\xi}_{T}$.

\begin{proof}[Proof of Lemma~\ref{lemma_key_NW_version2}]
    
Define $$r \coloneqq (r_1+1) \lor r_2.$$
Moreover, let
\begin{align*}
    \bm{\xi}_t \coloneqq (\bX_{t+1}, \ldots, \bX_{t+r})
\end{align*}
and, for $\bx = (\bx_1, \ldots, \bx_r) \in \mathcal{E}_W^{r}$,
\begin{align*}
    \bbf\big(\bx\big) \coloneqq \frac{\bm{\phi}_2(\bx_{r})}{\exp\Big(\bg_0\big( \bm{\phi}_1(\bx_{r-r_1}), \ldots, \bm{\phi}_1(\bx_{r-1}) \big)^{\top} \bm{\phi}_1(\bx_{r})\Big)},
     \qquad
     \bbs(\bx)  \coloneqq \stack\big(\bm{\phi}_2(\bx_{r-r_2 + 1}), \ldots, \bm{\phi}_2(\bx_r)\big).
\end{align*}
By Lemma \ref{lemma: markov}, for given $\bh$,
$\{\bm{\xi}_t\}_{t \in \mathbb{N}_0}$
is a homogeneous Markov chain on $\mathcal{E}_W^{r}$,
which admits a unique
stationary distribution $\rho$, satisfying
 \begin{align}
 \Big(\frac{\exp(-2 m_1 \delta_1 - 2 \delta_2)}{W}\Big)^{r}
 \leq    \rho(\bx)\leq \Big(\frac{\exp(2 m_1 \delta_1 + 2 \delta_2)}{W}\Big)^{r} ,\qquad \bx \in \mathcal{E}_W^{r} \label{rho_maximum_minimum}
\end{align}
and
\begin{align*}
    \sup_{\bx \in \mathcal{E}_W^{r}}\Vert P^n(\bx,\cdot)-\rho\Vert_{\operatorname{TV}}&\leq \Big(1-e^{-2 r (m_1 \delta_1 + \delta_2)}\Big)^{n/r},\quad n\in\mathbb{N},
\end{align*}
where $P^n$ denotes the $n$-step transition probability of the Markov chain.
This implies that $\{\bm{\xi}_t\}_{t \in \mathbb{N}_0}$ is uniformly exponentially ergodic.
We further define 
\[
\bbm(\bx) \coloneqq \mathbb{E}[\bbf(\bm{\xi}_1)|\bh, \bm{\xi}_0 = \bx], \qquad
\bm{\theta}(\bx) \coloneq \frac{\bbm(\bx)}{\big\|\bbm(\bx)\big\|_2}, \qquad \bx \in \mathcal{E}_W^{r}.
\]
Then we have the following results,  whose proof is deferred to Appendix~\ref{proof_lemma_application_keylemma_properties}.
See \eqref{def_weighted_average} for the definition of $\ol{\sum}$.
\begin{lemma} \label{lemma_application_keylemma_properties}
    Under the assumptions of Lemma~\ref{lemma_key_NW_version2} and the above notation,
    \begin{enumerate}
        \item [(i)] for any $\bx \in \mathcal{E}_W^{r}$,
        the directions of $\bbm(\bx)$ and 
        \[
        \ol{\sum}_{w \in [W]}
        \Big( \exp\big( \bh(\bbs(\bx))^{\top} \bm{\phi}_2(\be_w) \big) \, ; \, \bm{\phi}_2(\be_{w}) \Big)
        \]
        coincide,
        \item [(ii)] there exists $c_1>0$  depending only on $(m_1, m_2, \delta_1, \delta_2)$ such that for any $\bx \in \mathcal{E}_W^{r},$
        \[
        \| \bbm(\bx) \|_2 \geq c_1,
        \]
        \item [(iii)] there exists $c_{\btheta} >0$ depending only on $(m_2, \delta_2, \|\bh\|_{\mathcal{C}^{\beta_2}})$ such that for any $\bx, \bx' \in \mathcal{E}_W^{r}$ 
        \[
            \| \bm{\theta}(\bx) - \bm{\theta}(\bx')  \|_2 \leq  c_{\btheta} \|\bbs(\bx) - \bbs(\bx')\|_{\infty}^{\beta_2},
        \]       
        \item[(iv)] there exists $c_2$ depending only on $(m_1, \delta_1, \delta_2, r_2)$, 
                \[
        \inf_{\bx \in \mathcal{E}_W^{r}} \mathbb{E}_{\rho}[K_{\lambda}\big(\bbs(\bm{\xi}),\bbs(\bx)\big)] \geq  c_2 \varphi_W(\lambda)^{r_2},
        \]
        \item[(v)] with $c_1 > 0$ as defined in (ii), 
        \[
            \inf_{\bx \in \mathcal{E}_W^{r}} \frac{\big\Vert\E_{\rho}\left[ K_\lambda \big(\bbs(\bm{\xi}), \bbs(\bx)\big) P_{\bm{\theta}(\bx)} \bbm(\bm{\xi})\right]\big\Vert_2}{\E_{\rho}\big[ K_\lambda \big(\bbs(\bm{\xi}), \bbs(\bx)\big) \big]} \geq \frac{c_1}{4}.
        \]
    \end{enumerate}
\end{lemma}

Define random vectors $\bZ_{1,t} \in \mathbb{R}^{r_1 m_1}$, 
$\bZ_{2,t} \in \mathbb{R}^{r_2 m_2}$, and 
$\bY_{t} \in \mathbb{R}^{m_2}$ as  
$$
\bZ_{1,t}
\coloneqq
\begin{pmatrix}
\bm{\phi}_1(\bX_{t-r_1}) \\
\bm{\phi}_1(\bX_{t-r_1+1}) \\
\vdots \\
\bm{\phi}_1(\bX_{t-1})
\end{pmatrix},
\qquad
\bZ_{2,t}
\coloneqq
\begin{pmatrix}
\bm{\phi}_2(\bX_{t-r_2}) \\
\bm{\phi}_2(\bX_{t-r_2+1}) \\
\vdots \\
\bm{\phi}_2(\bX_{t-1})
\end{pmatrix},
\qquad
\bY_{t} \coloneqq 
\frac{\bm{\phi}_2(\bX_t)}{\exp[\bg_0(\bZ_{1,t} )^{\top} \bm{\phi}_1(\bX_t)]}.
$$
Moreover, define
\begin{align*}
    \wh{\bbm}_{\lambda}(\bZ_{2,T+1}) &\coloneqq
\ol{\sum}_{s \in [\log^2 T, T - \log^2 T]} \Big( K_\lambda\big(\bZ_{2,s}, \bZ_{2,T+1} \big) \, ; \, \bY_{s} \Big)
\end{align*}
and
\[
    \bm{\alpha} \coloneqq \bm{\theta}(\bX_{T-r+1}, \ldots, \bX_{T}) \in \mathbb{S}^{m_2 - 1}, \quad P_{\bm{\alpha}} \coloneq \bm{\alpha} \bm{\alpha}^{\top}, \quad P^{\perp}_{\bm{\alpha}} \coloneq I_{m_2} - \bm{\alpha} \bm{\alpha}^{\top}.
\]
Both $P_{\bm{\alpha}}$ and $P^{\perp}_{\bm{\alpha}}$ are projection matrices, where $P_{\bm{\alpha}}$ denotes the projection onto the one-dimensional subspace $\mathbb{R}\bm{\alpha}$ and $P^{\perp}_{\bm{\alpha}}$ denotes the projection onto its orthogonal complement.
By (ii) and (iii) of Lemma \ref{lemma_application_keylemma_properties}, 
we can apply 
Lemma \ref{lemma_key_NW}.
Moreover, using (iv) and (v) of Lemma \ref{lemma_application_keylemma_properties}
we obtain
\begin{align}
    \mathbb{E}\left[ \| P^{\perp}_{\bm{\alpha}} \wh{\bbm}_{\lambda}(\bZ_{2,T+1})  \|^2_2\vert \bh \right] &= \tilde{O} \Big( \lambda^{2 \beta} + \frac{1}{T \varphi_W(\lambda)^{r_2}} \Big) \label{tmp_34}
\end{align}
and for some $c>0$
\begin{align}
    \mathbb{P}\big[ \| P_{\bm{\alpha}} \wh{\bbm}_{\lambda}(\bZ_{2,T+1})  \|_2 < c \vert \bh\big] 
    \lesssim \frac{1}{T \varphi_W(\lambda)^{r_2}}. \label{tmp_35}
\end{align}
Here, the implicit constants and $c$ depend only on $(m_1, m_2, \delta_1, \delta_2, r_1, r_2, \|\bh\|_{\mathcal{C}^{\beta_2}})$.
Since the Euclidean distance between two unit vectors is upper bounded by their angle, it holds
\[
\left\|\frac{\wh{\bbm}_{\lambda}(\bZ_{2,T+1})}{\|\wh{\bbm}_{\lambda}(\bZ_{2,T+1})\|_2}
-
\bm{\alpha}\right\|_2 
\leq 
\tan^{-1}\left(\frac{\| P^{\perp}_{\bm{\alpha}} \wh{\bbm}_{\lambda}(\bZ_{2,T+1})  \|_2}{\| P_{\bm{\alpha}} \wh{\bbm}_{\lambda}(\bZ_{2,T+1})  \|_2}\right)
\leq
\frac{\| P^{\perp}_{\bm{\alpha}} \wh{\bbm}_{\lambda}(\bZ_{2,T+1})  \|_2}{\| P_{\bm{\alpha}} \wh{\bbm}_{\lambda}(\bZ_{2,T+1})  \|_2}.
\]
Combining with \eqref{tmp_34} and \eqref{tmp_35}, 
\begin{align}
    \mathbb{E}\left[\left\|\frac{\wh{\bbm}_{\lambda}(\bZ_{2,T+1})}{\|\wh{\bbm}_{\lambda}(\bZ_{2,T+1})\|_2}
-
\bm{\alpha}\right\|^2_2\Big\vert\bh\right]
& = 
\mathbb{E}\left[\left\|\frac{\wh{\bbm}_{\lambda}(\bZ_{2,T+1})}{\|\wh{\bbm}_{\lambda}(\bZ_{2,T+1})\|_2}
-
\bm{\alpha}\right\|^2_2 \mathbb{I}(\| P_{\bm{\alpha}} \wh{\bbm}_{\lambda}(\bZ_{2,T+1})  \|_2 \geq c)\Big\vert\bh\right] \notag \\
& \quad +
\mathbb{E}\left[\left\|\frac{\wh{\bbm}_{\lambda}(\bZ_{2,T+1})}{\|\wh{\bbm}_{\lambda}(\bZ_{2,T+1})\|_2}
-
\bm{\alpha}\right\|^2_2 \mathbb{I}(\| P_{\bm{\alpha}} \wh{\bbm}_{\lambda}(\bZ_{2,T+1})  \|_2 < c)\Big\vert\bh\right] \notag \\
&\lesssim 
\mathbb{E}\left[\frac{\| P^{\perp}_{\bm{\alpha}} \wh{\bbm}_{\lambda}(\bZ_{2,T+1})  \|^2_2}{\| P_{\bm{\alpha}} \wh{\bbm}_{\lambda}(\bZ_{2,T+1})  \|^2_2} \mathbb{I}(\| P_{\bm{\alpha}} \wh{\bbm}_{\lambda}(\bZ_{2,T+1})  \|_2 \geq c) \Big\vert\bh\right]
\notag \\
& \quad + \mathbb{P}\big[ \| P_{\bm{\alpha}} \wh{\bbm}_{\lambda}(\bZ_{2,T+1})  \|_2 < c \big\vert\bh\big]
\notag \\
&\lesssim 
\mathbb{E}\big[ \| P^{\perp}_{\bm{\alpha}} \wh{\bbm}_{\lambda}(\bZ_{2,T+1})  \|^2_2 \big\vert\bh\big]
+ \mathbb{P}\big[ \| P_{\bm{\alpha}} \wh{\bbm}_{\lambda}(\bZ_{2,T+1})  \|_2 < c\big\vert\bh \big] \notag \\
& = 
\tilde{O}\Big(\lambda^{2 \beta_2} + \frac{1}{T \varphi_W(\lambda)^{r_2}}\Big). \label{tmp_48}
\end{align} 
In addition, by the definition of $\kappa_W$ (Definition \ref{def_pos_assumption}), 
\[
\bigg\| \underbrace{
\ol{\sum}_{w \in [W]} \Big( \exp\big( \bh(\bZ_{2,T})^{\top} \bm{\phi}_2(\be_w) \big) \, ; \, \bm{\phi}_2(\be_{w}) \Big)
}_{\eqcolon \bm{\chi}_1}  - \underbrace{\frac{\int_{\bu \in \mathbb{S}^{m_2 - 1}} \exp(\bh(\bZ_{2,T})^{\top} \bu) \bu \, d\bu }{\int_{\bu \in \mathbb{S}^{m_2 - 1}} \exp(\bh(\bZ_{2,T})^{\top} \bu) \, d\bu  }}_{\eqcolon \bm{\chi}_2}\bigg\|_{2} \leq \kappa_W.\]
Moreover, we have 
$\bm{\chi}_1 / \| \bm{\chi}_1 \|_2 = \bm{\alpha}$
by (i) of Lemma \ref{lemma_application_keylemma_properties} and
$\bm{\chi}_2 / \| \bm{\chi}_2 \|_2 = \bh(\bZ_{2,T}) / \delta_2 $
by the first argument of Lemma \ref{lemma_VMF_property}. 
Hence, 
\begin{align*}
        \left\| \bm{\alpha} - \frac{\bh(\bZ_{2,T})}{\delta_2}\right\|_2
    = \left\| \frac{\bm{\chi}_1}{\|\bm{\chi}_1\|_2} - \frac{\bm{\chi}_2}{\|\bm{\chi}_2\|_2} \right\|_2
     \leq \left\| \frac{\bm{\chi}_1-\bm{\chi}_2}{\|\bm{\chi}_2\|_2} \right\|_2
    +
    \left\| \left(\frac{1}{\|\bm{\chi}_1\|_2} - \frac{1}{\|\bm{\chi}_2\|_2}\right) \bm{\chi}_1  \right\|_2
     \lesssim \kappa_W,     
\end{align*}
where we use the second argument of Lemma \ref{lemma_VMF_property} for the last inequality.
Combining with \eqref{tmp_48}, we obtain
\begin{align}
    \mathbb{E}\left[\left\|\frac{\delta_2 \wh{\bbm}_{\lambda}(\bZ_{2,T+1})}{\|\wh{\bbm}_{\lambda}(\bZ_{2,T+1})\|_2}
- \bh(\bZ_{2,T}) \right\|^2_2\Big\vert\bh\right] = \tilde{O}\Big(\lambda^{2 \beta_2} + \frac{1}{T \varphi_W(\lambda)^{r_2}} + \kappa_W^2\Big). \label{tmp_50} 
\end{align}

\end{proof}

\subsection{Proof of Lemma~\ref{lemma: markov}}
\label{appendix_proof_markovian}

\begin{proof}
The normalization function $c(\bx)$ satisfies for any $\bx \in \mathcal{E}_W^r,$
\[c(\bx)\coloneq \left(\sum_{\bz \in\mathcal{E}_W^r}\exp(g(\bx,\bz))\prod_{i=1}^{r-1} \mathbb{I}(\bz_i=\bx_{i+1})\right)^{-1} \in \bigg[\frac {\exp(-\Vert g\Vert_\infty)}W,\frac {\exp(\Vert g\Vert_\infty)}W\bigg]. \] Additionally, the Chapman--Kolmogorov equation gives for the $r$-step transition probability
\begin{align}
   P^{r}\big((\bx_1,\ldots,\bx_r),(\bz_1,\ldots,\bz_r)\big)    
&\geq \prod_{t=1}^{r}
c\big((\bx_t,\ldots,\bx_r,\bz_1,\ldots, \bz_{t-1})\big)\exp(-\Vert g\Vert_\infty) \notag
\\
&\geq (e^{2\Vert g\Vert_\infty} W)^{-r}. \label{eq:meyn}
\end{align}
Thus by Theorem 16.0.2 in \cite{meyntweedie} there exists a stationary distribution $\rho$ and $\mathbf{Y}$ is uniformly ergodic. Now applying \eqref{eq:meyn} together with Theorem 16.2.4 in \cite{meyntweedie} gives
\begin{align*}
    \sup_{\bx \in \mathcal{E}_W^r}\Vert P^n(\bx,\cdot)-\rho\Vert_{\operatorname{TV}}&\leq \Big(1-e^{-2\Vert g\Vert_\infty r}\Big)^{n/r}.
\end{align*}  
Lastly, arguing similarly as in \eqref{eq:meyn} gives 
\[P^{r}\big((\bx_1,\ldots,\bx_r),(\bz_1,\ldots,\bz_r)\big)    
\leq (e^{-2\Vert g\Vert_\infty} W)^{-r}, \]
and since
\begin{align*}
   \rho\big( (\bz_1,\ldots,\bz_r)\big)&=\sum P^{r}\big((\bx_1,\ldots,\bx_r),(\bz_1,\ldots,\bz_r)\big)    \rho\big((\bx_1,\ldots,\bx_r)\big)
\end{align*}
this implies together with \eqref{eq:meyn}
\begin{align*}
 (e^{2\Vert g\Vert_\infty} W)^{-r}\leq    \rho(\bz)\leq (e^{-2\Vert g\Vert_\infty} W)^{-r}.
\end{align*}
\end{proof}

\subsection{Proof of Lemma~\ref{lemma_key_NW}}
\label{appendix_proof_direction_conv_NW}

  \begin{proof} 
For simplicity we write $\mathbb{P}= \mathbb{P}_\rho$ and $\mathbb{E}=\mathbb{E}_{\rho}$. 
Furthermore, for $T\in\mathbb{N}$, we introduce
\[\mathcal{T}=\mathcal{T}(T)\coloneq \{n\in\mathbb{N}_0\colon \log^2T\leq n\leq T-\log^2T\}, \quad T^*\coloneq \vert \mathcal{T}\vert+T-\max(\mathcal{T}).\]
We then obtain that for any $\bx\in\mathcal{D}$,
\begin{align}
     &\Bigg\vert\mathbb{E}^{\bx}\left[ \| P^{\perp}_{\bm{\theta}(\bm{\xi}_T)} \wh{\bbm}_{\lambda}(\bm{\xi}_{T})  \|^2_2 \right]-\mathbb{E}\left[ \| P^{\perp}_{\bm{\theta}(\bm{\xi}_T)} \wh{\bbm}_{\lambda}(\bm{\xi}_{T})  \|^2_2 \right]\Bigg\vert \notag \\ 
&=\Bigg\vert\mathbb{E}^{\bx}\left[ \Big\| P^{\perp}_{\bm{\theta}(\bm{\xi}_T)} 
\ol{\sum}_{t \in \mathcal{T}} \Big( K_\lambda \big(\bbs(\bm{\xi}_t), \bbs(\bm{\xi}_{T})\big) \, ; \, \bbf(\bm{\xi}_{t+1}) \Big)
\Big\|^2_2 \right]
\notag \\
&\qquad
-\mathbb{E}\left[ \Big\| P^{\perp}_{\bm{\theta}(\bm{\xi}_T)} 
\ol{\sum}_{t \in \mathcal{T}} \Big( K_\lambda \big(\bbs(\bm{\xi}_t), \bbs(\bm{\xi}_{T})\big) \, ; \, \bbf(\bm{\xi}_{t+1}) \Big) \Big\|^2_2 \right]\Bigg\vert 
\notag    \\
&=\Bigg\vert\mathbb{E}^{\bx}\left[\mathbb{E}^{\bm{\xi}_{\min(\mathcal
T)}}\left[ \Big\| P^{\perp}_{\bm{\theta}(\bm{\xi}_{T^*})} 
\ol{\sum}_{t \in \{0,1,\ldots,\vert \mathcal{T}\vert-1\}} \Big( K_\lambda \big(\bbs(\bm{\xi}_t), \bbs(\bm{\xi}_{T^*})\big) \, ; \, \bbf(\bm{\xi}_{t+1}) \Big)
\Big\|^2_2 \right] \right]
\notag \\
&\qquad 
-\mathbb{E}\left[ \mathbb{E}^{\bm{\xi}_{\min(\mathcal
T)}}\left[ \Big\| P^{\perp}_{\bm{\theta}(\bm{\xi}_{T^*})} \ol{\sum}_{t \in \{0,1,\ldots,\vert \mathcal{T}\vert-1\}} \Big( K_\lambda \big(\bbs(\bm{\xi}_t), \bbs(\bm{\xi}_{T^*})\big) \, ; \, \bbf(\bm{\xi}_{t+1}) \Big)
\Big\|^2_2 \right]\right]\Bigg\vert 
\notag \\
&\leq \Vert \bbf \Vert_\infty^2\sup_{\bx\in\mathcal{D}}\Vert P^{\min(\mathcal{T})}(\bx,\cdot)-\rho\Vert_{\TV}
\notag \\
&\leq \Vert \bbf \Vert_\infty^2c_{\mathbf{M}}^{\min(\mathcal
T)}, \label{tmp_40}
\end{align}
where we applied the Markov property in the second equality. In the first inequality we used that \[\mathbb{E}^{\bm{\xi}_{\min(\mathcal
T)}}\left[ \Big\| P^{\perp}_{\bm{\theta}(\bm{\xi}_{T^*})} \ol{\sum}_{t \in \{0,1,\ldots,\vert \mathcal{T}\vert-1\}} \Big( K_\lambda \big(\bbs(\bm{\xi}_t), \bbs(\bm{\xi}_{T^*})\big) \, ; \, \bbf(\bm{\xi}_{t+1}) \Big)
\Big\|^2_2 \right]\] is a measurable function of $\bm{\xi}_{\min(\mathcal
T)},$ which is bounded by $\Vert \bbf\Vert_\infty^2$ together with the fact that $\bm{\xi}_{\min(\mathcal
T)}\sim P^{\min(\mathcal{T})}$ under $\mathbb{P}^{\bx}$ and  $\bm{\xi}_{\min(\mathcal
T)} \sim \rho$ under $\mathbb{P}$. We now apply Berbee's coupling lemma, see Lemma 5.1 in \cite{viennet97} and \cite{berbee79}. This states for two random variables $X,Y,$ that there exists a random variable $Y^*,$ which is independent of $X,$ identically distributed as $Y,$ and satisfies \[\P(Y\neq Y^*)=\beta(X,Y),\] 
where $\beta(X,Y)$ is the $\beta$-mixing coefficient of $X$ and $Y$.
Since the investigated Markov chain is stationary under $\mathbb{P},$ and uniformly ergodic, we can argue as in the proof of Theorem 3.1 in \cite{dexheimer22} to obtain
\begin{align*} \beta\Big(\big(\bm{\xi}_{\min(\mathcal{T})},
\bm{\xi}_{\min(\mathcal{T})+1},
\ldots, \bm{\xi}_{\max(\mathcal{T})}\big),\, \bm{\xi}_T\Big)
 &\leq c_{\mathbf{M}}^{T-\max(\mathcal{T})}.
\end{align*}
Hence, there exists a random variable $\tilde{\bm{\xi}}$ which is independent of $(\bm{\xi}_{\min(\mathcal{T})}, \bm{\xi}_{\min(\mathcal{T})+1},
\ldots,\bm{\xi}_{\max(\mathcal{T})}),$ such that
\begin{align}
  &\Bigg\vert \mathbb{E}\left[ \Big\| P^{\perp}_{\bm{\theta}(\bm{\xi}_T)} 
  \ol{\sum}_{t \in \mathcal{T}} \Big( K_\lambda \big(\bbs(\bm{\xi}_t), \bbs(\bm{\xi}_{T})\big) \, ; \, \bbf(\bm{\xi}_{t+1}) \Big)  
  \Big\|^2_2 \right]  
  -\mathbb{E}\left[ \Big\| P^{\perp}_{\bm{\theta}(\tilde{\bm{\xi}})}
  \ol{\sum}_{t \in \mathcal{T}} \Big( K_\lambda \big(\bbs(\bm{\xi}_t), \bbs(\tilde{\bm{\xi}})\big) \, ; \, \bbf(\bm{\xi}_{t+1}) \Big)  
  \Big\|^2_2 \right]\Bigg\vert \notag 
  \\ 
&\leq  2\Vert \bbf \Vert^2_\infty\mathbb{P}(\bxi_T\neq \tilde{\bxi}) \notag 
\\
 &\leq 2\Vert \bbf \Vert_\infty^2c_{\mathbf{M}}^{T-\max(\mathcal{T})}, \label{tmp_41}   
\end{align} where we again used that the weighted sum is bounded by $\Vert \bbf\Vert_\infty$.
By construction, $\tilde{\bxi}$ is identically distributed as $\bxi_T\sim \rho$ under $\mathbb{P},$ which implies
\begin{align}
   \mathbb{E}\left[ \Big\| P^{\perp}_{\bm{\theta}(\tilde{\bm{\xi}})}
   \ol{\sum}_{t \in \mathcal{T}} \Big( K_\lambda \big(\bbs(\bm{\xi}_t), \bbs(\tilde{\bm{\xi}})\big) \, ; \, \bbf(\bm{\xi}_{t+1}) \Big)  
   \Big\|^2_2 \right]
  =\sum_{\bx\in \mathcal{D}}\mathbb{E}\left[ \Big\| P^{\perp}_{\bm{\theta}(\bx)}
  \ol{\sum}_{t \in \mathcal{T}} \Big( K_\lambda \big(\bbs(\bm{\xi}_t), \bbs(\bx)\big) \, ; \, \bbf(\bm{\xi}_{t+1}) \Big)   
  \Big\|^2_2 \right]\rho(\bx). \label{tmp_42} 
\end{align}
To simplify notation we introduce
\[ U_t(\bx)\coloneq K_\lambda (\bbs(\bm{\xi}_t), \bbs(\bx)).\]
For any $\bx\in\mathcal{D}$, it now holds
\begin{align}
  &\mathbb{E}\left[ \Big\| P^{\perp}_{\bm{\theta}(\bx)}
  \ol{\sum}_{t \in \mathcal{T}} \Big( U_t(\bx) \, ; \, \bbf(\bm{\xi}_{t+1}) \Big)   
  \Big\|^2_2 \right]
  \notag \\
  &\leq \mathbb{E}\left[ \bigg\| P^{\perp}_{\bm{\theta}(\bx)}\frac{\sum_{t \in \mathcal{T}} U_t(\bx)\bbf(\bm{\xi}_{t+1})}{\sum_{t\in\mathcal{T}} U_t(\bx)} \bigg\|^2_2\mathbb{I}\Big(\sum_{t\in\mathcal{T}} U_t(\bx)\geq \frac{\vert\mathcal{T}\vert}{2}\E\big[U_0(\bx)\big] \Big) \right]
  +\Vert \bbf\Vert_\infty^2\P\Big(\sum_{t\in\mathcal{T}}U_t(\bx)<  \frac{\vert\mathcal{T}\vert}{2}\E\big[U_0(\bx)\big]\Big)
  \notag  \\
  &\leq\frac{4}{\vert \mathcal{T}\vert^2\E\big[U_0(\bx)\big]^2} \mathbb{E}\left[ \Big\| P^{\perp}_{\bm{\theta}(\bx)}\sum_{t \in \mathcal{T}} U_t(\bx) \bbf(\bm{\xi}_{t+1}) \Big\|^2_2 \right]
  +\Vert \bbf\Vert_\infty^2\P\Big(\Big\vert\sum_{t\in\mathcal{T}} \big(U_t(\bx)-\E\big[U_0(\bx)\big]\big)\Big\vert> \frac{\vert \mathcal{T}\vert}2 \E\big[U_0(\bx)\big]\Big). \label{tmp_43} 
\end{align}
We bound the probability using Markov's inequality together with Theorem 3.2 in \cite{paulin15}. This result provides a similar variance bound for stationary Markov chains as for sums of independent random variables. For a random variable $U$ with $|U|\leq 1$ a.s., one has $\Var(U)=\mathbb{E}[U^2]-\mathbb{E}^2[U]\leq \mathbb{E}[U^2]\leq \mathbb{E}[U].$ Using that $\Vert K_{\lambda}\Vert_\infty$ maps to $[0,1]$ and applying the previous inequality to $U_0(\bx)=K_\lambda \big(\bbs(\bm{\xi}_0), \bbs(\bx)\big),$ then gives
\begin{align}
   \P\Big(\Big\vert\sum_{t\in\mathcal{T}} U_t(\bx)-\E\big[U_0(\bx)\big]\Big\vert> \frac{\vert \mathcal{T}\vert}2 \E\big[U_0(\bx)\big]\Big)
   &\leq \frac{16}{\gamma_{\mathrm{ps}}(\mathbf{M})\vert\mathcal{T}\vert\E\big[U_0(\bx)\big]^2}\E\Big[ \big\vert U_0(\bx)-\E\big[U_0(\bx)\big]\big\vert^2\Big]
   \notag \\
   &\leq \frac{16}{\gamma_{\mathrm{ps}}(\mathbf{M})\vert\mathcal{T}\vert\E\big[U_0(\bx)\big]}, 
   \label{tmp_44}
\end{align}
where $\gamma_{\mathrm{ps}}(\mathbf{M})$ denotes the pseudo spectral gap of $\mathbf{M},$ which depends only on $c_{\mathbf{M}}$ (see Definition 1.3 and Proposition 3.4 in \cite{paulin15}). Now, applying the bias-variance decomposition yields 
\begin{align*}
\mathbb{E}\left[ \Big\| P^{\perp}_{\bm{\theta}(\bx)}\sum_{t \in \mathcal{T}} U_t(\bx) \bbf(\bm{\xi}_{t+1}) \Big\|^2_2 \right]
&=\underbrace{\mathbb{E}\left[ \Big\| P^{\perp}_{\bm{\theta}(\bx)}\sum_{t \in \mathcal{T}} U_t(\bx) \bbf(\bm{\xi}_{t+1})
- \vert\mathcal{T}\vert \E\Big[ P^{\perp}_{\bm{\theta}(\bx)} U_0(\bx) \bbf(\bm{\xi}_{1})\Big] \Big\|^2_2 \right]}_{\text{variance}}
\\&\quad+ \underbrace{\vert\mathcal{T}\vert^2
\Big\|  \E\Big[ P^{\perp}_{\bm{\theta}(\bx)} U_0(\bx) \bbf(\bm{\xi}_{1})\Big] \Big\|^2_2}_{\text{squared bias}},
\end{align*}
where we again used $\Vert K_{\lambda}\Vert_\infty\leq 1$ in the last step.
For the variance term we again apply the variance bound for Markov chains provided in Theorem 3.2 in \cite{paulin15} componentwise and obtain
\begin{align}
    \mathbb{E}\left[ \Big\| P^{\perp}_{\bm{\theta}(\bx)}\sum_{t \in \mathcal{T}} U_t(\bx) \bbf(\bm{\xi}_{t+1})
- \vert\mathcal{T}\vert \E\Big[ P^{\perp}_{\bm{\theta}(\bx)} U_0(\bx) \bbf(\bm{\xi}_{1})\Big] \Big\|^2_2 \right]
    & \leq
\mathbb{E}\left[ \Big\| \sum_{t \in \mathcal{T}} U_t(\bx) \bbf(\bm{\xi}_{t+1})
- \vert\mathcal{T}\vert \E\Big[  U_0(\bx) \bbf(\bm{\xi}_{1})\Big] \Big\|^2_2 \right] \notag\\
    &\leq \frac{4\vert\mathcal{T}\vert}{\gamma_{\mathrm{ps}}(\mathbf{M})}\mathbb{E}\left[ \Big\|  U_0(\bx) \bbf(\bm{\xi}_{1})-\E\Big[  U_0(\bx) \bbf(\bm{\xi}_{1})\Big] \Big\|^2_2 \right]
      \notag \\
    &\leq \frac{4\Vert \bbf \Vert_\infty^2\vert\mathcal{T}\vert}{\gamma_{\mathrm{ps}}(\mathbf{M})}\mathbb{E}\left[  U_0(\bx)  \right]. \label{tmp_45}
\end{align}
For the bias term we deduce by conditioning on $\bxi_0$ and by applying the definitions of $\bbm$ and $\btheta,$
\begin{align}
    \vert\mathcal{T}\vert^2
\Big\|  \E\Big[ P^{\perp}_{\bm{\theta}(\bx)} U_0(\bx) \bbf(\bm{\xi}_{1})\Big] \Big\|^2_2
    &=\vert \mathcal{T}\vert^2\Big\|\E\Big[  U_0(\bx) P^{\perp}_{\bm{\theta}(\bx)}\bbm(\bm{\xi}_{0})\Big] \Big\|^2_2
    \notag \\
    &=\vert \mathcal{T}\vert^2\Big\|\E\Big[\Vert\bbm(\bm{\xi}_{0})\Vert_2  U_0(\bx) P^{\perp}_{\bm{\theta}(\bx)}\btheta(\bm{\xi}_{0})\Big] \Big\|^2_2
      \notag  \\
    &=\vert \mathcal{T}\vert^2\Big\|\E\Big[\Vert\bbm(\bm{\xi}_{0})\Vert_2  U_0(\bx) P^{\perp}_{\bm{\theta}(\bx)}(\btheta(\bm{\xi}_{0})-\btheta(\bx))\Big] \Big\|^2_2
    \notag \\
    &\leq \vert \mathcal{T}\vert^2\Vert \bbf \Vert_\infty^2\E\Big[  U_0(\bx) \Vert\btheta(\bm{\xi}_{0})-\btheta(\bx)\Vert_2\Big] ^2, \label{tmp_46}
\end{align}
where we use $P^{\perp}_{\bm{\theta}(\bx)} \bm{\theta}(\bx) = \bm{0}_{d}$ for the last equality.
Furthermore, by the Hölder continuity of $\btheta,$ we deduce for any $\epsilon>0,$
\begin{align*}
   \E\Big[  U_0(\bx) \Vert\btheta(\bm{\xi}_{0})-\btheta(\bx)\Vert_2\Big]
   &= \sum_{\bx'\in\mathcal{D}}K_\lambda \big( \bbs(\bx'), \bbs(\bx) \big) \Vert\btheta(\bx')-\btheta(\bx)\Vert_2\rho(\bx')
   \\
   &\leq c_{\btheta} \sum_{\bx'\in\mathcal{D}}K_\lambda \big( \bbs(\bx'), \bbs(\bx) \big) \Vert\bbs(\bx') - \bbs(\bx)\Vert^\beta_2\rho(\bx')
   \\
   &\leq  c_{\btheta}\bigg(\epsilon^\beta \sum_{\bx'\in\mathcal{D}:\Vert \bbs(\bx') - \bbs(\bx) \Vert_2\leq \epsilon}
   \exp\Big(- \frac{\Vert \bbs(\bx') - \bbs(\bx) \Vert_2^2}{2\lambda^2}\Big) \rho(\bx')
   \\&\qquad \quad + \exp\Big(-\frac{\epsilon^2}{2\lambda^2}\Big)\sum_{\bx'\in\mathcal{D}:\Vert \bbs(\bx') - \bbs(\bx) \Vert_2>\epsilon} \Vert\bbs(\bx') - \bbs(\bx)\Vert^\beta_2\rho(\bx')\bigg)
     \\
   &\leq  c_{\btheta}\bigg(\epsilon^\beta \E\big[  U_0(\bx) \big]
+ 2^{\beta} \Vert \bbs \Vert_\infty^{\beta}
\exp\Big(-\frac{\epsilon^2}{2\lambda^2}\Big)\bigg),
\end{align*}
where we used that $K_\lambda$ is nonnegative in the last bound.
Choosing
\[ \epsilon=\lambda\sqrt{2\log\Big((\E[  U_0(\bx) ]\lambda^\beta)^{-1}\Big)},\]
   we get 
   \begin{align}
   &\E\Big[  U_0(\bx) \Vert\btheta(\bm{\xi}_{0})-\btheta(\bx)\Vert_2\Big] \leq  c_{\btheta}\lambda^\beta \E[ U_0(\bx)]\bigg(\Big(2\log\big((\E[  U_0(\bx) ]\lambda^\beta)^{-1}\big)\Big)^{\beta/2}
+ 2^{\beta} \Vert \bbs \Vert_\infty^{\beta} \bigg). \label{tmp_51}
\end{align}
Combining this inequality with \eqref{tmp_40}, \eqref{tmp_41}, \eqref{tmp_42}, \eqref{tmp_43}, \eqref{tmp_44},
\eqref{tmp_45} and \eqref{tmp_46} yields, for any $\bx\in\mathcal{D}$
\begin{align*}
    \mathbb{E}^{\bx}\left[ \| P^{\perp}_{\bm{\theta}(\bm{\xi}_T)} \wh{\bbm}_{\lambda}(\bm{\xi}_{T})  \|^2_2 \right]
    &\leq 
    3\Vert \bbf \Vert_\infty^2c_{\mathbf{M}}^{\log^2 T-1}+\frac{32\Vert \bbf \Vert_\infty^2}{\gamma_{\mathrm{ps}}(\mathbf{M})\vert\mathcal{T}\vert\inf_{\bx\in\mathcal{D}}\E\big[U_0(\bx)\big]}
    \\
    &\quad+4\Vert \bbf \Vert_\infty^2c^2_{\btheta}\lambda^{2\beta} \Big(\Big(2\log\Big((\inf_{\bx\in\mathcal{D}}\E[  U_0(\bx) ]\lambda^\beta)^{-1}\Big)\Big)^{\beta/2}
+ 2^{\beta} \Vert \bbs \Vert_\infty^{\beta} \Big)^2,
\end{align*}
which together with the assumption $\lambda\geq T^{-1}$ and $c_{\mathbf{M}}^{\log^2 T-1}=O(T^{-2\beta})=O(\lambda^{2\beta})$ proves the first assertion  \eqref{lemma_key_first_argument}. 

\medskip \medskip

To prove the second assertion, define
\[
\nu(\lambda, \rho, \bbm) \coloneqq
\inf_{\bx \in \mathcal{D}}  \frac{\big\Vert\E\left[ U_0(\bx) P_{\bm{\theta}(\bx)} \bbm(\bm{\xi}_{0})\right]\big\Vert_2}{\E\big[ U_0(\bx) \big]}.
\]
Consider the event 
\[
    A(\bu):=\bigg\{\big\| P_{\bm{\theta}(\bu)} \wh{\bbm}_{\lambda}(\bu)  \big\|_2 < \frac{\nu(\lambda, \rho, \bbm)}{4}\bigg\}.
\]
Using the closeness to the invariant distribution $\mathbb{P}$ and arguing similarly as in \eqref{tmp_40}, \eqref{tmp_41}, and \eqref{tmp_42}, we  obtain for $\tilde{\bxi}$ as in \eqref{tmp_41} and  any $\bx \in \mathcal{D}$,
\begin{align}
    \mathbb{P}^{\bx}\big(A(\bm{\xi}_{T})\big)
    &\leq \Vert P^{\min(\mathcal{T})}(x,\cdot)-\rho\Vert_{\operatorname{TV}}+\mathbb{P}\big(A(\bm{\xi}_{T})\big)\notag
    \\
    &\leq c_{\mathbf{M}}^{\min(\mathcal{T})}+\mathbb{P}(\bxi_T\neq \tilde{\bxi})+\mathbb{P}\big(A(\tilde{\bxi})\big)\notag
    \\
    &\leq c_{\mathbf{M}}^{\min(\mathcal{T})} + 2 c_{\mathbf{M}}^{T-\max(\mathcal{T})} + \sum_{\bx\in\mathcal{D}} \mathbb{P}\big(A(\bx)\big) \rho(\bx). \label{tmp_54}
\end{align}
Moreover, 
\begin{align}
    c_{\mathbf{M}}^{\min(\mathcal{T})} + 2 c_{\mathbf{M}}^{T-\max(\mathcal{T})}
    \leq 3c_{\mathbf{M}}^{\log^2 T}
    \lesssim \frac 1 T.
    \label{eq.temp_54b}
\end{align}
It therefore suffices to investigate the stationary case.
Additionally it holds for any $\bx \in \mathcal{D}$,
\begin{align}
    \mathbb{P}\big(A(\bx)\big) 
    &=\mathbb{P} \left( \frac{\Vert\sum_{t \in \mathcal{T}} U_t(\bx) P_{\bm{\theta}(\bx)} \bbf(\bm{\xi}_{t+1})\Vert_2}{\sum_{t \in \mathcal{T}} U_t(\bx)}< \frac{\nu(\lambda, \rho, \bbm)}{4}  \right) 
    \notag \\
    &\leq 
    \mathbb{P}\left( \frac{1}{\vert\mathcal{T}\vert}\sum_{t \in \mathcal{T}} U_t(\bx) \geq 2 \E\big[ U_0(\bx) \big]     \right)
    + \mathbb{P}\left( \frac{1}{\vert\mathcal{T}\vert}\Big\Vert\sum_{t \in \mathcal{T}} U_t(\bx) P_{\bm{\theta}(\bx)} \bbf(\bm{\xi}_{t+1})\Big\Vert_2< \frac{\nu(\lambda, \rho, \bbm)}{2}  \E\big[ U_0(\bx) \big]     \right)
    \notag \\
    &\leq 
    \mathbb{P}\left( \frac{1}{\vert\mathcal{T}\vert}\sum_{t \in \mathcal{T}} U_t(\bx) \geq 2 \E\big[ U_0(\bx) \big]     \right)
    \notag \\
    & \quad + \mathbb{P}\left( \frac{1}{\vert\mathcal{T}\vert}\Big\Vert\sum_{t \in \mathcal{T}} U_t(\bx) P_{\bm{\theta}(\bx)} \bbf(\bm{\xi}_{t+1})\Big\Vert_2< \frac{1}{2}\Big\Vert\E\left[ U_0(\bx) P_{\bm{\theta}(\bx)} \bbm(\bm{\xi}_{0})\right]\Big\Vert_2       \right),\label{tmp_55}
\end{align}
where we used the definition of $\nu(\lambda, \rho, \bbm)$ in the last step.
The first term can be bounded similarly to \eqref{tmp_44} by
\begin{align}
    \mathbb{P}\left( \frac{1}{\vert\mathcal{T}\vert}\sum_{t \in \mathcal{T}} U_t(\bx) \geq 2 \E\big[ U_0(\bx) \big]     \right)
    \leq \frac{4}{\gamma_{\mathrm{ps}}(\mathbf{M})\vert\mathcal{T}\vert\E\big[U_0(\bx)\big]}. \label{tmp_56}
\end{align}
Furthermore, since $\|\bu\|_2 < \tfrac 12 \|\bv\|_2$ implies that $\|\bu-\bv\|_2 > \tfrac 12 \|\bv\|_2$, we obtain for the second term
\begin{align*}
    &\mathbb{P}\bigg( \frac{1}{\vert\mathcal{T}\vert}\Big\Vert\sum_{t \in \mathcal{T}} U_t(\bx) P_{\bm{\theta}(\bx)} \bbf(\bm{\xi}_{t+1})\Big\Vert_2
    < 
    \frac{1}{2}\Big\Vert\E\left[ U_0(\bx) P_{\bm{\theta}(\bx)} \bbm(\bm{\xi}_{0})\right]\Big\Vert_2 \bigg)\\
    & \leq \mathbb{P}\bigg( \bigg\Vert \sum_{t \in \mathcal{T}} U_t(\bx) P_{\bm{\theta}(\bx)} \bbf(\bm{\xi}_{t+1}) - \vert\mathcal{T}\vert \E\left[ U_0(\bx) P_{\bm{\theta}(\bx)} \bbm(\bm{\xi}_{0})\right] \bigg\Vert_2  >  
    \frac{\vert\mathcal{T}\vert}{2}\Big\Vert\E\left[ U_0(\bx) P_{\bm{\theta}(\bx)} \bbm(\bm{\xi}_{0})\right]\Big\Vert_2 \bigg)\\
    & \leq \frac{4}{\big\Vert\E\left[ U_0(\bx) P_{\bm{\theta}(\bx)} \bbm(\bm{\xi}_{0})\right]\big\Vert_2^2 |\mathcal{T}|^2} \E \left[ \Big\Vert\sum_{t \in \mathcal{T}} U_t(\bx) \bbf(\bm{\xi}_{t+1}) - \vert\mathcal{T}\vert \E\left[ U_0(\bx) \bbf(\bm{\xi}_{1})\right] \bigg\Vert_2^2 \right]\\
    & \leq 
    \frac{4}{\big\Vert\E\left[ U_0(\bx) P_{\bm{\theta}(\bx)} \bbm(\bm{\xi}_{0})\right]\big\Vert_2^2 |\mathcal{T}|^2}
    \,
    \frac{4\Vert \bbf \Vert_\infty^2\vert\mathcal{T}\vert}{\gamma_{\mathrm{ps}}(\mathbf{M})}\mathbb{E}\left[  U_0(\bx)  \right]\\
    & = \frac{16 \Vert \bbf \Vert_\infty^2}{\gamma_{\mathrm{ps}}(\mathbf{M}) \vert\mathcal{T}\vert} 
    \,
    \frac{\E\big[ U_0(\bx) \big]}{\big\Vert\E\left[ U_0(\bx) P_{\bm{\theta}(\bx)} \bbm(\bm{\xi}_0)\right]\big\Vert_2^2} 
\end{align*}
where we use \eqref{tmp_45} for the last inequality.
Combining this inequality with \eqref{tmp_54}, \eqref{eq.temp_54b}, \eqref{tmp_55}, and \eqref{tmp_56} yields the second assertion \eqref{lemma_key_second_argument}.
\end{proof}

\subsection{Proof of Lemma~\ref{lemma_application_keylemma_properties}}
\label{proof_lemma_application_keylemma_properties}

\begin{proof}
\textbf{(i)}
Recall that for any $\bx = (\bx_1, \ldots, \bx_r) \in \mathcal{E}_W^{r}$, $\bz_{1} \coloneqq \stack(\bm{\phi}_1(\bx_{r-r_1 + 1}), \ldots, \bm{\phi}_1(\bx_r)) \in \mathbb{R}^{m_1 r_1}$ and
$\bz_{2} \coloneqq \stack(\bm{\phi}_2(\bx_{r-r_2 + 1}), \ldots, \bm{\phi}_2(\bx_r)) \in \mathbb{R}^{m_2 r_2}$,
\begin{align*}    \mathbb{P}\big(\bX_{r+1} = \be_w \big|\bh, (\bX_1, \ldots, \bX_r) = (\bx_1, \ldots, \bx_r) \big) 
\propto
\exp\Big(\bg_0(\bz_{1})^{\top} \bm{\phi}_1(\be_w) + \bh(\bz_{2})^{\top} \bm{\phi}_2(\be_w)\Big), \quad w \in [W].
\end{align*}
Hence,
    \begin{align}
    \bbm(\bx)
    & = \mathbb{E}\left[ \frac{\bm{\phi}_2(\bX_{r+1})}{\exp\big(\bg_0(\bz_{1})^{\top} \bm{\phi}_1(\bX_{r+1})\big) } \middle|\bh, (\bX_1, \ldots, \bX_r) = (\bx_1, \ldots, \bx_r) \right] \quad \nonumber \\
    & = \frac{\sum_{w=1}^W \exp\Big(\bg_0(\bz_{1})^{\top} \bm{\phi}_1(\be_w) + \bh(\bz_2)^{\top} \bm{\phi}_2(\be_w)\Big) \frac{\bm{\phi}_2(\be_{w})}{\exp\big(\bg_0(\bz_{1})^{\top} \bm{\phi}_1(\be_w)\big)}}{\sum_{w=1}^W \exp\Big(\bg_0(\bz_{1})^{\top} \bm{\phi}_1(\be_w) + \bh(\bz_2)^{\top} \bm{\phi}_2(\be_w)\Big)}  \nonumber \\
    & = \frac{\sum_{w=1}^W \exp\Big( \bh(\bz_2)^{\top} \bm{\phi}_2(\be_w)\Big)
    \bm{\phi}_2(\be_{w})}{\sum_{w=1}^W \exp\Big(\bg_0(\bz_{1})^{\top} \bm{\phi}_1(\be_w) + \bh(\bz_2)^{\top} \bm{\phi}_2(\be_w)\Big)}  \nonumber \\
    & = \underbrace{\frac{\sum_{w=1}^W \exp\big( \bh(\bz_2)^{\top} \bm{\phi}_2(\be_w) \big) \bm{\phi}_2(\be_{w})}{\sum_{w=1}^W \exp\big( \bh(\bz_2)^{\top} \bm{\phi}_2(\be_w) \big)}}_{\displaystyle \eqcolon \, \tilde{\bm{\theta}}(\bz_2) \, \in \, \mathbb{R}^{m_2}} \, \underbrace{\frac{\sum_{w=1}^W \exp\big(\bh(\bz_2)^{\top} \bm{\phi}_2(\be_w)\big)}{\sum_{w=1}^W \exp\big(\bg_0(\bz_{1})^{\top} \bm{\phi}_1(\be_w) + \bh(\bz_2)^{\top} \bm{\phi}_2(\be_w)\big)}}_{\displaystyle \in \, \mathbb{R}}, \label{tmp_47}
    \end{align}  
which gives us the assertion.

\medskip

\textbf{(ii)}
Applying the reverse triangle inequality, the second argument of Lemma \ref{lemma_VMF_property}, the definition of $\kappa_W$ (Definition \ref{def_pos_assumption}), and the assumption
$\kappa_W \leq (1 + m_2 / \delta_2)^{-1}/2$,
\begin{align}
    \big\| \tilde{\bm{\theta}}(\bz_2) \big\|_2 &\geq 
\left\| \frac{\int_{\bu \in \mathbb{S}^{m_2 - 1}} \exp( \bh(\bz_2)^{\top} \bu) \bu d\bu }{\int_{\bu \in \mathbb{S}^{m_2 - 1}} \exp(\bh(\bz_2)^{\top} \bu) d\bu  }\right\|_{2}
- \left\| \tilde{\bm{\theta}}(\bz_2)  - \frac{\int_{\bu \in \mathbb{S}^{m_2 - 1}} \exp(\bh(\bz_2)^{\top} \bu) \bu d\bu }{\int_{\bu \in \mathbb{S}^{m_2 - 1}} \exp(\bh(\bz_2)^{\top} \bu) d\bu  }\right\|_{2} \notag \\ 
&\geq \frac{1}{1 + m_2 / \delta_2} - \kappa_W
\geq 
\frac{1}{2 + 2 m_2 / \delta_2}
\coloneqq \varrho > 0. \label{tmp_37}
\end{align}
On the other hand, we have 
\[
\frac{\sum_{w=1}^W \exp\big(\bh(\bz_2)^{\top} \bm{\phi}_2(\be_w)\big)}{\sum_{w=1}^W \exp\big(\bg_0(\bz_{1})^{\top} \bm{\phi}_1(\be_w) + \bh(\bz_2)^{\top} \bm{\phi}_2(\be_w)\big)} 
\geq
\frac{W \exp(-\delta_2)}{W \exp(\delta_1 m_1 + \delta_2)} = \exp(- \delta_1 m_1 - 2 \delta_2).
\]
Combining these inequalities with $\eqref{tmp_47}$ gives us the assertion.

\medskip

\textbf{(iii)} 
By Lemma~\ref{lemma_softmax_lipschitz} and the smoothness of $\bh$,
\begin{align*}
    \big\|\tilde{\bm{\theta}}\big( \bbs(\bx) \big) - \tilde{\bm{\theta}}\big( \bbs(\bx')\big)\big\|_{\infty} 
    \leq  2 \big\|\bh\big( \bbs(\bx) \big) - \bh\big( \bbs(\bx')\big)\big\|_{2} 
&\leq 2 \sqrt{m_2} \big\|\bh\big( \bbs(\bx) \big) - \bh\big( \bbs(\bx')\big)\big\|_{\infty} \\
&\leq 2 \sqrt{m_2} \big\|\bh\|_{\mathcal{C}^{\beta_2}} \|\bbs(\bx) - \bbs(\bx')\big\|_{\infty}^{\beta_2}.
\end{align*}
Using $\bm{\theta}(\bx) = \bbm(\bx) / \|\bbm(\bx)\|_2 = \tilde{\bm{\theta}}(\bbs(\bx)) / \|\tilde{\bm{\theta}}(\bbs(\bx))\|_2$ and \eqref{tmp_37}, we get the assertion by
\begin{align*}
    \| \bm{\theta}(\bx) - \bm{\theta}(\bx')  \|_2 
&= \left\|\frac{\tilde{\bm{\theta}}\big( \bbs(\bx) \big)}{\|\tilde{\bm{\theta}}\big( \bbs(\bx) \big)\|_2}
-
\frac{\tilde{\bm{\theta}}\big( \bbs(\bx')\big)}{\|\tilde{\bm{\theta}}\big( \bbs(\bx')\big)\|_2}\right\|_2\\
&\leq \left\|\frac{\tilde{\bm{\theta}}\big( \bbs(\bx) \big) - \tilde{\bm{\theta}}\big( \bbs(\bx')\big)}{\|\tilde{\bm{\theta}}\big( \bbs(\bx) \big)\|_2}\right\|_2
+
\left\| \left(\frac{1}{\|\tilde{\bm{\theta}}\big( \bbs(\bx)\big)\|_2}
-
\frac{1}{\|\tilde{\bm{\theta}}\big( \bbs(\bx')\big)\|_2}\right) \tilde{\bm{\theta}}\big( \bbs(\bx')\big)\right\|_2\\
& \leq \frac{4 m_2 \|\bh\|_{\mathcal{C}^{\beta_2}}}{\varrho} \|\bbs(\bx) - \bbs(\bx')\|_{\infty}^{\beta_2}. 
\end{align*}

\medskip

\textbf{(iv)} 
By \eqref{rho_maximum_minimum},
$\max_{\bx \in \mathcal{E}_W^{r}} \rho(\bx) \big/ \min_{\bx \in \mathcal{E}_W^{r}} \rho(\bx)$ is bounded above by a constant that depends only on $(m_1, \delta_1, \delta_2)$. 
For any $\bx \in \mathcal{E}_W^{r}$,
it follows using the definition $K_\lambda(\bz,\bz')=\exp(-\Vert \bz - \bz' \Vert^2/(2\lambda^2)),$
\begin{align}    \mathbb{E}_{\rho}\big[K_{\lambda}\big( \bbs(\bm{\xi}), \bbs(\bx)\big)\big]
&\geq     \mathbb{E}_{\rho}\big[K_{\lambda}\big( \bbs(\bm{\xi}), \bbs(\bx)\big) \mathbb{I}\big(\|\bbs(\bm{\xi}) - \bbs(\bx)\|_2 \leq \sqrt{r_2} \lambda\big) \big] \notag \\
&\geq
\exp(- r_2/2) \rho\big( \bm{\xi}: \|\bbs(\bm{\xi}) - \bbs(\bx)\|_2 \leq \sqrt{r_2} \lambda\big) \notag \\
&\gtrsim 
\frac{1}{W^{r}}  \Big|\Big\{(\bx'_1, \ldots, \bx'_r) \in \mathcal{E}_W^{r} : \max_{1 \leq i \leq r_2} \| \bm{\phi}_2(\bx'_i) - \bm{\phi}_2(\bx_i)  \|_2 \leq \lambda \Big\}\Big| \notag \\
&= 
\frac{1}{W^{r_2}}  \Big|\Big\{(\bx'_1, \ldots, \bx'_{r_2}) \in \mathcal{E}_W^{r_2} : \max_{1 \leq i \leq r_2} \| \bm{\phi}_2(\bx'_i) - \bm{\phi}_2(\bx_i)  \|_2 \leq \lambda \Big\}\Big| \notag \\
& \geq \left(\inf_{\bx \in \mathcal{E}_W} \frac{1}{W}  \Big|\Big\{\bx' \in \mathcal{E}_W : \| \bm{\phi}_2(\bx') - \bm{\phi}_2(\bx)  \|_2 \leq \lambda \Big\}\Big|\right)^{r_2} \notag \\
& = \varphi_W(\lambda)^{r_2}, \notag
\end{align}
where the implicit constant of $\gtrsim$ depends only on $(m_1, \delta_1, \delta_2, r_2)$.

\medskip

\textbf{(v)}
By definition, $P_{\bm{\theta}(\bx)} \bv = \langle \bm{\theta}(\bx), \bv \rangle \bm{\theta}(\bx)$ for any $\bv \in \mathbb{R}^m$. Hence for any $\bx \in \mathcal{E}_W^r$, 
\begin{align}
    \Big\|\E_{\rho}\big[K_\lambda \big(\bbs(\bm{\xi}), \bbs(\bx)\big) P_{\bm{\theta}(\bx)} \bbm(\bm{\xi})\big]\Big\|_2
    & = 
    \Big|\sum_{\bx' \in\mathcal{D}}K_\lambda (\bbs\big(\bx'), \bbs(\bx)\big) \rho(\bx') \langle \bm{\theta}(\bx), \bbm(\bx') \rangle \Big| \notag \\
    & \geq 
    \sum_{\bx' \in\mathcal{D}}K_\lambda (\bbs\big(\bx'), \bbs(\bx)\big) \rho(\bx') \langle \bm{\theta}(\bx), \bbm(\bx') \rangle \notag \\
    & =
    \sum_{\bx' \in\mathcal{D}}K_\lambda (\bbs\big(\bx'), \bbs(\bx)\big) \rho(\bx') \big|\langle \bm{\theta}(\bx), \bbm(\bx') \rangle\big| \notag \\
    & \quad -2 \sum_{\bx' \in\mathcal{D}: \langle \bm{\theta}(\bx), \bm{\theta}(\bx') \rangle < 0}K_\lambda (\bbs\big(\bx'), \bbs(\bx)\big) \rho(\bx') \big| \langle \bm{\theta}(\bx), \bbm(\bx') \rangle\big|.
    \label{tmp_53}
\end{align}
Since $\bm{\theta}(\bx) =\bbm(\bx)/\Vert\bbm(\bx)\Vert_2$ for any $\bx\in\mathcal{D}$ it follows for any $\bx' \in\mathcal{D}$,  that $|\langle \bm{\theta}(\bx), \bbm(\bx') \rangle| \geq c_1 \langle \bm{\theta}(\bx), \bm{\theta}(\bx') \rangle = 
\tfrac{c_1}{2} ( 2 - \|\bm{\theta}(\bx)- \bm{\theta}(\bx')\|_2^2)$. 
Hence, 
\begin{align}
    \sum_{\bx' \in\mathcal{D}}K_\lambda (\bbs\big(\bx'), \bbs(\bx)\big) \rho(\bx') \big|\langle \bm{\theta}(\bx), \bbm(\bx') \rangle\big|
&\geq
\frac{c_1}{2} \sum_{\bx' \in\mathcal{D}}K_\lambda (\bbs\big(\bx'), \bbs(\bx)\big) \rho(\bx') \big( 2 - \|\bm{\theta}(\bx)- \bm{\theta}(\bx')\|_2^2\big) \notag \\
& = c_1 \E_{\rho} \big[K_\lambda \big(\bbs(\bm{\xi}), \bbs(\bx)\big)\big]
- 
\frac{c_1}{2} \E_{\rho} \big[K_\lambda \big(\bbs(\bm{\xi}), \bbs(\bx)\big) \Vert\btheta(\bm{\xi})-\btheta(\bx)\Vert_2\big] \notag \\
&\geq c_1 \E_{\rho} \big[K_\lambda \big(\bbs(\bm{\xi}), \bbs(\bx)\big)\big], \label{tmp_52}
\end{align}
where the second inequality follows from $\|\bm{\theta}(\bx)- \bm{\theta}(\bx')\|_2^2 \leq 2 \|\bm{\theta}(\bx)- \bm{\theta}(\bx')\|_2$, and the last inequality follows from \eqref{tmp_51} by  taking $\lambda$ sufficiently small.
On the other hand, since $\langle \bm{\theta}(\bx), \bm{\theta}(\bx') \rangle < 0$ implies $\|\bm{\theta}(\bx) - \bm{\theta}(\bx') \|_2 \geq \sqrt{2}$ and hence $\|\bbs(\bx) - \bbs(\bx') \|_2 \geq (\sqrt{2}/c_{\bm{\theta}})^{1/\beta_2}$, we have by the definition of $K_{\lambda}$ and the Cauchy--Schwarz inequality
\begin{align*}
    \sum_{\bx' \in\mathcal{D}: \langle \bm{\theta}(\bx), \bm{\theta}(\bx') \rangle < 0}K_\lambda (\bbs\big(\bx'), \bbs(\bx)\big) \rho(\bx') \big| \langle \bm{\theta}(\bx), \bbm(\bx') \rangle\big|&\leq \exp \Big(- \frac{(\sqrt{2}/c_{\bm{\theta}})^{2/\beta_2}}{2 \lambda^2}\Big) \sup_{\bx' \in \mathcal{E}_W^r} \| \bbm(\bx') \|_2 \\
    & \leq \exp \Big(- \frac{(\sqrt{2}/c_{\bm{\theta}})^{2/\beta_2}}{2 \lambda^2} + \delta_1 m_1 \Big).
\end{align*}
Combining this inequality with \eqref{tmp_53}, \eqref{tmp_52}, and (iv), we obtain
\begin{align*}
    \frac{\big\Vert\E_{\rho}\left[ K_\lambda \big(\bbs(\bm{\xi}), \bbs(\bx)\big) P_{\bm{\theta}(\bx)} \bbm(\bm{\xi})\right]\big\Vert_2}{\E_{\rho}\big[ K_\lambda \big(\bbs(\bm{\xi}), \bbs(\bx)\big) \big]}
    & \geq \frac{c_1}{2} - 
    2\frac{\exp \big(- \frac{(\sqrt{2}/c_{\bm{\theta}})^{2/\beta}}{2 \lambda^2}+\delta_1m_1\big)}{\E\big[U_0(\bx)\big]}\\
    & \geq  \frac{c_1}{2} - 2
    \frac{\exp \big(- \frac{(\sqrt{2}/c_{\bm{\theta}})^{2/\beta}}{2 \lambda^2} + \delta_1 m_1 \big)}{ \varphi_W(\lambda)^{r_2}}\\
    & \geq  \frac{c_1}{2} - 2
    W^{r_2} \exp \bigg(- \frac{(\sqrt{2}/c_{\bm{\theta}})^{2/\beta} \log^2 W}{2} + \delta_1 m_1 \bigg) \\
    & \geq  \frac{c_1}{4},
\end{align*}
provided that $W$ is sufficiently large.
We used that by definition, $\varphi_W(\lambda) \geq 1/W$ for the third inequality.
\end{proof}

\section{Technical Lemmas}

\begin{lemma}[Corollary A.7 of \cite{pmlr-v162-edelman22a}]
\label{lemma_softmax_lip}
For any real-valued vectors $\bx$ and $\bx'$ of the same dimension,  
\[
\|\operatorname{softmax}(\bx)-\operatorname{softmax}(\bx')\|_{\infty}
\leq
\|\operatorname{softmax}(\bx)-\operatorname{softmax}(\bx')\|_{1}
\leq 2\|\bx-\bx'\|_{\infty}.
\]
\end{lemma}

\begin{lemma} \label{lemma_KL_MSE}
If $u_1,\ldots,u_m, v_1,\ldots, v_m$ are real numbers, $U\coloneq\sum_{k=1}^m e^{u_k}$, and $V\coloneq\sum_{k=1}^m e^{v_k},$ then, 
\[
    \sum_{j=1}^m \frac{e^{u_j}}{U}\log\Big(e^{u_j-v_j}\frac{V}{U}\Big)
    \leq 4\max_j |u_j-v_j|^2.
\]
\end{lemma}

\begin{proof}
Let $\Delta\coloneq\max_j |u_j-v_j|.$ Using the log sum inequality to bound $\log(V/U)$,
    \[
    \sum_{j=1}^m \frac{e^{u_j}}{U}\log\Big(e^{u_j-v_j}\frac{V}{U}\Big)
    = 
    \sum_{j=1}^m \frac{e^{u_j}}{U}\big(u_j-v_j\big) +\log\Big(\frac{V}{U}\Big)
    \leq 
    \sum_{j=1}^m \frac{e^{u_j}}{U}\big(u_j-v_j\big) +\frac{e^{v_j}}{V}\big(v_j-u_j\big).
\]
The right hand side can be upper bounded by $\leq 2\Delta.$ The right hand side can also be upper bounded by
\begin{align*}
\sum_{j=1}^m \big(u_j-v_j\big)\Big(\frac{e^{u_j}}{U}-\frac{e^{v_j}}{V}\Big)
&= \sum_{j=1}^m \big(u_j-v_j\big)\Big(\frac{e^{u_j}-e^{v_j}}{U}+\frac{e^{v_j}}{V}\cdot \frac{V-U}{U}\Big)\\
&\leq \Delta \big(e^\Delta-1\big)
+\Delta \frac{|V-U|}{U} \\
&\leq 2\Delta \big(e^\Delta-1\big),
\end{align*}
using that for any two real numbers $u,v,$ $|e^u-e^v|\leq (e^{|u-v|}-1)e^u.$ Combining both upper bounds and using that for any $x\leq \log(2),$ $e^x-1=x+\sum_{k= 2}^\infty x^k/k!\leq x+\tfrac 12 x\sum_{\ell=1}^\infty x^\ell/\ell!\leq x+\tfrac 12 xe^x\leq 2x$ yields the upper bound
\[
    \leq 2\Delta\min\Big(1,e^\Delta-1\Big)\leq 2\Delta \min(1,2\Delta)\leq 4\Delta^2.
\]
\end{proof}

\begin{lemma} \label{lemma_softmax_lipschitz}
For any $\delta>0,$ any $\bv_1, \bv_2 \in \delta \mathbb{S}^{m-1},$ and any $\bu_1, \ldots, \bu_W \in \mathbb{S}^{m-1}$, 
    \[
    \left\|
    \frac{\sum_{w=1}^W \exp\big( \bv_1^{\top} \bu_w \big) \bu_w }{\sum_{w=1}^W \exp\big( \bv_1^{\top} \bu_w \big)}
    -
    \frac{\sum_{w=1}^W \exp\big( \bv_2^{\top} \bu_w \big) \bu_w }{\sum_{w=1}^W \exp\big( \bv_2^{\top} \bu_w \big)}
    \right\|_{\infty} \leq 2 \|\bv_1 - \bv_2 \|_{2}.
    \]
\end{lemma}
\begin{proof}
    Define for $\bv \in \delta \mathbb{S}^{m-1}$ 
    \[ 
    \bp(\bv) \coloneqq \Big(\frac{\exp\big( \bv^{\top} \bu_w \big)}{\sum_{i=1}^W \exp\big( \bv^{\top} \bu_i \big)}\Big)_{w \in [W]} = \operatorname{softmax}\big(\by(\bv)\big), \qquad \by(\bv) \coloneqq (\bv^{\top} \bu_1, \ldots, \bv^{\top} \bu_W)^{\top}.  
    \]
    Since for each coordinate $j \in [m]$,
    \begin{align*}
&\left|
    \frac{\sum_{w=1}^W \exp\big( \bv_1^{\top} \bu_w  \big) (\bu_w)_{(j)}}{\sum_{w=1}^W \exp\big( \bv_1^{\top} \bu_w \big)}
    -
    \frac{\sum_{w=1}^W \exp\big( \bv_2^{\top} \bu_w \big) (\bu_w)_{(j)} }{\sum_{w=1}^W \exp\big( \bv_2^{\top} \bu_w \big)}
    \right|\\
    &=
    \left| \sum_{w=1}^W \bp(\bv_1)_{(w)} (\bu_w)_{(j)}
    -
    \sum_{w=1}^W \bp(\bv_2)_{(w)} (\bu_w)_{(j)}\right|\\
    &\leq  \sum_{w=1}^W \big|\bp(\bv_1)_{(w)} - \bp(\bv_2)_{(w)} \big|
    \big|(\bu_w)_{(j)}\big|\\
    & \leq \big\| \bp(\bv_1) - \bp(\bv_2) \big\|_1
    \end{align*}    
    and $\|\operatorname{softmax}(\bx) - \operatorname{softmax}(\bx')\|_1 \leq 2 \| \bx - \bx'\|_{\infty}$ (see Lemma~\ref{lemma_softmax_lip}), we get the assertion by
    \begin{align*}
        \left\|
    \frac{\sum_{w=1}^W \exp\big( \bv_1^{\top} \bu_w \big) \bu_w }{\sum_{w=1}^W \exp\big( \bv_1^{\top} \bu_w \big)}
    -
    \frac{\sum_{w=1}^W \exp\big( \bv_2^{\top} \bu_w \big) \bu_w }{\sum_{w=1}^W \exp\big( \bv_2^{\top} \bu_w \big)}
    \right\|_{\infty}
    & \leq  \|\bp(\bv_1) - \bp(\bv_2)\|_{1}\\
    & \leq 2 \| \by(\bv_1) - \by(\bv_2)  \|_\infty\\
    &= 2 \max_{w \in [W]} \big|(\bv_1 - \bv_2)^{\top} \bu_w\big|\\
    & \leq 2 \|\bv_1 - \bv_2 \|_{2}.
    \end{align*}
\end{proof}

\begin{lemma} \label{lemma_existence_constant}
    Let $\bZ$ be uniformly distributed on $\mathbb{S}^{m-1}$. 
    Then there exists a constant $c_m>0$, depending only on $m$, such that for any $\epsilon \in [0,1]$ and any
$\bz \in \mathbb{S}^{m-1}$,
    \begin{align*}
    \mathbb{P}( \| \bZ - \bz \|_2  \leq \epsilon/2 ) \geq c_m \epsilon^{m-1}.
\end{align*}  
\end{lemma}
\begin{proof}
By the uniformity, it is sufficient to consider $\bz=\be_1 = (1,0,\ldots,0)^{\top} \in \mathbb{R}^m$. Write $\bZ=(Z_1,\ldots,Z_m)^\top\eqcolon(Z_1,\bZ_{(-1)})^\top\in \mathbb{R}\times \mathbb{R}^{m-1}.$ On the event $Z_{1} \geq 0$, triangle inequality, $\|\bZ\|_2=1$, and the formula $\sqrt{a}-\sqrt{b}=(a-b)/(\sqrt{a}+\sqrt{b})$ that holds for all positive $a,b,$ yield
\[
\| \bZ - \be_1 \|_2 \leq | Z_1 - 1  | + \| \bZ_{(-1)} \|_2 = \Big|1 - \sqrt{1 - \|\bZ_{(-1)} \|_2^2} \Big| + \|\bZ_{(-1)} \|_2 
\leq \|\bZ_{(-1)} \|_2^2 + \|\bZ_{(-1)} \|_2.
\]
Hence,
\begin{align*}
    \mathbb{P}( \| \bZ - \bz \|_2  \leq \epsilon/2 ) = \mathbb{P}( \| \bZ - \be_1 \|_2  \leq \epsilon/2 ) 
    & \geq \mathbb{P}( Z_1>0, \, \| \bZ_{(-1)} \|_2 
    \leq \epsilon/4 ) \\ 
    &= \frac 12  \mathbb{P}(\| \bZ_{(-1)} \|_2 
    \leq \epsilon/4 )
    = C_m (\epsilon/4)^{m-1}
\end{align*}
for some constant $C_m>0$ depending only on $m$,
where the last equality follows from the volume formula of an $(m-1)$-dimensional Euclidean ball with radius $\epsilon/4.$ 
\end{proof}

\begin{lemma}
\label{lemma_VMF_property}
For any $m = 2,3,\ldots$ and $\bv \in \mathbb{R}^m$, define
\[
\bbf(\bv) \coloneqq \frac{\int_{\bz \in \mathbb{S}^{m - 1}} \exp(\bv^{\top} \bz) \bz \,  d\bz }{\int_{\bz \in \mathbb{S}^{m - 1}} \exp(\bv^{\top} \bz) \, d\bz  }.
\]
Then,
   \[
   \frac{\bbf(\bv)}{\|\bbf(\bv)\|_2} = \frac{\bv}{\|\bv\|_2} \quad \text{ and } \quad
    \|\bbf(\bv)\|_2 >  \frac{1}{1 + m / \|\bv\|_2}
   \]
\end{lemma}
\begin{proof}
     We denote by $\operatorname{vMF}_{\bv}$ the von Mises--Fisher
distribution on $\mathbb{S}^{m-1}$ with mean direction $\bv/\|\bv\|_2$ and concentration parameter
$\|\bv\|_2$, whose probability density function is given by 
\[\bz \mapsto c_{\bv}  \exp(\bv^{\top} \bz),\quad \bz\in\mathbb{S}^{m-1},\]
with $c_{\bv}$ the normalizing constant. The expectation is given by
\begin{align*}
\mathbb{E}_{\bZ \sim \operatorname{vMF}(\bv)}[\bZ] 
=  \bbf(\bv)
= \frac{I_{m/2}(\|\bv\|_2)}{I_{m/2-1}(\|\bv\|_2)} \frac{\bv}{\|\bv\|_2},
\end{align*}
with $I_{v}$ the modified Bessel function of the first kind (see e.g.\ \cite{zbMATH01375577}, Equation (9.3.8)). This establishes the first assertion.

Moreover, classical Amos-type bounds for Bessel function ratios (e.g., see Theorem 3 of \cite{hornik2013amos}) yield that
\begin{align*}
    \frac{I_{m/2}(\|\bv\|_2)}{I_{m/2-1}(\|\bv\|_2)}
    \geq
    \frac{\|\bv\|_2}{\frac{m-1}{2}+\sqrt{(\frac{m+1}{2})^2 + \|\bv\|_2^2}}.  
\end{align*}
We obtain the second assertion by using that $\sqrt{a^2+b^2}<a+b$ for positive $a,b,$ and therefore
\begin{align*}
    \frac{\|\bv\|_2}{\frac{m-1}{2}+\sqrt{(\frac{m+1}{2})^2 + \|\bv\|_2^2}} 
    >     \frac{\|\bv\|_2}{\frac{m-1}{2}+\frac{m+1}{2} + \|\bv\|_2} 
    = \frac{1}{1 + m / \|\bv\|_2}. 
\end{align*}
\end{proof}

\end{document}